\documentclass[11pt, reqno]{amsart}

\usepackage{amscd, amsmath, amssymb, color, bm}

\newcommand{\de}{\partial}
\newcommand{\db}{\overline{\partial}}
\newcommand{\ddbar}{i\partial\overline\partial}
\newcommand{\ov}[1]{\overline{#1}}
\newcommand{\mn}{i}
\newcommand{\tr}[2]{\mathrm{tr}^{#1}{#2}}
\newcommand{\ti}[1]{\tilde{#1}}
\newcommand{\vp}{\varphi}
\newcommand{\ve}{\varepsilon}

\renewcommand{\leq}{\leqslant}
\renewcommand{\geq}{\geqslant}

\renewcommand{\epsilon}{\varepsilon}

\newcommand{\N}{\mathbb{N}}
\newcommand{\R}{\mathbb{R}}
\newcommand{\C}{\mathbb{C}}

\newtheorem{theorem}{Theorem}[section]
\newtheorem{lemma}[theorem]{Lemma}
\newtheorem{corollary}[theorem]{Corollary}
\newtheorem{proposition}[theorem]{Proposition}

\numberwithin{equation}{section}

\theoremstyle{definition}
\newtheorem{rk}[theorem]{Remark}

\theoremstyle{definition}
\newtheorem{definition}[theorem]{Definition}

\begin{document}

\title{Higher-order estimates for collapsing Calabi-Yau metrics}
\begin{abstract}
We prove a uniform $C^\alpha$ estimate for collapsing Calabi-Yau metrics on the total space of a proper holomorphic submersion over the unit ball in $\C^m$. The usual methods of Calabi, Evans-Krylov, Caffarelli, et al.~do not apply to this setting because the background geometry degenerates. We instead rely on blowup arguments and on linear and nonlinear Liouville theorems on cylinders. In particular, as an intermediate step, we use such arguments to prove sharp new Schauder estimates for the Laplacian on cylinders. If the fibers of the submersion are pairwise biholomorphic, our method yields a uniform $C^\infty$ estimate.
We then apply these local results to the case of collapsing Calabi-Yau metrics on compact Calabi-Yau manifolds. In this global setting, the $C^0$ estimate required as a hypothesis in our new local $C^\alpha$ and $C^\infty$ estimates is known to hold thanks to earlier work of the second-named author.
\end{abstract}
\author{Hans-Joachim Hein}
\address{Mathematisches Institut, WWU M\"unster, Einsteinstrasse 62, 48149 M\"unster, Germany}
\address{Department of Mathematics, Fordham University, Bronx, NY 10458, USA}
\email{hhein@fordham.edu}
\author{Valentino Tosatti}
\address{Department of Mathematics and Statistics, McGill University, Montr\'eal, Qu\'ebec H3A 0B9, Canada}
\address{Department of Mathematics, Northwestern University, Evanston, IL 60208, USA}
\email{tosatti@math.northwestern.edu}
\date{\today}
\thanks{Partially supported by NSF grants DMS-1745517 (H.H.) and DMS-1610278, DMS-1903147 (V.T.) and by a Chaire Poincar\'e (V.T.)}
\maketitle
\thispagestyle{empty}
\markboth{Higher-order estimates for collapsing Calabi-Yau metrics}{Hans-Joachim Hein and Valentino Tosatti}

\section{Introduction}\label{intro}

The main object of study in this paper are Ricci-flat K\"ahler metrics on compact Calabi-Yau manifolds, and we wish to understand their behavior in families when their total volume approaches zero. We will work on a fixed Calabi-Yau manifold which admits the structure of a holomorphic fiber space onto a lower-dimensional space (which, away from the singular fibers and from the singularities of the base, is a proper holomorphic submersion), and degenerate the K\"ahler class to the pullback of a K\"ahler class from the base. The Calabi-Yau theorem \cite{Ya} assures the existence of Ricci-flat K\"ahler metrics on the total space in the corresponding K\"ahler classes, whose volume is approaching zero. This setup has been much studied in recent years, starting from the work of Gross-Wilson \cite{GW} on elliptically fibered $K3$ surfaces, and more recently in general dimensions in \cite{GTZ,GTZ2,HT,Li,To,TWY,TZ,TZ2} and elsewhere. It is a very interesting problem with many different aspects, and we refer the reader to \cite{GTZ,GW,KSo,TWY,TZ2} for further ramifications of this circle of ideas. From these previous works, we know that the Ricci-flat metrics collapse to the pullback of a canonical K\"ahler metric on the base, uniformly on compact sets away from the singular fibers. The strongest topology in which this collapse was known to happen is $C^0$ by \cite{TWY}, and $C^\infty$ when the smooth fibers are tori or finite \'etale quotients of tori by \cite{GTZ,HT}. Our main results in this paper substantially improve on these previous works. In particular, our main technical results are purely local on the base and do not require a compact Calabi-Yau total space.

\subsection{$C^\infty$ estimates if the smooth fibers are pairwise biholomorphic} Our results are strongest and easiest to state if all of the smooth fibers are pairwise biholomorphic. To explain the setup, let $Y$ be an $n$-dimensional compact K\"ahler manifold with $c_1(Y)=0$ in $H^2(Y,\mathbb{R})$, equipped with a Ricci-flat K\"ahler form $\omega_Y$. Let $B$ denote the unit ball in $\mathbb{C}^m$ ($m\geq 1$), equipped with a Euclidean K\"ahler form $\omega_{\mathbb{C}^m}$. For each $t\geq 0$ consider the Ricci-flat product K\"ahler form
\begin{align}\omega_t=\omega_{\mathbb{C}^m}+e^{-t}\omega_Y\end{align}
on the product complex manifold $B\times Y$. Further suppose that $\omega^\bullet_t$ is some Ricci-flat K\"ahler form on $B\times Y$ (with respect to the product complex structure) that satisfies
\begin{align}
\omega^\bullet_t=\omega_t+\ddbar\psi_t
\end{align}
for some smooth function $\psi_t$, as well as the Monge-Amp\`ere equation
\begin{equation}\label{mageneral}
(\omega^\bullet_t)^{m+n}=e^F\omega_t^{m+n}=\binom{m+n}{n}e^{-nt+F}\omega_{\mathbb{C}^m}^m\wedge\omega_Y^n
\end{equation}
for some fixed smooth function $F$ (which must be the pullback under ${\rm pr}_B$ of a pluriharmonic function on $B$). Assume in addition that we have a uniform estimate
\begin{equation}\label{unifequiv}
C^{-1}\omega_t \leq \omega^\bullet_t\leq C\omega_t
\end{equation}
on $B\times Y$ for all $t\geq 0$ and for a constant $C$ independent of $t$.

With these preparations, our first main result is the following.

\begin{theorem}\label{mainlocal}
For all compact sets $K\subset B$ and all $k \in \N$, there exists a constant $C_{K,k}$ independent of $t$ such that for all $t \in [0,\infty)$ we have that
\begin{equation}\label{estimate}
\|\omega^\bullet_t\|_{C^k(K\times Y,\omega_t)}\leq C_{K,k}.
\end{equation}
\end{theorem}

Observe that these estimates trivially imply uniform $C^k$ bounds on $K\times Y$ with respect to the \emph{fixed} product metric $\omega_{\mathbb{C}^m}+\omega_Y$ or indeed with respect to any other product metric $\omega_{\C^m} + \tilde{\omega}_Y$. However, the \emph{collapsing} $C^k$ norms in \eqref{estimate} change by unbounded factors if we replace $\omega_Y$ by $\tilde\omega_Y$ unless $\nabla^{\omega_Y}\tilde\omega_Y = 0$; thus, in order for \eqref{estimate} to hold it is actually essential that ${\rm Ric}(\omega_Y) = 0$. Also note that Theorem \ref{mainlocal} would be false in general without the assumption that $Y$ is compact without boundary (cf. Remarks \ref{cntrxmpl}, \ref{cntrxmpl2}), and indeed our method of proof is fundamentally global on $Y$ (cf. Remark \ref{ccc}).

\begin{rk}\label{dirichlet}
If for each $t$ we are given a smooth function $\psi_t$ on $\de B\times Y$ with $\omega_t +\ddbar\psi_t>0$ there, then $\psi_t$ uniquely extends to a solution of \eqref{mageneral} on $B\times Y$. Indeed, by \cite[Prop 7.10]{Bo} we can first extend $\psi_t$ to a smooth strictly $\omega_{t}$-psh function on $B\times Y$. Adding a large multiple of a smooth strictly psh function on $B$ that vanishes on $\partial B$, we obtain a subsolution of \eqref{mageneral}. Then a solution exists e.g. by \cite[Thm A]{Bo}. However, this solution may not satisfy \eqref{unifequiv} with $C$ independent of $t$.
\end{rk}

The main application of this ``local on the base'' result is to compact Calabi-Yau manifolds. Let now $f:X\to B$ be a surjective holomorphic map with connected fibers (also called a {\em fiber space}), where $X$ is a compact K\"ahler Calabi-Yau manifold of dimension $m+n$ and $B$ is a compact K\"ahler space of dimension $m$ (which is necessarily irreducible, and we assume is reduced). The set of critical points of $f$ (including by default the preimages of singular points of $B$) will be denoted by $S$,
and we will let $X_b =$ $f^{-1}(b)$ be the smooth fiber over any $b \in B \setminus f(S)$, which is a K\"ahler $n$-manifold with $c_1(X_b)=0$ in $H^2(X_b,\mathbb{R})$ (i.e., also a Calabi-Yau manifold).
Fix K\"ahler metrics $\omega_X,\omega_B$ on $X,B$, with $\omega_X$ Ricci-flat, and put $\omega_\infty =f^*\omega_B$, which is a smooth semipositive definite real $(1,1)$-form on $X$. For all $t \in [0,\infty)$ let $\omega^\bullet_t$ be the unique solution of the
 complex Monge-Amp\`ere equation
\begin{equation}\label{mageneral2}
(\omega^\bullet_t)^{m+n}=(\omega_\infty+e^{-t}\omega_X+\ddbar\psi_t)^{m+n}=c_t e^{-nt}\omega_X^{m+n}, \quad \sup\nolimits_X \psi_t=0,
\end{equation}
whose existence is guaranteed by Yau's theorem \cite{Ya}. In other words, $\omega^\bullet_t$ is the unique Ricci-flat K\"ahler metric on $X$ cohomologous to $\omega_\infty+e^{-t}\omega_X$. Here the constant $c_t$ is defined by integrating \eqref{mageneral2} over $X$, and it has a positive limit as $t\to\infty$. This is exactly the same setup which is studied for example in \cite{GTZ,GTZ2,HT,Li,To,TWY,TZ,TZ2}. A key result, conjectured in \cite{KT} and proved independently in \cite{DP, EGZ}, is that $\sup_X |\psi_t|\leq C$, independent of $t$. In \cite{To} (after earlier results in \cite{ST0} when $n=m=1$) it was proved that $\ddbar\psi_t$ is uniformly bounded on compact sets away from $f^{-1}(f(S))$ (i.e., an analog of \eqref{unifequiv} was proved on any such compact set), and in \cite{TWY} it was proved that in fact $\ddbar\psi_t$ has a well-defined limit in $C^0_{\rm loc}(X\setminus f^{-1}(f(S)))$.
As a corollary of our main Theorem \ref{mainlocal}, we can improve this to uniform $C^\infty$ estimates for $\psi_t$ on compact sets away from $f^{-1}(f(S))$ if the fibers $X_b$ ($b\in B\setminus f(S))$ are pairwise biholomorphic to each other, thus resolving \cite[Question 4.2]{To2} and \cite[Question 5.2]{To3} in this case.

\begin{corollary}\label{main}
Assume that all the fibers  $X_b$ $(b\in B\setminus f(S))$ are biholomorphic to the same Calabi-Yau manifold $Y$. Over any small coordinate ball $U$ compactly contained in $B\setminus f(S)$, use \cite{FG} to trivialize $f$ holomorphically to a product $U\times Y\to U$. As before define Ricci-flat reference K\"ahler forms on $U\times Y$ by $\omega_t=\omega_{\mathbb{C}^m}+e^{-t}\omega_Y$. Then for any $k \in \N$, there exists a constant $C_{U,k}$ such that
\begin{equation}\label{estimate2}
\|\omega^\bullet_t\|_{C^k(U\times Y,\omega_t)}\leq C_{U,k}
\end{equation}
holds uniformly for all $t \in [0,\infty)$. In particular, given any compact set $K\subset X\setminus f^{-1}(f(S))$ and any $k \in \N$, there exists a constant $C_{K,k}$ such that
\begin{equation}\label{estimate3}
\|\omega^\bullet_t\|_{C^k(K,\omega_X)}\leq C_{K,k}
\end{equation}
holds uniformly for all $t \in [0,\infty)$.
\end{corollary}

To see that such fiber spaces exist with $S\neq \emptyset$, let $E$ be an elliptic curve with the involution $\sigma$ induced by $z\mapsto -z$, let $Y$ be a $K3$ surface with a nonsymplectic involution $\tau$, and let $X =(E\times Y)/(\sigma\times\tau)$ be the quotient by the diagonal action with its natural map $f: X\to B = E/\sigma=\mathbb{CP}^1$. If $\tau$ is free (i.e., is the covering involution of an Enriques surface), then $X$ is smooth and $f$ is a fibration of the required kind (with $4$ double fibers that are Enriques surfaces), although in this case $\omega_t^\bullet$ comes from a product metric on $E\times Y$, hence is itself a product metric locally away from the singular fibers. If $\tau$ is \emph{not} free (e.g., is the covering involution of a double sextic), then $X$ is singular, but replacing $X$ by a blowup we again obtain a smooth Calabi-Yau $3$-fold fibered over $\mathbb{CP}^1$ with all smooth fibers biholomorphic (and with $4$ reduced singular fibers that are normal crossing divisors in $X$); cf. \cite{Bor,Vo}. Then $\omega_t^\bullet$ is certainly not a product metric even locally away from the singular fibers. However, it may still be possible to construct $\omega_t^\bullet$ using a gluing method in the spirit of \cite[Problem 1.11]{He}; cf. \cite{CVZ}.

It is worth remarking that in the setting of Corollary \ref{main}, assuming $Y$ is not a torus (or a finite \'etale quotient of a torus), then the Ricci-flat metrics $\omega^\bullet_t$ do not have uniformly bounded sectional curvature as $t\to \infty$ on $f^{-1}(U)$ for any $U\subset B\setminus f(S)$.  Indeed, if the curvature of $\omega^\bullet_t$ does remain bounded on $f^{-1}(U)$, then it follows from \cite[Thm 3.1]{TZ} that $(Y,\omega_Y)$ must be flat. Conversely, if $Y$ is (a finite \'etale quotient of) a torus, then smooth collapsing of the Ricci-flat metrics $\omega_t^\bullet$ with bounded curvature was proved in \cite{GTZ,HT}, and this is the only case where \eqref{estimate3} was known previously. In fact, \cite{GTZ,HT} proved $C^\infty$ estimates in the torus-fibered case without assuming that the fibers are pairwise biholomorphic.

\subsection{A general $C^\alpha$ estimate} If the smooth fibers are not necessarily biholomorphic to each other, we can push our techniques to their limit and obtain the following ``local on the base'' $C^\alpha$ estimate.

Let $f: X \to B$ be a proper surjective holomorphic submersion onto the unit ball $B=B_1(0)\subset\C^m$ such that the fibers of $f$ are $n$-dimensional Calabi-Yau manifolds. Suppose $X$ is equipped with a Ricci-flat K\"ahler metric $\omega_X$. Applying Yau's theorem fiberwise, it is easy to construct a smooth closed real $(1,1)$-form $\omega_F=\omega_X+\ddbar\rho$ on $X$ whose restriction to every fiber $X_z=f^{-1}(z)$ is the Ricci-flat K\"ahler metric on $X_z$ cohomologous to $\omega_X|_{X_z}$ (see Section \ref{nonpr} for details). Letting $\omega_\infty=f^*\omega_{\C^m}$, it is not hard to show that up to shrinking $B$ slightly and taking $t$ sufficiently large, the forms $\omega_\infty+e^{-t}\omega_F$ define K\"ahler metrics on $X$. Suppose $\omega^\bullet_t$ is a Ricci-flat K\"ahler metric on $X$ which satisfies
\begin{align}\omega^\bullet_t=\omega_\infty+e^{-t}\omega_F+\ddbar\psi_t\end{align}
for some smooth function $\psi_t$ together with
\begin{equation}\label{mageneralx}
(\omega^\bullet_t)^{m+n}=c_te^{-nt+G}\omega_\infty^m\wedge\omega_F^n
\end{equation}
for some smooth function $G$ pulled back from $B$. Assume in addition that
\begin{equation}\label{unifequivx}
C^{-1}(\omega_\infty+e^{-t}\omega_F)\leq \omega^\bullet_t\leq C(\omega_\infty+e^{-t}\omega_F)
\end{equation}
holds on $X$ for all $t\geq 0$ and for some constant $C$ independent of $t$.

Up to shrinking $B$ again, there is a $C^\infty$ trivialization $\Phi:B\times Y\to X$, where $Y=f^{-1}(0)$ is viewed as a smooth real $2n$-manifold, such that $\Phi|_{\{0\}\times Y}:\{0\}\times Y\to Y$ is the identity. (Of course $\Phi$ is never unique but our main result holds for every possible choice of $\Phi$.) For $z\in B$ let $g_{Y,z}$ denote the trivial extension to $B\times Y$ of the pullback via $\Phi$ of the Ricci-flat Riemannian metric associated with $\omega_F|_{X_z}$. We can then define a family of collapsing product Riemannian metrics on $B\times Y$ by $$g_{z,t}=g_{\C^m}+e^{-t}g_{Y,z}.$$
Each of these metrics is uniformly equivalent (with a constant independent of $t$) to the metric obtained by pulling back the metric associated with $\omega_\infty+e^{-t}\omega_F$ via $\Phi$. Let further $g^\bullet_t$ denote the Riemannian metric on $B\times Y$ obtained by pulling back the metric associated with $\omega^\bullet_t$ via $\Phi$.

With these preparations, our second main result may now be stated as follows.

\begin{theorem}\label{thm51}
For all $0 < \alpha < 1$, there exists a constant $C_\alpha$ such that
\begin{equation}\label{estimatex}
\sup_{x=(z,y) \in B_{\frac{1}{4}}(0) \times Y} \sup_{x' \in B^{g_{z,t}}(x,\frac{1}{8})}\frac{|{g}_t^\bullet(x) -  \mathbf{P}^{g_{z,t}}_{x'x}({g}_t^\bullet(x'))|_{g_{z,t}(x)}}{d^{g_{z,t}}(x,x')^\alpha} \leq C_\alpha
\end{equation}
holds uniformly for all $t\in [0,\infty)$.
\end{theorem}

Here $\mathbf{P}^{g_{z,t}}_{x'x}$ denotes the Riemannian parallel transport operator from $x'$ to $x$ associated with $g_{z,t}$, and $d^{g_{z,t}}$ denotes the Riemannian distance associated with $g_{z,t}$. The estimate \eqref{estimatex} is subtly weaker than a $C^\alpha$ bound for $g^\bullet_t$ with respect to $g_{z_0,t}$ for any fixed $z_0 \in B$; in fact, as we will see in Remark \ref{totalfuckage}, a $C^\alpha$ bound of the latter kind cannot hold unless $f$ is a local product or the fibers are flat. However, by Remark \ref{important}, \eqref{estimatex} does imply a uniform $C^\alpha$ bound for $g^\bullet_t$ with respect to any \emph{$t$-independent} product metric on $B \times Y$.  With this in mind, we obtain the following direct application of Theorem \ref{thm51}.

\begin{corollary}\label{main2}
Given a fiber space $f:X\to B$ where $X$ is a compact $(m+n)$-dimensional Calabi-Yau manifold and $B$ is an $m$-dimensional compact K\"ahler space, let $\omega^\bullet_t$ be the Ricci-flat K\"ahler metrics on $X$ defined by \eqref{mageneral2}. Then for any compact set $K\subset X\setminus f^{-1}(f(S))$ and for any $0<\alpha<1$, there exists a constant $C_{K,\alpha}$ such that for all $t \in [0,\infty)$,
\begin{equation}\label{estimate4}
\|\omega^\bullet_t\|_{C^\alpha(K,\omega_X)}\leq C_{K,\alpha}.
\end{equation}
\end{corollary}

As a consequence of the new uniform boundedness results of Corollaries \ref{main} and \ref{main2}, together with the results of \cite{GTZ,HT,To,TZ}, we obtain the following adiabatic limit theorem.

\begin{corollary}\label{coro}
In the setting of Corollary \ref{main2}, let $\omega_{\rm can}$ denote the unique weak solution of
\begin{equation}\label{limit}
\omega_{{\rm can}}^m=(\omega_B+\ddbar v)^m=\frac{\int_B \omega_B^m}{\int_X \omega_X^{m+n}}f_*(\omega_X^{m+n})
\end{equation}
with $v \in L^\infty(B) \cap C^\infty_{\rm loc}(X \setminus f(S))$, where fiber integration, $f_*$, is defined only on $X\setminus f^{-1}(f(S))$ but the right-hand side of \eqref{limit} defines a unique measure on $B$ with $L^p$ density w.r.t. $\omega_B^m$ for some $p > 1$. Then the Ricci-flat metrics $\omega^\bullet_t$ converge to $f^*\omega_{\rm can}$ as $t\to\infty$ in the topology of $C^\alpha_{{\rm loc}}(X\setminus f^{-1}(f(S)))$ for all $0<\alpha<1$. Moreover, the convergence takes place in $C^\infty_{{\rm loc}}(X\setminus f^{-1}(f(S)))$ if the  regular fibers are tori or finite \'etale quotients of tori, or are pairwise biholomorphic to each other.
\end{corollary}

The stated properties of $\omega_{\rm can}$ are known thanks to \cite{ST}, and \eqref{limit} implies that the Ricci curvature of $\omega_{\rm can}$ is equal to a certain Weil-Petersson form; see \cite{ST, To} for details. The main result of \cite{TWY} is that $\omega^\bullet_t\to f^*\omega_{{\rm can}}$ in $C^0_{{\rm loc}}(X\setminus f^{-1}(f(S)))$ (with weaker convergence established earlier in \cite{To}). In the case of torus fibers, $C^\infty_{\rm loc}$ convergence was established in \cite{GTZ, HT}. By Corollaries \ref{main} and \ref{main2}, these results are now improved to $C^\infty_{\rm loc}$ if the smooth fibers are pairwise isomorphic, and to $C^\alpha_{\rm loc}$ in general. Since our proof does not rely on \cite{TWY}, this also gives a new proof of the main result \cite[Thm 1.3]{TWY}.

\subsection{Previous work} Partial results in the direction of Corollaries \ref{main} and \ref{main2} had been proved earlier. If $f$ is an elliptic fibration of a $K3$ surface with $24$ singular fibers of Kodaira type $I_1$, Gross-Wilson \cite{GW} obtained a complete asymptotic description of $\omega_t^\bullet$ using a gluing method. In the general setting, a $C^0_{\rm loc}$ estimate was proved in \cite{To}. Certain components of the first derivative were bounded in \cite[Thm 2.3]{To} and \cite[Prop 4.8]{TWY}. An even stronger partial estimate was proved in \cite[Thm 3.1]{TZ}: the restriction of $e^t\omega^\bullet_t$ to $f^{-1}(U)$ (for $U$ a small ball in $B\setminus f(S)$) converges in the pointed $C^\infty$ Cheeger-Gromov topology (i.e., modulo ``stretching'' diffeomorphisms applied to the base directions) to the product of a flat $\mathbb{C}^m$ with a fiber of $f$ equipped with its preferred Ricci-flat metric. However, it is not clear how to use these ideas even to prove a full $C^\alpha_{\rm loc}$ estimate because in this paper we do not allow any reparametrization by diffeomorphisms, and the ellipticity of  \eqref{mageneral2} degenerates as $t \to \infty$, so that the methods of Calabi-Yau \cite{Cal, Ya}, Evans-Krylov \cite{Tr}, Tian \cite{gangsta}, Caffarelli \cite{Ca}, and Wang-Wu \cite{WaWu} cannot be applied directly.

The only known exception to this statement  is the case when $X_b$ is finitely covered by a torus (even without assuming that all smooth fibers are pairwise biholomorphic). As mentioned above, if $X_b$ is a torus, \eqref{estimate3} was proved in \cite{GTZ} if $X$ is projective and in \cite{HT} in general, and the case of finite quotients of tori was pointed out in \cite[p.2942]{TZ}. It turns out that in this case the standard methods can be set to work after all using the following idea: take any ball $U\subset B \setminus f(S)$ and pull back the Ricci-flat metrics $\omega_t^\bullet$ to the universal cover of $f^{-1}(U)$, which is biholomorphic to $U\times\mathbb{C}^n$; then stretch the coordinates on $\mathbb{C}^n$ by a factor of $e^{t/2}$ to make \eqref{mageneral2} uniformly elliptic, and apply the standard theory. To make this rigorous, a construction of semi-flat reference metrics on $f^{-1}(U)$ is required \cite{GSVY,GTZ, GW,He, HT,TZ2}. In fact, in \cite{GTZ,HT}, \eqref{estimate2} was proved with $\omega_t$ replaced by these (collapsing) semi-flat metrics. See also \cite{Gi} for analogous estimates for the K\"ahler-Ricci flow on $B\times T$, where $T$ is a torus and $c_1(B)<0$.

Also, if $S=\emptyset,$ i.e., if $f$ is a submersion, then global $C^\infty$ estimates are implied by the more general work of Fine \cite{Fi0,Fi} on cscK metrics, but the assumption that $S=\emptyset$ is very strong if $X$ is Calabi-Yau because it implies that $f$ is a holomorphic fiber bundle (cf. \cite{TZ3}, \cite[Thm 3.3]{TZ2}). If $S \neq \emptyset$, then the methods of \cite{Fi0,Fi} can still be used to some extent, but they only give us one particular family of solutions of \eqref{mageneral2} with good $C^\infty$ bounds on each tube $f^{-1}(U)$, and because of local non-uniqueness (cf. Remark \ref{dirichlet}) there is then no reason why this good family would agree with $\omega^\bullet_t|_{f^{-1}(U)}$.

In \cite[Question 4.2]{To2} and \cite[Question 5.2]{To3} it is conjectured that \eqref{estimate3} should still hold without the assumption that all smooth fibers are pairwise biholomorphic (i.e., that \eqref{estimate4} can be improved to $C^k$ for all $k$). This is known only if the smooth fibers are flat \cite{GTZ,HT,TZ} or if $f$ is a submersion \cite{Fi0,Fi}. In Remark \ref{discuss} we discuss why our current method is not sufficient to prove this conjecture.

\subsection{Overview of the proofs (part 1)} What allows us to go beyond the known partial estimates is a systematic use of iterated blowup-and-contradiction type arguments. Ultimately the reason why we get a contradiction in the end (``solutions of polynomial growth on $\C^m \times Y$ that are not polynomials on $\C^m$'') is separation of variables. At the linear level, this amounts to a Fourier decomposition along $Y$ of harmonic forms or functions on $\C^m \times Y$. In particular, our results are fundamentally local on the base but global on the fibers. Our purpose in this subsection is to clarify this point.

\begin{rk}\label{cntrxmpl}
The fact that $Y$ has no boundary is crucial for the estimates of Theorem \ref{mainlocal} to hold. In fact, the corresponding local result (where $Y$ would be a ball in $\C^n$) is false even if we shrink $Y$ on the left-hand side of \eqref{estimate}. We are grateful to A. Figalli and O. Savin for the following counterexample in the real setting. Consider the convex function $u_t$ on the unit square $[0,1]\times [0,1]$ given by
$$u_t(x_1,x_2)=\frac{x_1^2+e^{-t}x_2^2}{2}+e^{-t} \delta w\left({x_1}e^{\frac{t}{2}},{x_2}\right),$$
where $\delta > 0$ is small and $w$ is an $x_1$-periodic perturbation by $O(\delta)$ of
$$w'(x_1,x_2)= \sin(2\pi x_1)e^{x_2}.$$
We choose the perturbation $w$ such that $u_0$ solves the Monge-Amp\`ere equation $\det(D^2 u_0)=1$ on the unit square. This is possible for $\delta$ small because $w'$ is harmonic, so taking $w=w'$ solves the linearized Monge-Amp\`ere equation. Then $u_t$ is smooth and convex and solves $\det(D^2 u_t)=e^{-t}$ on $[0,1]\times [0,1]$, and $D^2u_t$ is uniformly equivalent (with constants independent of $t$) to
$$g_t=\left( {\begin{array}{cc}
   1 & 0 \\
   0 & e^{-t} \\
  \end{array} } \right).$$
Thus, $u_t$ satisfies the appropriate analogs of \eqref{mageneral} and \eqref{unifequiv}. Nevertheless, for all $k \geq 3$,
$$|D^k u_t|_{g_t}\sim e^{(k-2)\frac{t}{2}} \to \infty\;\,{\rm as}\;\,t \to \infty$$
a.e. in $[0,1]\times [0,1]$. A similar counterexample for the complex Monge-Amp\`ere equation is given by
$$\ti{u}_t(z_1,z_2)=|z_1|^2+e^{-t}|z_2|^2+e^{-t}\delta w\left({x_1}e^{\frac{t}{2}},{x_2}\right)$$
in the unit polydisc in $\mathbb{C}^2$ with coordinates $z_j=x_j+iy_j$, where $w$ is the same as before.
\end{rk}

\begin{rk}\label{ccc} Let us conversely explain why it is more reasonable to expect higher order estimates if $Y$ has no boundary. First of all, if we modify the example of Remark \ref{cntrxmpl} by replacing $[0,1] \times [0,1]$ by $[0,1]\times S^1$ (so that the fibers $S^1 = \mathbb{R}/\mathbb{Z}$ are now compact without boundary), the harmonic function $\sin(2\pi x_1)e^{x_2}$ used above is ruled out, but $e^{x_1}\sin(2\pi x_2)$ is not. Unlike above, the latter does not remain uniformly bounded if we replace $x_1$ by $x_1 e^{t/2}$, so we now need to pick $\delta$ small relative to \begin{small}$e^{-e^{t/2}}$\end{small} rather than just absolutely small. Then $D^2 u_t \sim g_t$ as before, but for all $k \geq 3$ and $0 < \epsilon < \frac{1}{2}$,
$$
\sup_{[\epsilon,1-\epsilon]\times S^1} |D^k u_t|_{g_t} \leq C_{k,\epsilon} e^{-\epsilon e^{\frac{t}{2}}} e^{(k-2)\frac{t}{2}} = O_{k,\epsilon}(1)\;\,{\rm as}\;\,t\to\infty.
$$
Now even at the linear level the question remains as to how to go about proving that this behavior is in fact universal. For example, why would a solution $v$ to $\Delta_{\R^d} v + e^t \Delta_Y v = 0$ on $B \times Y$ (with $B$ the unit ball in $\R^d$ and $Y$ a compact manifold without boundary) satisfy uniform interior estimates?

If $Y$ is a torus, one can simply pass to the universal cover. Let $\tilde{v}$ denote the lift of $v$ to the universal cover, and let $(z,y)$ denote fixed linear coordinates on the universal cover. Then $\hat{v}(z,y) = \tilde{v}(z,e^{\frac{t}{2}}y)$ is harmonic on $B \times B$ and the resulting standard interior estimates for $\hat{v}$ translate back into precisely the right interior estimates for $v$ on $B \times Y$. In a nutshell, this is the philosophy of \cite{GTZ,HT}, where a $C^\infty$ version of Corollary \ref{main2} was proved by an analogous covering trick if the regular fibers $X_b$ are tori. In fact, these papers establish a close analog of Corollary \ref{main}, where the collapsing product metrics $\omega_t$ get replaced by carefully constructed collapsing \emph{semi-flat} metrics \cite{GSVY,GTZ, GW,He, HT,TZ2}.

If $Y$ is not a torus, this covering trick is no longer available. The philosophy of the present paper is to instead use separation of variables, expanding $(\Delta_{\R^d} + e^t\Delta_Y)$-harmonic functions $v$ on $B \times Y$ according to the eigenfunctions of $\Delta_Y$ on each fiber $\{z\} \times Y$. This suggests that we might expect that
$$\sup_{B_{1-\epsilon} \times Y} |D^k v|_{g_t} \leq C_{k,\epsilon} e^{- \lambda \epsilon e^{\frac{t}{2}}} e^{k\frac{t}{2}}  \sup_{B \times Y} |v|$$
for all $k \in \N$ and $0 < \epsilon < 1$, where $\lambda > 0$ and $\lambda^2$ denotes the first positive eigenvalue of $\Delta_Y$. This idea is the basic source of all the new estimates in this paper, but a great deal of technical work is required to make this idea sufficiently precise even at the linear level (cf. Section \ref{ov2}).
\end{rk}

\begin{rk}\label{cntrxmpl2}
We can also compare Theorem \ref{mainlocal} to a formally identical equation that arises naturally in K\"ahler geometry where $\partial Y \neq \emptyset$, and where \eqref{unifequiv} and higher order estimates fail even though the $C^{1,1}$ norm of the potential remains bounded. Indeed, let $(X^n,\omega)$ be a closed K\"ahler manifold. Let $\Sigma$ be a closed annulus in $\mathbb{C}$, with coordinate $z$. Let $\pi:X\times\Sigma\to X$ be the projection, and let $\vp_0,\vp_1$ be two $\omega$-psh functions on $X$. For $\ve>0$, an $\ve$-geodesic connecting $\vp_0$ and $\vp_1$ is a smooth function $\Phi_\ve$ on $X\times\Sigma$ such that $\pi^*\omega+\ve idz\wedge d\ov{z}+\ddbar\Phi_\ve>0$, $\Phi_\ve$ is equal to $\vp_0$ on one boundary component of $X\times\Sigma$ (where $\vp_0$ is taken to be constant on the $S^1$ factor) and to $\vp_1$ on the other, and it solves
\begin{equation}\label{bubbo1}
(\pi^*\omega+\ve idz\wedge d\ov{z}+\ddbar\Phi_\ve)^{n+1}=\ve \pi^*\omega^n\wedge idz\wedge d\ov{z}.
\end{equation}
These $\ve$-geodesics always exist thanks to \cite{Ch} (see \cite{Bo} for a good exposition), and formally this equation is the same as \eqref{mageneral} where $\ve$ corresponds to $e^{-t}$, the $Y$ factor is replaced by $\Sigma$ and $\mathbb{C}^m$ by $X$. Since the boundary data is $S^1$ invariant, a maximum principle argument shows that so is $\Phi_\ve$, and so in the $\Sigma$ factor it only depends on $r=|z|$.
However, unless $\vp_0- \vp_1 = {\rm const}$, in general only the $C^{1,1}$ norm of $\Phi_\ve$ remains uniformly bounded as $\ve\to 0$, while higher order derivatives blow up, see \cite{CTW} and references therein. But in this situation the analog of \eqref{unifequiv} already fails: we claim that if it is satisfied uniformly in $\ve$, then $\vp_0 - \vp_1 = {\rm const}$. Indeed, it is enough to just assume that
\begin{equation}\label{bubbo2}
\pi^*\omega+\ve idz\wedge d\ov{z}+\ddbar\Phi_\ve\geq C^{-1}\pi^*\omega,
\end{equation}
on $X\times\Sigma$ for some $C$ independent of $\ve$, which is much weaker than the analog of \eqref{unifequiv}. Observe that it follows from \eqref{bubbo2} together with
\eqref{bubbo1} that
$$0\leq (\pi^*\omega+\ve idz\wedge d\ov{z}+\ddbar\Phi_\ve)|_{\{x\}\times\Sigma}\leq C\ve idz\wedge d\ov{z},$$
for all $x\in X$ and $\ve>0$. Thus,
$$\sup\nolimits_X|\ddot{\Phi}_\ve|\leq C\ve,$$
where dots denote derivatives with respect to $r$,
so the $C^{1,1}$ limit $\Phi=\lim_{\ve\to 0}\Phi_\ve$ satisfies $\ddot{\Phi}=0$ a.e. and hence is a trivial geodesic, which implies that $\vp_0$ and $\vp_1$ only differ by a constant.

It is perhaps worth remarking that this failure of \eqref{bubbo2} appears to be a genuine ``boundary'' issue. Indeed, note that in the setting of Corollary \ref{main2}, the analog of \eqref{unifequiv} or \eqref{bubbo2} is \eqref{unifequivx}, and was proved in \cite[Lemma 3.1]{To} (cf. \cite{ST0}) using a Yau Schwarz lemma argument \cite{Ya2}. If we try to imitate the same computation in the $\epsilon$-geodesic setting, aiming to prove \eqref{bubbo2}, we get
\begin{equation}\label{bubbo3}
\Delta^{\omega_\ve}(\log\tr{\omega_\ve}{(\pi^*\omega)}-A\Phi_\ve)\geq \tr{\omega_\ve}{(\pi^*\omega)}-C,
\end{equation}
where $A,C$ are uniform constants ($A$ is sufficiently large) and we have set $\omega_\ve=\pi^*\omega+\ve idz\wedge d\ov{z}+\ddbar\Phi_\ve$. Now suppose that $\vp_0 - \vp_1 \neq {\rm const}$, so that the estimate $\tr{\omega_\ve}{(\pi^*\omega)}\leq C$ (which is equivalent to \eqref{bubbo2}) cannot possibly hold with a uniform $C$, as shown above. Since $\sup_{X\times\Sigma}|\Phi_\ve|\leq C$ by \cite{Ch}, the maximum principle applied to \eqref{bubbo3} tells us that $\sup_{X \times \Sigma} (\log\tr{\omega_\ve}{(\pi^*\omega)}-A\Phi_\ve)$ must eventually be achieved on the boundary (and then $\sup_{X\times\de\Sigma}\tr{\omega_\ve}{(\pi^*\omega)}$ must of course blow up as $\ve\to 0$).
\end{rk}

\subsection{Overview of the proofs (part 2)}\label{ov2} As we already said, Theorems \ref{mainlocal} and \ref{thm51} will be proved by contradiction, and in the previous subsection we explained the source of the contradiction (separation of variables, relying on the fact that the fibers have no boundary). We will now explain the structure of the blowup argument more carefully. If the desired estimates do not hold, we obtain a sequence of solutions where the desired quantity blows up to infinity. We then distinguish three cases according to whether this quantity blows up faster than the ``natural'' parameter $e^{t}$, at the same rate, or slower. Rescaling our setting appropriately, we obtain as blowup limit spaces $\C^{m+n}, \C^m\times Y$ and $\C^m$ in the three cases respectively, and our Ricci-flat metrics converge in a suitable sense to Ricci-flat metrics on these spaces which are not ``trivial'' but are uniformly equivalent to the obvious model metrics in each case. (The fact that the limit metric is Ricci-flat is not obvious in the $\C^m$ case but was already proved in \cite{To}.) This contradicts certain Liouville theorems for Ricci-flat K\"ahler metrics, which are standard on $\C^{m+n}$ and $\C^m$ \cite{RS}, but in the $\C^m \times Y$ case were only proved relatively recently in \cite{He2,LLZ}.

In \cite{CW} the usual Liouville theorem for Ricci-flat K\"ahler metrics on $\C^d$ was used to prove the Evans-Krylov estimate for the complex Monge-Amp\`ere equation on a ball in $\C^d$ by blowup and contradiction. This corresponds to the rapidly forming case with blowup limit $\C^{m+n}$ in our setting, although here we are happy to \emph{assume} Evans-Krylov to simplify matters. The ``natural'' case with blowup limit $\C^m \times Y$ is then reasonably similar to the first case, given the new Liouville theorem on $\C^m \times Y$ from \cite{He2, LLZ}. Thus, for us, almost all of the difficulty is concentrated in the slowly forming case with limit $\C^m$.

More specifically, it is a priori unclear in this case how to prove that the collapsing Ricci-flat metrics pass to the limit in a sufficiently strong topology to ensure that their limit is not flat. We overcome this issue using a combination of two arguments. First of all, we prove a sharp new Schauder estimate for the Laplacian on balls of arbitrary radii in $\C^m \times Y$. This is itself proved by blowup and contradiction in the spirit of \cite{HN, Si}, where again the same three cases arise as in the overall nonlinear argument. (The collapsing case with limit space $\C^m$ is again the hardest case here and suffers from similar ``weak convergence'' issues as the collapsing case in the overall nonlinear argument. However, these issues are less severe in the linear setting, which prevents a logical cycle.) The use of this Schauder estimate is to slightly improve the regularity of the collapsing Ricci-flat metrics. But since this improved regularity is itself measured with respect to a collapsing rather than a fixed reference metric, there is no obvious version of the Ascoli-Arzel\`a theorem for tensors that would imply convergence in a sufficiently strong topology. Our second key argument (after the linear Schauder estimates) overcomes this final issue by exploiting the K\"ahler property of the metrics $\omega_t^\bullet$ in a delicate manner (precisely, the fact that they can be written as the derivative of another tensor after subtracting a suitable reference metric). In essence, this is also what is needed to pass to a contradictory limit on $\C^m$ in the collapsing case of the linear Schauder theory. See Lemma \ref{l:braindead} and Proposition \ref{bd} for this crucial ``exactness'' argument.

This outline covers both Theorems \ref{mainlocal} and \ref{thm51}. However, whereas the proof of Theorem \ref{mainlocal} follows this outline rather closely, the proof of Theorem \ref{thm51} is more involved. Most importantly, if the complex structure is not a product, it turns out that there is no clean way to isolate the required linear Schauder theory as a separate step; instead, the three cases of the linear blowup argument must be carried out as a nested sub-step within the third case of the nonlinear blowup argument.

\begin{rk}
As in \cite{FZ,TWY} (see also \cite[\S 5.14]{To4}), we expect that the methods we introduced in this paper in the elliptic context (including the Schauder estimates of Section \ref{linear}) will adapt to the parabolic context, with the aim of proving analogs of  Corollaries \ref{main} and \ref{main2} for the K\"ahler-Ricci flow on compact K\"ahler manifolds with semiample canonical bundle and intermediate Kodaira dimension.
\end{rk}

\begin{rk}
In this direction, let us also point out that Theorems \ref{mainlocal} and \ref{thm51} do not rely on the Ricci-flatness of $\omega^\bullet_t$ in any deep differential-geometric way. All we use in the proofs is that the fibers of $f$ are Calabi-Yau manifolds and that the K\"ahler metrics $\omega^\bullet_t$ satisfy the Monge-Amp\`ere equation \eqref{mageneral} (resp. \eqref{mageneralx}) with the function $F$ (resp. $G$) pulled back from $B$. This does not in general imply that $\omega^\bullet_t$ is Ricci-flat (it does imply that its Ricci curvature is a pullback from $B$), but in our arguments these properties suffice to conclude that the limit metrics obtained after blowup (on $\C^{m+n}, \C^m\times Y$ or $\C^m$) are in fact Ricci-flat, contradicting the appropriate Liouville theorems for Ricci-flat metrics. Since our main geometric applications concern Calabi-Yau manifolds, we will not belabor this point.
\end{rk}

\subsection{Organization of the paper} Section \ref{partialC1} gathers some local estimates and Liouville theorems for Ricci-flat metrics from the literature and adapts them slightly to fit our needs. The first new technical ingredient, proved in Section \ref{linear}, is a Schauder estimate on balls in Riemannian cylinders $\mathbb{R}^d\times Y$ (here $Y$ is an arbitrary closed manifold), with sharp dependence of the constants on the radius of the ball. For convenience, and to highlight exactly what the ingredients are, we prove this estimate in a general framework in Theorem \ref{t:schauder}, which we then specialize to $\ddbar$-exact real $(1,1)$-forms in Theorem \ref{moron} and to scalar functions in Theorem \ref{scalar}. Theorem \ref{mainlocal} is proved in Section \ref{sekt}, via a blowup argument and using the local estimates and Liouville theorems of Section \ref{partialC1} and the Schauder estimates of Section \ref{linear}. Section \ref{nonpr} contains the proof of Theorem \ref{thm51}, which is similar in spirit to the proof of Theorem \ref{mainlocal} but requires new ideas because of the varying fiberwise complex structures. In particular, instead of using the ready-made Schauder Theorems \ref{moron} or \ref{scalar}, we will go back to the general Schauder Theorem \ref{t:schauder} and use the key steps of its proof as ingredients of the overall proof of Theorem \ref{thm51}. Lastly Corollaries \ref{main} and \ref{main2} are quickly derived from Theorems \ref{mainlocal} and \ref{thm51} respectively in Section \ref{compac}.

\subsection{Acknowledgments} We are grateful to A. Figalli and O. Savin for showing us the example in Remark \ref{cntrxmpl}, which motivated us to develop the methods of this paper. We would also like to thank Y. Zhang for some very helpful conversations regarding Proposition \ref{p:closure}, Y. Rubinstein for pointing out the reference \cite{gangsta}, and F.T.-H. Fong and M.-C. Lee for finding a mistake in Sections \ref{claim3product} and \ref{claim3nonproduct} in a previous version of this paper. This work was completed during the second-named author's visits to the Center for Mathematical Sciences and Applications at Harvard University and to the Institut Henri Poincar\'e in Paris (supported by a Chaire Poincar\'e at IHP funded by the Clay Mathematics Institute), which he would like to thank for the hospitality.

\section{Local estimates and Liouville theorems for Ricci-flat metrics}\label{partialC1}

In this section we gather together some known results and adapt them slightly to our purposes.

\subsection{Local estimates}
To start, we recall the following local $C^\infty$ bounds for Ricci-flat K\"ahler metrics, which go back to \cite{Ya} and appear explicitly e.g. in \cite[Sections 3.2, 3.3]{HT} and \cite[Lemma 2.2]{TZ}. These can be proved using the methods of \cite{Ca, Cal, gangsta, Tr, Ya, WaWu} and of elliptic bootstrapping.

\begin{proposition}\label{localest2}
For all $d,k \in \N_{\geq 1}$ and $A \geq 1$, there exists a constant $C_k = C_k(d,A)$ such that the following holds. Let $B_1(0)$ denote the unit ball in $\mathbb{C}^{d}$ together with the standard Euclidean K\"ahler form $\omega_{\mathbb{C}^d}$.
If  $\omega$ is a Ricci-flat K\"ahler form on $B_1(0)$ such that
\begin{equation}\label{2assum}
A^{-1}\omega_{\mathbb{C}^d}\leq \omega \leq A\, \omega_{\mathbb{C}^d},
\end{equation}
then it holds for all $k \in \N_{\geq 1}$ that
\begin{equation}\label{toprove}
\|\omega\|_{C^k(B_{3/4}(0))}\leq C_k.
\end{equation}
\end{proposition}

We also need a uniform version of these estimates for mildly varying families of complex structures. For this and also for some later purposes, a version of the Newlander-Nirenberg theorem is required. We follow the approach of \cite[\S 5.7]{Hor}, which in turn originated from \cite[\S 12]{Koh}.

\begin{proposition}\label{p:nn}
For all $d,k \in \N_{\geq 1}$ and $0 < \alpha < 1$, there exist $\kappa_0 = \kappa_0(d,\alpha) > 0$ and $C_k = C_k(d,\alpha)$ such that the following holds. Let $J$ be a complex structure on the unit ball $B_1(0) \subset \C^d$ with
\begin{equation}\|J-J_{\C^d}\|_{C^{1,\alpha}(B_1(0))} \leq \kappa\end{equation} for some $\kappa \in (0,\kappa_0)$. Then there exist $J$-holomorphic coordinates $\hat{z}^1, \ldots, \hat{z}^d$ on $B_{3/4}(0)$ such that \begin{equation}
\|\hat{z}^j - z^j\|_{C^{2,\alpha}(B_{3/4}(0))} \leq C_1 \kappa\;\,{\rm\textit{for\;all}}\;\,j \in \{1,\ldots, d\},
\end{equation}
where $z^1, \ldots, z^d$ are the standard coordinates on $\C^d$. Moreover, if we assume in addition that
\begin{align}\label{seinfeld}
\|J-J_{\C^d}\|_{C^{k,\alpha}(B_1(0))} \leq A
\end{align} for some $k \geq 2$ and some constant $A$, then these coordinates may be chosen to also satisfy
\begin{align}
\|\hat{z}^j - z^j\|_{C^{k+1,\alpha}(B_{3/4}(0))} \leq C_kA\;\;{\rm\textit{for\;all}}\;\,j \in \{1,\ldots, d\}.
\end{align}
\end{proposition}

\begin{proof}
First of all, a simple local calculation \cite[p.443]{TWWY} shows that for any $C^2$ function $u$ we have
\begin{equation}\label{calculate}
i\partial^J\ov{\partial}{}^J u=(D^2u)^J + J\circledast DJ\circledast Du,
\end{equation}
where $(D^2u)^J$ is the $J$-invariant part of the coordinate Hessian of $u$ and $\circledast$ is a tensorial contraction. It follows that if $J$ and $J_{\C^d}$ are sufficiently $C^1$ close (depending at worst on $d$), then the function $|z|^2$ is still strictly $J$-psh on $B_1(0)$. Thus, by using H\"ormander's $L^2$ estimates (cf. \cite[proof of Thm 5.7.4]{Hor}, and in particular \cite[Lemma 5.7.1]{Hor}), we obtain functions $u^j \in W^{1,2}_{\rm loc}(B_1(0))$ such that
$$\|u^j\|_{W^{1,2}(B_{9/10}(0))} \leq C\kappa$$ and $\ov{\partial}{}^J(z^j + u^j) = 0$ holds in the weak sense. Now consider the operator
$$L(u) = \tr{\omega_{\C^d}}{(i\partial^J\ov{\partial}{}^J u)}.$$ This is elliptic because of the $C^0$ closeness of $J$ and $J_{\C^d}$. In fact, the second-order coefficients of $L$ are close to the identity in $C^{1,\alpha}$, and the first-order coefficients of $L$ are small in $C^\alpha$, by \eqref{calculate}. Moreover, by construction, the distribution $L(u^j) = - L(z^j)$ lies in $C^{\alpha}$ with $$\|L(u^j)\|_{C^\alpha(B_1(0))}\leq C\kappa.$$ The second-order part of $L$ can be written in divergence form without loss because its coefficients are close to the identity in $C^{1,\alpha}$. By \cite[Thm 5.5.3(b), p.153]{Mor} with $q = 2$ (see \cite[p.151]{Mor} for definitions), $u^j$ $\in$ $C$\begin{small}$^{1,\mu}_{\rm loc}$\end{small}$(B_1(0))$ for any $\mu \in (0,1)$, allowing us to absorb the first-order terms of $L(u^j)$ into the right-
hand side. By \cite[Thm 5.20]{GM2ndEd}, $u^j$ $\in$ $C$\begin{small}$^{2,\alpha}_{\rm loc}$\end{small}$(B_1(0))$. Thus, by any version of Schauder theory,
\[\begin{split}
\|u^j\|_{C^{2,\alpha}(B_{3/4}(0))}&\leq C (\|L(u^j)\|_{C^\alpha(B_{7/8}(0))}+\|u^j\|_{L^\infty(B_{7/8}(0))})\\
&\leq C\kappa + C\|u^j\|_{L^\infty(B_{7/8}(0))}.
\end{split}\]
Finally, by \cite[Thm 8.17]{GT} (which is implicit in the above references to \cite{GM2ndEd, Mor}),
$$\|u^j\|_{L^\infty(B_{7/8}(0))} \leq C\|u^j\|_{L^2(B_{9/10}(0))} \leq C\kappa.$$
In particular, the functions $\hat{z}^j = z^j + u^j$ are indeed coordinates because their gradients are pointwise linearly independent. This proves the first part of the statement. If we assume in addition that \eqref{seinfeld} holds, it follows that the same functions $u^j$ as above satisfy $\|L(u^j)\|_{C^{k-1,\alpha}(B_1(0))}\leq CA$ and that the coefficients of $L$ are bounded in $C$\begin{small}$^{k,\alpha}$\end{small}, so the claim again follows from Schauder theory.
\end{proof}

With these preparations, we can now easily prove the required uniform local estimate for Ricci-flat K\"ahler metrics with respect to a mildly varying family of complex structures.

\begin{proposition}\label{localest3}
For all $d,k \in \mathbb{N}_{\geq 1}$, $0 < \alpha < 1$, and $A \geq 1$, there exist constants $\kappa_0 = \kappa_0(d,\alpha) > 0$ and $C_k = C_k(d,\alpha,A)$ such that the following holds. Let $B_1(0)$ denote the unit ball in $\mathbb{C}^{d}$ together with the standard Euclidean metric $g_{\mathbb{C}^d}$. Let $J$ be a complex structure on $B_1(0)$ such that
\begin{align}\label{hmpf}
\|J-J_{\C^d}\|_{C^{1,\alpha}(B_1(0))} < \kappa_0\;\,{\rm \textit{and}}\;\,\|J-J_{\C^d}\|_{C^{k,\alpha}(B_1(0))} \leq A.
\end{align}
If $g$ is a Ricci-flat $J$-K\"ahler metric on $B_1(0)$ that satisfies
\begin{equation}\label{2assum2666}
A^{-1}g_{\mathbb{C}^d}\leq g\leq A\, g_{\mathbb{C}^d},
\end{equation}
then it holds for the same $k$ as in \eqref{hmpf} that
\begin{equation}\label{toprove2667}
\|g\|_{C^{k,\alpha}(B_{1/2}(0))}\leq C_k.
\end{equation}
\end{proposition}
\begin{proof}
Proposition \ref{p:nn} yields $J$-holomorphic coordinates on $B_{3/4}(0)$ close to the standard ones in $C^{2,\alpha}$ (as close as we like if we are willing to decrease $\kappa_0$), and differing from them by a bounded amount in $C^{k+1,\alpha}$. We can then simply apply Proposition \ref{localest2} in these new coordinates to get $C^\ell$ bounds for $g$ for all $\ell\geq 1$ and translate these bounds back to the standard coordinates to get \eqref{toprove2667}.
\end{proof}
\subsection{Liouville theorems}\label{liouville}

Recall the following well-known Liouville theorem (cf. \cite[Thm 2]{RS}).

\begin{theorem}\label{lio1}
Let $\omega$ be a Ricci-flat K\"ahler form on $\mathbb{C}^m$ such that
\begin{equation}\label{bounds}
C^{-1}\omega_{\mathbb{C}^m}\leq\omega\leq C\omega_{\mathbb{C}^m}
\end{equation}
for some constant $C \geq 1$, where $\omega_{\C^m}$ is the standard K\"ahler form on $\C^m$. Then $\omega$ is constant.
\end{theorem}
\begin{proof}
For convenience, here is a simple proof. Let $S=|\nabla^{\mathbb{C}^m} g|^2_g$, where $\nabla^{\mathbb{C}^m}$ is the covariant derivative of the Euclidean metric $\omega_{\mathbb{C}^m}$. Choose a cutoff function $\rho$ which is supported in $B_{2R}$, is identically $1$ on $B_R$, and has $|\nabla^{\mathbb{C}^m} \rho|^2_{g_{\mathbb{C}^m}}\leq C/R^2$ and $\Delta^{g_{\mathbb{C}^m}}(\rho^2)\geq -C/R^2$. Thanks to \eqref{bounds}, similar bounds hold if $g_{\mathbb{C}^m}$ is replaced by $g$. A well-known calculation using  Calabi's $C^3$ estimate (see e.g. \cite{HT}) gives
$$\Delta^g (\rho^2 S)\geq S\Delta^g(\rho^2)-8S|\nabla^g \rho|^2_g\geq -\frac{C}{R^2}S.$$
On the other hand, using Yau's $C^2$ estimate calculation and again \eqref{bounds},
$$\Delta^g(\tr{g_{\mathbb{C}^m}}{g})\geq C^{-1}S.$$
It follows from this that
$$\Delta^g \left(\rho^2 S+\frac{C}{R^2}\tr{g_{\mathbb{C}^m}}{g}\right)\geq 0,$$
so $\sup_{B_R}S\leq C/R^2$ by the maximum principle, and hence $S\equiv 0$ by letting $R \to \infty$.
\end{proof}
Instead of the Calabi-Yau $C^2$ and $C^3$ computations \cite{Cal,Ya}, one can also prove Theorem \ref{lio1} by using the theories of Evans-Krylov \cite{Tr} or Caffarelli \cite{Ca}. An elegant new proof of Theorem \ref{lio1} that does not rely on any of these methods was very recently given in \cite{LLZ}. In fact, combining this new approach with the blowup argument of \cite{CW} leads to a new way of proving the Evans-Krylov estimate for the complex Monge-Amp\`ere equation which is completely independent of \cite{Ca,Cal,Tr,Ya}.

In \cite[p.2937]{TZ} the following straightforward generalization of Theorem \ref{lio1} was proved.
\begin{theorem}\label{lio}
Let $(Y,\omega_Y)$ be a compact Ricci-flat K\"ahler manifold without boundary. Let $\omega_{\C^m}$ be the standard K\"ahler form on $\C^m$. Let $\omega=\omega_{\mathbb{C}^m}+\omega_Y + \ddbar u$ for some smooth function $u$ be a Ricci-flat K\"ahler form on $\C^m \times Y$ such that for some $C \geq 1$,
\begin{equation}\label{bounds2}
C^{-1}(\omega_{\C^m}+\omega_{Y})\leq\omega\leq C(\omega_{\C^m}+\omega_{Y}).
\end{equation}
If $\omega|_{\{z\}\times Y} = \omega_Y$ for all $z\in\mathbb{C}^m$, then $\omega$ is the product of a constant K\"ahler form on $\mathbb{C}^m$ with $\omega_Y$.
\end{theorem}
\begin{proof}
By assumption, $(\ddbar u)|_{\{z\}\times Y}=0$ for all $z\in \mathbb{C}^m$, so $u|_{\{z\}\times Y}$ is a constant (depending on $z$), so $u$ is the pullback of some smooth function on $\mathbb{C}^m$. Then $\hat{\omega}=\omega_{\mathbb{C}^m}+\ddbar u$ is a K\"ahler form on $\mathbb{C}^m$ and $\omega=\hat{\omega}+\omega_Y$ is a product K\"ahler form. Clearly $\hat{\omega}$ is Ricci-flat, and it satisfies
$C^{-1}\omega_{\mathbb{C}^m}\leq\hat{\omega}\leq C\omega_{\mathbb{C}^m}$ by \eqref{bounds2}. Thus,  $\hat\omega$ is constant by Theorem \ref{lio1}.
\end{proof}

More recently, the first-named author proved the following stronger result \cite{He2}. A simpler proof was given slightly later in \cite{LLZ}, using the same elegant idea that led to a new proof of Theorem \ref{lio1}.

\begin{theorem}\label{lioHein}
Let $(Y,\omega_Y)$ be a compact Ricci-flat K\"ahler manifold without boundary. Let $\omega_{\C^m}$ be the standard K\"ahler form on $\C^m$. Let $\omega$ be a Ricci-flat K\"ahler form on $\mathbb{C}^m\times Y$ that satisfies
\begin{equation}
C^{-1}(\omega_Y+\omega_{\mathbb{C}^m})\leq\omega\leq C(\omega_Y+\omega_{\mathbb{C}^m})
\end{equation}
for some $C\geq 1$. Choosing $\omega_Y$ suitably, we may assume that $\omega$ is $d$-cohomologous to $\omega_{\mathbb{C}^m} + \omega_Y$. Then, after changing $\omega$ by a biholomorphism, $\omega$ is the product of $\omega_Y$ and a constant K\"ahler form on $\mathbb{C}^m$. The biholomorphism is the identity if and only if $\omega$ is $i\partial\ov{\partial}$-cohomologous to $\omega_{\C^m} + \omega_Y$, and $\omega$ is parallel with respect to $\omega_{\C^m} + \omega_Y$ even before applying the biholomorphism.
\end{theorem}

It is instructive to see why the proof of Theorem \ref{lio1} breaks down in the situation of Theorem \ref{lioHein}. The fundamental reason is that in the Calabi-type calculation for the Laplacian of $|\nabla^{g_{\C^m}+g_{Y}}g|^2_g$,  there are some new terms coming from the Riemann curvature tensor of $g_Y$ if $Y$ is not flat that destroy the maximum principle argument above. As observed in \cite[Thm 1.1]{TZ} (cf. \cite[Prop 4.8]{TWY}), \emph{partial} bounds on $\nabla^{g_{\C^m} + g_Y}g$ can be obtained by stretching out the base directions. In particular, this method controls the ``all fibers'' component of this tensor, but is unable to prove a uniform bound for the ``all base'' component. Now one might suspect that there are ways of improving the Calabi $C^3$ quantity by taking the holomorphic product structure of $\C^m \times Y$ into account, but some bad terms remain. Specifically, if we let $\mathcal{P}$ denote the projection operator onto the base tangent directions, and let $\omega_P=\omega_{\C^m}+\omega_Y$, then we might for instance consider the quantity $S=|\Psi|^2_{g},$ where $\Psi$ is the tensor obtained by composing $\nabla^{g_P}g$ with $\mathcal{P}$ in all three arguments. Bounding $S$ would indeed bound the ``all base'' component of $\nabla^{g_P}g$. However, $\Delta^gS$ still contains some bad terms due to the fact that $\nabla^{g}\mathcal{P}$ need not vanish.

Nevertheless, in the setting of compact Calabi-Yau manifolds of Corollaries \ref{main} and \ref{main2}, it follows from the main theorem of \cite{TWY} that after applying a base stretching diffeomorphism, the metrics $e^t\omega^\bullet_t$ converge smoothly to a Ricci-flat metric on $\C^m\times Y$ that {does} satisfy the hypotheses of Theorem \ref{lio}, and therefore must split as a product (this was observed in \cite{TZ}). In the local setting of Theorems \ref{mainlocal} and \ref{thm51}, the global techniques of \cite{TWY} do not apply, so Theorem \ref{lioHein} must be used instead to recover the conclusions of \cite{TZ}. The basic idea of the present paper is that by pushing this approach to its limit, \emph{full} higher-order estimates for collapsing Calabi-Yau metrics can be proved without using any Calabi- type calculations whatsoever (except for the standard local ones, or their counterparts in \cite{Ca, Tr}, that lead to Proposition \ref{localest2}, although again even these can be avoided  by using \cite{CW, LLZ}), hence in particular without using any of the results of \cite{TWY} (which only apply in the compact setting anyway).

\section{Schauder estimates on cylinders}\label{linear}

\subsection{Technical preliminaries}

\begin{definition}\label{d:holder} Let $(X,g)$ be a complete Riemannian manifold. Let $E \to X$ be a vector bundle on $X$ with a fiber metric $h$ and an $h$-preserving connection $\nabla$. If $x,x' \in X$ and if there is a unique minimal $g$-geodesic $\gamma$ joining $x$ to $x'$, then we let $\mathbf{P}^{g}_{xx'}$ denote $\nabla$-parallel transport on $E$ along $\gamma$. If there is no unique minimal $g$-geodesic $\gamma$ from $x$ to $x'$, then $\mathbf{P}^g_{xx'}$ is undefined. Let $B^g(p,R)$ be the $g$-geodesic ball of radius $R > 0$ centered at $p \in X$. Then we define
\begin{align}\label{e:holderdef}
[\sigma]_{C^{\alpha}(B^g(p,R))} = \sup \left\{\frac{|\sigma(x) -  \mathbf{P}^g_{x'x}(\sigma(x'))|_{h(x)}}{d^g(x,x')^\alpha} : x, x' \in B^g(p,R), \;x\neq x',\;\mathbf{P}^g_{xx'}\;\text{is defined}\right\}
\end{align}
for all sections $\sigma \in C^{\alpha}_{\rm loc}(B^g(x,2R),E)$. Here we implicitly used the simple fact that if $\gamma$ is a minimal $g$-geodesic connecting two points $x,x'\in B^g(p,R)$, then $\gamma$ is contained in $B^g(p,2R)$.
\end{definition}

Our notation \eqref{e:holderdef} suppresses $h$ and $\nabla$, but in practice $(E,h,\nabla)$ will be derived from $(TX,g,\nabla^g)$ in some natural way, so that the $C^{\alpha}$ seminorm \eqref{e:holderdef} is actually completely determined by $g$. In a small number of special cases we will slightly modify \eqref{e:holderdef} by replacing $B^g(p,R)$ by some open set $U$ which is not a $g$-ball but is \emph{$g$-geodesically convex}; in these cases, it is enough for $\sigma$ to be defined and $C^\alpha_{\rm loc}$ on $U$, and we will write $[\sigma]_{C^\alpha(U,g)}$ to indicate the dependence of the seminorm on $g$.

\begin{rk}
We will use several times the simple observation that if $[\sigma]_{C^{\alpha}(B)}=0$ for some geodesic ball $B \subset X$, then $\sigma$ is parallel on $B$, so in particular  $\sigma$ is smooth on $B$ and
$\nabla \sigma=0$ on $B$.
\end{rk}

The following lemma is one of the cornerstones of the whole paper.

\begin{lemma}\label{l:braindead}
Let $Y$ be a compact Riemannian manifold without boundary. Let $E$ be a metric vector bundle over $Y$ with a metric connection $\nabla$.
Then for all $k \in \N_{\geq 1}$, $\alpha \in (0,1)$ there exists a constant $C_k = C_k(Y,E,\alpha)$ such that for all $\sigma\in C^{k,\alpha}(Y,E)$,
\begin{align}\label{e:idioticestimate}
\|\nabla\sigma\|_{L^\infty(Y)} \leq C_k[\nabla^k\sigma]_{C^{\alpha}(Y)}.
\end{align}
\end{lemma}

\begin{proof} It is enough to prove this for $k = 1$. Indeed, if this is known for $k = 1$, then for all $k \geq 2$ and $j \in \{1,\ldots,k\}$, the $k = 1$ case applied to the section  $\nabla^{j-1}\sigma$ of $(T^*Y)^{\otimes(j-1)}\otimes E$ tells us that
$$\|\nabla^j\sigma\|_{L^\infty(Y)} \leq C_j [\nabla^j\sigma]_{C^{\alpha}(Y)}.$$
Moreover, for all $j \in \{1,\ldots,k-1\}$ it is easy to see that
$$[\nabla^j\sigma]_{C^{\alpha}(Y)}\leq C_j \|\nabla^{j+1}\sigma\|_{L^\infty(Y)} .$$
The claim then follows by iteration, and it remains to prove the base case $k = 1$.

To this end, define $\mathcal{P} = \{\sigma \in C^{1,\alpha}(Y,E): \nabla\sigma = 0\}$. Then $\dim\, \mathcal{P} < \infty$ and $\mathcal{P} \subset C^\infty(Y,E)$. Let $\pi$ be the $L^2$-orthogonal projection onto $\mathcal{P}$. Suppose the lemma fails for $k = 1$. Then there exists a sequence $\sigma_i \in C^{1,\alpha}(Y,E)$ with $[\nabla\sigma_i]_{C^{\alpha}(Y)} < \frac{1}{i}\|\nabla\sigma_i\|_{L^\infty(Y)}$. Replacing $\sigma_i$ by $\sigma_i - \pi(\sigma_i)$, we may assume that $\sigma_i \in {\rm ker}\,\pi$. Dividing $\sigma_i$ by $\|\nabla \sigma_i\|_{L^\infty(Y)} > 0$, we may further assume that $\|\nabla \sigma_i\|_{L^\infty(Y)} = 1$.\medskip\

\noindent \emph{Claim.} There exists a constant $C$ such that $\|\sigma_i\|_{L^\infty(Y)} \leq C$ for all $i$.   \medskip\

\noindent \emph{Proof of the Claim.} Suppose that this is false. Then we may assume that $\|\sigma_i\|_{L^\infty(Y)} > i$. Dividing $\sigma_i$ by $\|\sigma_i\|_{L^\infty(Y)}$, we may further assume that $\|\sigma_i\|_{L^\infty(Y)} = 1$, $\|\nabla\sigma_i\|_{L^\infty(Y)} < \frac{1}{i}$, and $[\nabla\sigma_i]_{C^{\alpha}(Y)} < \frac{1}{i^2}$. By passing to a subsequence, we may then also assume that $\sigma_i$ converges to some $\sigma \in C^{1,\alpha}(Y,E)$ in the $C^{1,\beta}$ sense for all $\beta < \alpha$. By construction, this limit satisfies $\|\sigma\|_{L^\infty(Y)} = 1$, $\nabla \sigma = 0$, and $\sigma \in {\rm ker}\,\pi$. By the second and third of these properties, $\sigma\in \mathcal{P} \cap {\rm ker}\,\pi = \{0\}$, which contradicts the first. \hfill $\Box$     \medskip\

Given the claim, and passing to a subsequence if needed, we may now assume that $\sigma_i$ converges to some $\sigma \in C^{1,\alpha}(Y,E)$ in the $C^{1,\beta}$ topology for every $\beta < \alpha$. By construction, this limit $\sigma$ satisfies the following properties: $[\nabla \sigma]_{C^{\alpha}(Y)} = 0$, $\|\nabla\sigma\|_{L^\infty(Y)} = 1$, and $\sigma \in {\rm ker}\,\pi$. The first property implies that $\sigma$ is smooth with $\nabla\nabla\sigma = 0$. Thus, relying crucially on the fact that $\partial Y = \emptyset$,
$$\int_Y |\nabla\sigma|^2 = \int_Y \langle \sigma,\nabla^*\nabla\sigma\rangle  = -\int_Y \langle\sigma,{\rm tr}(\nabla\nabla\sigma)\rangle  = 0.$$
This implies that $\nabla\sigma = 0$, which contradicts the second property of $\sigma$.\end{proof}

Next, we have the following iteration lemma, which will also be used many times over.

\begin{lemma}\label{l:iterate}
For all $0 \leq \epsilon < 1$, $\beta_1 < \ldots < \beta_k$, and $\gamma_1 < \ldots < \gamma_m$, there exists a constant $C$ such that the following holds. Let $f_1, \ldots, f_k: [0,T] \to \R$ be bounded nonnegative functions such that
\begin{align}\label{e:iterate_hyp}
\sum_{j=1}^k (R-\rho)^{\beta_j}f_j(\rho) \leq \epsilon \sum_{j=1}^k (R-\rho)^{\beta_j}f_j(R) + \sum_{\ell = 1}^{m} A_\ell (R-\rho)^{\gamma_{\ell}}
\end{align}
for some $A_1, \ldots, A_m\geq 0$ and for all $0 \leq \rho < R \leq T$. Then for all $0 \leq \rho < R \leq T$,
\begin{align}\label{e:iterate_concl}
\sum_{j=1}^k(R-\rho)^{\beta_j}f_j(\rho) \leq C \sum_{\ell=1}^m A_\ell (R-\rho)^{\gamma_\ell}.
\end{align}
\end{lemma}

\begin{proof}
This is a minor extension of \cite[Lemma 8.18]{GM2ndEd}. Observe that by multiplying \eqref{e:iterate_hyp} and \eqref{e:iterate_concl} by $(R-\rho)^{-\beta_1}$, we may assume without loss of generality that $\beta_1 = 0$.

Given $\tau \in (0,1)$, define $\rho_0 = \rho$ and $\rho_{i+1} = \rho_i + \tau^i (1-\tau)(R - \rho)$ for all $i \in \N$. Also define
\begin{align*}
P_i = \sum_{j=1}^k [\tau^i(1-\tau)(R-\rho)]^{\beta_j} f_j(\rho_i), \;\, Q_i = \sum_{\ell = 1}^m A_\ell [\tau^i(1-\tau)(R-\rho)]^{\gamma_\ell}.
\end{align*}
Applying \eqref{e:iterate_hyp} with $\rho_i, \rho_{i+1}$ in place of $\rho, R$, we get
\begin{align*}
P_i \leq \epsilon\tau^{-\beta_k} P_{i+1} + Q_i \leq \delta P_{i+1} + Q_i
\end{align*}
for any fixed $\delta \in (\epsilon,1)$, provided that $\tau \geq (\epsilon/\delta)^{1/\beta_k}$ if $\beta_k>0$. Thus, by iteration,
\begin{align*}(1-\tau)^{\beta_k}\sum_{j=1}^k (R-\rho)^{\beta_j} f_j(\rho) \leq P_0 \leq \liminf_{i\to\infty} (\delta^{i+1} P_{i+1} + Q_0 + \delta Q_1 + \cdots + \delta^{i}Q_{i}).\end{align*}
The sequence $P_{i+1}$ is bounded because each $f_j$ is bounded and $\beta_j \geq 0$. On the other hand,
$$\delta^i Q_i  \leq (\delta \tau^{\gamma_1})^i (1-\tau)^{\gamma_1}\sum_{\ell = 1}^m A_\ell (R-\rho)^{\gamma_\ell},$$
which is summable provided that $\tau >\delta^{1/|\gamma_1|}$ if $\gamma_1 < 0$.\end{proof}

The following lemma provides a precise interpolation inequality on Riemannian cylinders.

\begin{lemma}\label{l:higher-interpol}
Let $(Y,g_Y)$ be a compact Riemannian manifold without boundary. Let $E \to Y$ be a metric vector bundle with a metric connection $\nabla$. Extend $E$ trivially to $\R^d \times Y$ and extend $\nabla$ by adding $\nabla^{\R^d}$. Let $g_P = g_{\R^d} + g_Y$ on $\R^d \times Y$. Then for all $k \in \N_{\geq 1}$, $\alpha \in (0,1)$ there is a $C_k = C_k(\alpha)$ such that
\begin{align}\label{hohoho}
\sum_{j=1}^k (R-\rho)^j \|\nabla^j\sigma\|_{L^\infty(B^{g_P}(p,\rho))} \leq C_k ((R-\rho)^{k+\alpha}[\nabla^k\sigma]_{C^{\alpha}(B^{g_P}(p,R))} + \|\sigma\|_{L^\infty(B^{g_P}(p,R))})
\end{align}
for all $p \in \R^d \times Y$, $0 < \rho < R$, and $\sigma \in C^{k,\alpha}_{\rm loc}(B^{g_P}(p,2R),E)$.
\end{lemma}

\begin{proof} Aiming to apply Lemma \ref{l:iterate}, for $j \in \{1,\ldots,k\}$ define $\beta_j = j$ and
$f_j(\rho) =\|\nabla^j\sigma\|_{L^\infty(B^{g_P}(p,\rho))}$.
In order to prove an inequality of the form \eqref{e:iterate_hyp}, consider the following three cases.\medskip\

\noindent \emph{Case 1}: $R - \rho < {\rm inj}(Y)$. Fix any $j \in \{1,\ldots,k\}$ and write $\tau = \nabla^{j-1}\sigma$. Fix any $x \in B^{g_P}(p,\rho)$. Let $v$ be a unit tangent vector at $x$ maximizing the quantity $|(\nabla_w\tau)(x)|$ among all unit tangent vectors $w$ at $x$. Let $\gamma(t) = \exp^{g_P}_{x}(tv)$. This curve is the unique length minimizer between any two of its points as long as $|t| \leq R-\rho$. Let $x' = \gamma(\ell)$ for $\ell = \epsilon(R-\rho)$ and $\epsilon \in (0,1)$. Then
\begin{align*}
\tau(x) - \mathbf{P}_{x'x}^{g_P}(\tau(x')) = \int_0^\ell \frac{d}{dt} \mathbf{P}^{g_P}_{\gamma(\ell-t)x}(\tau(\gamma(\ell-t))) \, dt
= -\int_0^\ell \mathbf{P}^{g_P}_{\gamma(\ell-t)x}[(\nabla_{\dot{\gamma}(\ell-t)}\tau)(\gamma(\ell-t))]\,dt.
\end{align*}
We can rewrite the last integrand as $(\nabla_v \tau)(x) + \psi(t)$, where for all $t \in [0,\ell]$,
$$|\psi(t)| \leq
\begin{cases}
[\nabla\tau]_{C^{\alpha}(B^{g_P}(x,\ell-t))} (\ell-t)^\alpha &{\rm for\; all}\;j,\\
\|\nabla^2\tau\|_{L^\infty(B^{g_P}(x,\ell-t))}(\ell-t) &{\rm for \;all}\;j < k.
\end{cases}$$
Here we have used the definition of the $C^{\alpha}$ seminorm and the fact that $\nabla_{\dot\gamma}\dot\gamma = 0$. This leads to
$$\ell|\nabla_v\tau(x)| \leq |\tau(x)|+|\mathbf{P}_{x'x}^{g_P}(\tau(x'))|+
\begin{cases}
C\ell^{1+\alpha} [\nabla^{j}\sigma]_{C^{\alpha}(B^{g_P}(p,\rho+\ell))} & {\rm for \; all}\;j,\\
C\ell^2\|\nabla^2\tau\|_{L^\infty(B^{g_P}(p,\rho+\ell))}&{\rm for \;all}\;j < k.
\end{cases}$$
Taking the sup over all $x \in B^{g_P}(p,\rho)$, we deduce that
\begin{align}\label{e:gazillionth_random_inequality}
\ell f_j(\rho) &\leq Cf_{j-1}(\rho + \ell) + \begin{cases}
C\ell^{1+\alpha} [\nabla^{j}\sigma]_{C^{\alpha}(B^{g_P}(p,\rho+\ell))} & {\rm for \; all}\;j,\\
C\ell^2 f_{j+1}(\rho+\ell) &{\rm for\;all}\;j < k.
\end{cases}
\end{align}
Thus, working backwards from $j = k$ to $j = 1$, decreasing and renaming $\epsilon$  in each step,
\begin{align}\label{e:gazillion-and-one}
\sum_{j=1}^k (R-\rho)^{j }f_j(\rho) \leq  \epsilon \sum_{j=1}^k (R-\rho)^j f_j(R)  + \epsilon (R-\rho)^{k+\alpha}[\nabla^k\sigma]_{C^{\alpha}(B^{g_P}(p,R))} + C_\epsilon \|\sigma\|_{L^\infty(B^{g_P}(p,R))}.
\end{align}
This is the desired inequality of type \eqref{e:iterate_hyp}. Here $\epsilon \in (0,1)$ is arbitrary.
\medskip\

\noindent \emph{Case 2}: $R-\rho \in [{\rm inj}(Y), 10\,{\rm diam}(Y)]$. This case can be  reduced to Case 1. Let $R' = \rho + \frac{1}{2}{\rm inj}(Y)$ and apply Case 1 to the pair of radii $(\rho, R')$ instead of $(\rho,R)$. In \eqref{hohoho} with $R$ replaced with $R'$, notice that trivially $B^{g_P}(x,R') \subset B^{g_P}(x,R)$ and $(R'-\rho)^{k+\alpha} \leq (R-\rho)^{k+\alpha}$, so in order to obtain \eqref{hohoho} as written we only need to observe that $(R'-\rho)^j \geq ({\rm inj}(Y)/20\,{\rm diam}(Y))^j (R - \rho)^j$ for $j = 1, \ldots,k$.     \medskip\

\noindent \emph{Case 3}: $R-\rho > 10 \, {\rm diam}(Y)$. Using the same idea as in Case 1, we can prove that \eqref{e:gazillionth_random_inequality} still holds with $\ell f_j(\rho)$ replaced by $\ell \|\nabla_{\mathbf{b}}\nabla^{j-1}\sigma\|_{L^\infty(B(p,\rho))}$ on the left-hand side. (Here and below, a subscript ${\mathbf{b}}$ and $\mathbf{f}$ denotes covariant derivatives in the base and fiber directions, respectively.) This is because we can take $\gamma$ to be a horizontal line in this case, which is then the unique length minimizer between any two of its points. On the other hand, by Lemma \ref{l:braindead}, for all $x \in B^{g_P}(p,\rho)$,
$$|(\nabla_{\mathbf{f}}\nabla^{j-1}\sigma)(x)| \leq C[\nabla_{\mathbf{f}}^{k-j+1}\nabla^{j-1}\sigma]_{C^{\alpha}(Y_x,g_Y)} \leq C[\nabla^k \sigma]_{C^{\alpha}(B^{g_P}(p,R))},$$
where $Y_x$ denotes the fiber through $x$, and $Y_x \subset B^{g_P}(p,R)$ because $R - \rho > {\rm diam}(Y)$. Proceeding as in Case 1 (working backwards from $j = k$ and making use of the safety factor $10$), we get
\begin{align*}\
\sum_{j=1}^k (R-\rho)^{j }f_j(\rho) \leq  \epsilon \sum_{j=1}^k (R-\rho)^j f_j(R)  + C_\epsilon (R-\rho)^{k+\alpha}[\nabla^k\sigma]_{C^{\alpha}(B^{g_P}(p,R))} + C_\epsilon \|\sigma\|_{L^\infty(B^{g_P}(p,R))}.
\end{align*}
Notice that unlike in \eqref{e:gazillion-and-one}, the constant in front of the $[\nabla^k\sigma]_{C^{\alpha}}$ term is $C_\epsilon$ rather than $\epsilon$. \medskip\

Lemma \ref{l:higher-interpol} now follows from Lemma \ref{l:iterate}.
\end{proof}

Our final lemma allows us to compare H\"older norms with respect to different metrics. One key point here is that $C^\alpha$ {seminorms} are rarely ever uniformly comparable, but the full $C^\alpha$ norms often are.

\begin{lemma}\label{tiammazzo1}
For all $d,q \in \N$, $A \geq 1$, $\alpha \in (0,1)$ there exists a constant $C = C(d,q,A,\alpha) \geq 1$ such that the following holds. Let $g$ be a Riemannian metric on the unit ball $B^{g_{\R^d}}(0,1) \subset \R^d$ such that
\begin{align}
A^{-1}g_{\R^d} \leq g \leq Ag_{\R^d},\label{totschlag1}\\
|\nabla^{g_{\R^d}} g|_{g_{\R^d}} + |\nabla^{2,g_{\R^d}} g|_{g_{\R^d}} \leq A.\label{totschlag2}
\end{align}
Then for all tensors $T$ of rank $q$ defined on $B^{g_{\R^d}}(0,1)$ and for all $x,x' \in B^g(0,C^{-1})$,
\begin{align}\label{totschlag3}
\frac{|T(x) - \mathbf{P}^g_{x'x}(T(x'))|_{g(x)}}{d^g(x,x')^\alpha} \leq C\frac{|T(x) - \mathbf{P}^{g_{\R^d}}_{x'x}(T(x'))|_{g_{\R^d}(x)}}{d^{g_{\R^d}}(x,x')^\alpha} + C\|T\|_{L^\infty(B^{g_{\R^d}}(0,1))}.
\end{align}
Moreover, for all $x,x' \in B^{g_{\R^d}}(0,C^{-2})$, a similar estimate holds with $g$ and $g_{\R^d}$ interchanged and with the $L^\infty$ norm of $T$ over $B^g(0,C^{-1})$ on the right-hand side.  In particular, it follows that
\begin{align}\label{totschlag4}
\|T\|_{C^{\alpha}(B^{g_{\R^d}}(0,C^{-2}))} \leq C\|T\|_{C^\alpha(B^{g}(0,C^{-1}))}\leq C^2\|T\|_{C^\alpha(B^{g_{\R^d}}(0,1))}.
\end{align}
\end{lemma}

\begin{proof}
Choose $C$ so large that $B^{g_{\R^d}}(0,C^{-2}) \subset B^g(0,C^{-1}) \subset B^{g_{\R^d}}(0,1)$ and $B^{g}(0, C^{-1})$ is convex with respect to $g$. The latter is possible thanks to \cite[Thm 5.14]{ChEb} combined with \cite[Thm 4.3]{ChGrTa}, noting that all the relevant quantities of \cite{ChEb, ChGrTa} are suitably bounded thanks to \eqref{totschlag1}, \eqref{totschlag2}. Then pick any two distinct points $x,x' \in B^g(0,C^{-1})$. By our choice of $C$ there exists a unique minimal $g$-geodesic $\gamma$ from $x'$ to $x$, and $\gamma$ is contained in $B^g(0,C^{-1})$. In order to compare the $g$-H\"older difference quotient of $T$ at $x,x'$ to the standard Euclidean one, it suffices to estimate the quantity
\begin{align}\label{totschlag5}
\frac{|\mathbf{P}^g_{\gamma}(T(x')) - \mathbf{P}^{g_{\R^d}}_{\gamma}(T(x'))|_{g_{\R^d}(x)}}{d^{g_{\R^d}}(x,x')^\alpha}.
\end{align}
Expressing $T$ in terms of the standard coordinates on $\R^d$, and writing $\Gamma$ for the Christoffel symbols of $g$ with respect to these coordinates, we can obviously bound the numerator of \eqref{totschlag5}  by
$$\biggl|\int_0^{d^g(x,x')} \frac{d}{dt} \mathbf{P}^{g_{\R^d}}_{\gamma(t)x}(\mathbf{P}^g_{x'\gamma(t)}(T(x')))\, dt \biggr|_{g_{\R^d}(x)} \leq C\int_0^{d^g(x,x')} |\Gamma(\gamma(t))| |\mathbf{P}^{g}_{x'\gamma(t)}(T(x'))|_{g_{\R^d}(\gamma(t))}\, dt.$$
Using \eqref{totschlag1}, \eqref{totschlag2}, and the fact that $\mathbf{P}^g$ is a $g$-isometry, it readily follows that \eqref{totschlag5} is bounded by $C$ times the $L^\infty$ norm of $T$ over $B^{g_{\R^d}}(0,1)$. This proves \eqref{totschlag3}. The proof of the analogous inequality with $g$ and $g_{\R^d}$ interchanged is verbatim the same (note that the restriction that $x,x' \in B^{g_{\R^d}}(0,C^{-2})$ is artificial and simply serves to make the statement of \eqref{totschlag4} more symmetric).
 \end{proof}

\begin{rk}\label{tiammazzo2}
We will also often use the following remark, which is related to Lemma \ref{tiammazzo1} but is easier and slightly more standard. If a sequence of Riemannian metrics $g_i$ converges to a limiting metric $g_\infty$ in $C^1$, if $T_i$ are tensors converging uniformly to $T_\infty$, if $x_i,x'_i$ are points converging to $x_\infty,x'_\infty$, and if $x_i,x'_i$ can be joined by a unique minimal $g_i$-geodesic $\gamma_i$, then $\gamma_i$ converges to a minimal $g_\infty$-geodesic $\gamma_\infty$ (this is clear) and \begin{small}$\mathbf{P}^{g_i}_{x'_ix_i}$\end{small}$T_i$ converges to \begin{small}$\mathbf{P}^{g_\infty}_{\gamma_\infty}$\end{small}$T_\infty$ (this is true because we have a sequence of ODEs on a convergent sequence of time intervals, with convergent initial values and with coefficient functions that converge uniformly). If in addition $x_\infty, x_\infty'$ can be joined by a \emph{unique} minimal geodesic with respect to $g_\infty$, then of course $\gamma_\infty$ is that geodesic and $\mathbf{P}^{g_\infty}_{\gamma_\infty}T_\infty =  \mathbf{P}^{g_\infty}_{x'_\infty x_\infty}T_\infty$.
\end{rk}

\subsection{An abstract Schauder estimate on cylinders}\label{s:abstract}

Throughout this section, $(Y,g_Y)$ will denote a compact Riemannian manifold of dimension $e$ without boundary. We consider the cylinder $\R^d \times Y$ endowed with the product metric $g_P = g_{\R^d} + g_Y$. (Formally $d,e = 0$ are possible.) As usual, $r$ denotes a Euclidean radius function on $\R^d$ or $\R^{d+e}$. Given a metric $g$, we define $L^g = d + \delta^g$ acting on $q$-forms of some fixed degree $q$. We will use the intrinsic Definition \ref{d:holder} of the H\"older seminorms.

With these conventions understood, our main result may be stated as follows.

\begin{theorem}\label{t:schauder}
Let $k\in\N_{\geq 1}$ and $0<\alpha<1$ be given. Let $\mathcal{S}$ be a presheaf of vector spaces of $q$-forms of class $C$\begin{small}$^{k,\alpha}_{\rm loc}$\end{small} on $\R^d \times Y$ such that the following two properties hold.

{\rm (1)} If $U_i$ is an exhaustion of $\mathbb{R}^d \times Y$ by open sets and if $\eta_i \in \mathcal{S}(U_i)$ converge to $\eta_\infty \in C^\infty_{\rm loc}(\mathbb{R}^d \times Y)$ in the $C$\begin{small}$^{k,\beta}_{\rm loc}$\end{small} topology for some $\beta<\alpha$, then $\eta_\infty \in \mathcal{S}(\R^d\times Y)$.

{\rm (2)} If $\eta\in\mathcal{S}(\mathbb{R}^d \times Y)$, $|\nabla^{k,g_P}\eta|_{g_P} = O(r^\alpha)$, and $\nabla^{k-1,g_P} L^{g_P}\eta$ is $g_P$-parallel, then $\nabla^{k+1,g_P}\eta = 0$.

\noindent Then there exists a $C>0$ such that for all $p \in \R^d \times Y$, $0 < \rho < R$, and $\eta \in \mathcal{S}(B^{g_P}(p,2R))$,
\begin{align}
[\nabla^{k,g_P}\eta]_{C^{\alpha}(B^{g_P}(p,\rho))} \leq C([\nabla^{k-1,g_P}L^{g_P}\eta]_{C^{\alpha}(B^{g_P}(p,R))} +(R-\rho)^{-k-\alpha}\|\eta\|_{L^\infty(B^{g_P}(p,R))} ).
\end{align}
\end{theorem}\smallskip

For clarity we isolate the main step of the proof of Theorem \ref{t:schauder} as a separate proposition.

\begin{proposition}\label{p:schauder}
Under the assumptions of Theorem \ref{t:schauder}, for all $\epsilon>0$ there exist $\delta_0,C>0$ such that for all $p \in \R^d \times Y$, $R > 0$, $\eta \in \mathcal{S}(B^{g_P}(p,2R))$, and $0 < \delta \leq \delta_0$,
\begin{equation}
\begin{split}
[\nabla^{k,g_P}\eta]_{C^{\alpha}(B^{g_P}(p,\delta R))} \leq \;&{\epsilon[\nabla^{k,g_P}\eta]_{C^{\alpha}(B^{g_P}(p,R))}} + C [\nabla^{k-1,g_P}L^{g_P} \eta]_{C^{\alpha}(B^{g_P}(p,R))}\\ &{+}\sum_{j=0}^k CR^{-k+j-\alpha}\|\nabla^{j,g_P}\eta\|_{L^\infty(B^{g_P}(p,\delta R))}.
\end{split}\end{equation}
\end{proposition}\smallskip

Theorem \ref{t:schauder} follows quickly from this.

\begin{proof}[Proof of Theorem \ref{t:schauder}]
Fix an $\epsilon \in (0, \frac{1}{10}]$ and obtain $\delta = \min\{\frac{1}{10}, \delta_0\}$ and $C$ from Proposition \ref{p:schauder}. Let $x,x' \in B^{g_P}(p,\rho)$ have  a unique $g_P$-minimal geodesic connecting them. If $d^{g_P}(x,x')$ $<$ $\delta(R-\rho)$, then Proposition \ref{p:schauder} (applied to $B^{g_P}(x,R-\rho)$ instead of $B^{g_P}(p,R)$) allows us to estimate the $C^{\alpha}$ difference quotient of $\nabla^{k,g_P}\eta$ at $x,x'$. If $d^{g_P}(x,x') \geq \delta(R-\rho)$, we estimate this difference quotient trivially, at the expense of an additional term $2(\delta(R-\rho))^{-\alpha}\|\nabla^{k,g_P}\eta\|_{L^\infty(B^{g_P}(p,\rho))}$ on the right-hand side. Thus, for all $0 < \epsilon \leq \frac{1}{10}$ there exist $\delta,C > 0$ with $\delta \leq \frac{1}{10}$ such that
\[\begin{split}
[\nabla^{k,g_P}\eta]_{C^{\alpha}(B^{g_P}(p,\rho))} \leq \;&\epsilon[\nabla^{k,g_P}\eta]_{C^{\alpha}(B^{g_P}(p,R))}+C[\nabla^{k-1,g_P}L^{g_P}\eta]_{C^{\alpha}(B^{g_P}(p,R))} \\
&{+}\sum_{j=0}^k C(R-\rho)^{-k+j-\alpha}\|\nabla^{j,g_P} \eta\|_{L^\infty(B^{g_P}(p,\rho+\delta(R-\rho)))}
\end{split}\]
for all $p\in\R^d\times Y$, $0 < \rho < R$, and $\eta \in\mathcal{S}(B^{g_P}(p,2R))$. Lemma \ref{l:higher-interpol} (with $\rho,R$ replaced by $\rho + \delta(R-\rho)$, $\rho + (\delta + \delta')(R-\rho)$ for some $\delta' \in (0,\frac{1}{10}]$ small enough so that $C(\delta')^\alpha \leq \epsilon$) lets us remove the $\|\nabla^{j,g_P}\eta\|_{L^\infty}$ terms with $j > 0$ from the right-hand side, and then Theorem \ref{t:schauder} follows from Lemma \ref{l:iterate}.
\end{proof}

We will now prove Proposition \ref{p:schauder}, thereby completing the proof of Theorem \ref{t:schauder}. We write $\Lambda^qX$ for the bundle of $q$-forms on a manifold $X$ and $\mathcal{A}^qX$ for the space of $C^{k,\alpha}_{\rm loc}$ sections of $\Lambda^q X$.

\begin{proof}[Proof of Proposition \ref{p:schauder}] Suppose that the statement is false. Then there exists an $\epsilon>0$ such that there exist sequences $p_i \in \R^d \times Y$, $R_i>0$, $\eta_i \in \mathcal{S}(B^{g_{P}}(p_i, 2R_i))$, and $0 < \delta_i \leq \frac{1}{i}$ such that
\begin{align}
\label{e:start1} 1 = [\nabla^{k,g_P}\eta_i]_{C^{\alpha}(B^{g_{P}}(p_i,\delta_iR_i))} >\;&\epsilon [\nabla^{k,g_P}\eta_i]_{C^{\alpha}(B^{g_{P}}(p_i,R_i))} + i[\nabla^{k-1,g_P}L^{g_P} \eta_i]_{C^{\alpha}(B^{g_{P}}(p_i,R_i))}\\
\label{e:start2}
&{+}\sum_{j=1}^k i R_i^{-k+j-\alpha} \|\nabla^{j,g_P}\eta_i\|_{L^\infty(B^{g_P}(p_i, \delta_iR_i))}.
\end{align}
(We can always make the left-hand side equal to $1$ by dividing $\eta_i$ by $[\nabla^{k,g_P}\eta_i]_{C^{\alpha}(B^{g_{P}}(p_i,\delta_i R_i))} > 0$ if necessary.) Select $x_i, x_i'$ in the closure of $B^{g_{P}}(p_i,\delta_iR_i)$ such that the supremum in the definition of the seminorm on the left-hand side of \eqref{e:start1} is attained at $x_i, x_i'$.

It turns out to be very useful for the sake of deriving a contradiction to pass from $\eta_i$ to a certain modified sequence $\eta_i'$, which we will now define. Let $\mathcal{L} = \{\eta\in\mathcal{S}(\R^d\times Y): \nabla^{k+1,g_P}\eta = 0\}$. Given a form $\eta$ on a neighborhood of ${x}_i$, define its $k$-jet at $x_i$ by $J^k_i\eta = (\eta(x_i), (\nabla^{g_P}\eta)(x_i), \ldots, (\nabla^{k,g_P}\eta)(x_i))$, and define its \emph{partial} $k$-jet $\mathcal{L}J^k_i\eta$ to be the $g_P(x_i)$-orthogonal projection of $J^k_i\eta$ onto the space $J^k_i(\mathcal{L})$. As $J^k_i$ is injective on $\mathcal{L}$, there exists a unique $\eta_i^\sharp \in \mathcal{L}$ with $J^k_i\eta_i^\sharp = \mathcal{L}J^k_i\eta_i^\sharp = \mathcal{L}J^k_i\eta_i$, and we can use this to define $\eta_i' = \eta_i - \eta_i^\sharp \in \mathcal{S}(B^{g_P}(p,2R_i))$. The idea of defining and using this partial $k$-jet comes from \cite{HN}. The following claim states the key properties of $\eta_i'$. \medskip\

\noindent \emph{Claim 1}. There exists a constant $C$ such that after passing to a subsequence,
\begin{align}
1 = [\nabla^{k,g_P}\eta_i']_{C^{\alpha}(B^{g_P}(p_i,\delta_i R_i))} > \epsilon [\nabla^{k,g_P}{\eta}_i']_{C^{\alpha}(B^{g_P}({p}_i,R_i))} + i[\nabla^{k-1,g_P}L^{{g}_{P}} {\eta}_i']_{C^{\alpha}(B^{{g}_{P}}({p}_i,R_i))},\label{e:fucking_bounds_1_prime}\\
\text{the supremum on the LHS of \eqref{e:fucking_bounds_1_prime} is attained at}\;\,{x}_i, {x}_i'\;\,\text{as above},\label{e:fucking_bounds_2_prime}\\
\sum_{j=0}^k |(\nabla^{j,g_P}\eta_i')({x}_i)|_{{g}_P(x_i)}\leq C.\label{e:fucking_bounds_4_prime}
\end{align}
(In all of the following arguments, it will be necessary to pass to subsequences or diagonal sequences many times, but we will often not mention this again explicitly.)\medskip\

\noindent \emph{Proof of Claim 1}. Notice that \eqref{e:fucking_bounds_1_prime}, \eqref{e:fucking_bounds_2_prime} are trivial from \eqref{e:start1} and the definition of $x_i, x_i'$ because $\eta_i' = \eta_i - \eta_i^\sharp$, where $\nabla^{k+1,g_P}\eta_i^\sharp = 0$ and hence $\nabla^{k-1,g_P}(L^{{g}_{P}}\eta_i^\sharp)=0$. Also, \eqref{e:fucking_bounds_4_prime} is trivial from \eqref{e:start2} as long as $R_i \leq C$ because the orthogonal projection map onto $J_i^k(\mathcal{L})^\perp$ is norm nonincreasing. Thus, passing to a subsequence, it suffices to prove \eqref{e:fucking_bounds_4_prime} under the assumption that $R_i \to \infty$.

Then let us assume that $R_i \to \infty$, and assume for a contradiction that
\begin{align}
\label{e:wtf101}\mu_i = \sum_{j=0}^k |(\nabla^{j,g_P}\eta_i')({x}_i)|_{g_P(x_i)} \to \infty.
\end{align}
Now \eqref{e:fucking_bounds_1_prime} implies in particular that
\begin{align}
\label{e:wtf100} [\nabla^{k,g_P}\eta_i']_{C^{\alpha}(B^{g_P}(x_i,(1-\delta_i)R_i))}
\leq [\nabla^{k,g_P}\eta_i']_{C^{\alpha}(B^{g_P}(p_i, R_i))} < 1/\epsilon.
\end{align}
Thus, after translating the $\R^d$-factor if necessary, we may assume that $x_i \to x_\infty \in \R^d \times Y$ and that $\hat\eta_i' = \mu_i^{-1}\eta_i'$ converges to some $q$-form $\hat{\eta}_\infty' \in C^{k,\alpha}_{\rm loc}(\R^d \times Y)$ in the $C^{k,\beta}_{\rm loc}$ sense for all $\beta < \alpha$. Then

$\bullet$ $\nabla^{k,g_P}\hat\eta_\infty'$ is $g_P$-parallel by \eqref{e:wtf100} (so in particular, $\hat\eta_\infty'$ is actually smooth),

$\bullet$ $\hat\eta_\infty' \in \mathcal{S}(\R^d\times Y)$ by Assumption (1) of Theorem \ref{t:schauder}, and

$\bullet$ $\sum_{j=0}^k |(\nabla^{j,g_P}\hat\eta_\infty')(x_\infty)|_{g_P(x_\infty)} = 1$ by \eqref{e:wtf101}.

\noindent In particular, it follows that $\hat\eta'_\infty \in \mathcal{L}$, and that $\mathcal{L} J^k_\infty\hat\eta'_\infty = J^k_\infty\hat\eta'_\infty \neq 0$ in the obvious sense. However, $\mathcal{L}J^k_i\eta_i' = 0$ by construction, hence $\mathcal{L}J^k_i\hat\eta_i' = 0$ and $\mathcal{L}J^k_\infty\hat\eta_\infty' = 0$. This is a contradiction.  \hfill $\Box$\medskip\

Let $\lambda_i = d^{g_{P}}(x_i, x_i')^{-1} \geq (2\delta_i R_i)^{-1}$, rescale $g_{P} = \lambda_i^{-2}\tilde{g}_{i}$, $\eta_i' = \lambda_i^{-q-k-\alpha} \tilde\eta_i'$, and write $\tilde{p}_i, \tilde{x}_i, \tilde{x}_i'$ instead of $p_i, x_i, x_i'$ to emphasize that these should be viewed as points in rescaled spaces. Then
\begin{align}
1 = [\nabla^{k,\tilde{g}_i}\tilde{\eta}_i']_{C^{\alpha}(B^{\tilde{g}_{i}}(\tilde{p}_i,\lambda_i \delta_i R_i))} > \epsilon [\nabla^{k,\tilde{g}_i}\tilde{\eta}_i']_{C^{\alpha}(B^{\tilde{g}_{i}}(\tilde{p}_i,\lambda_iR_i))} + i[\nabla^{k-1,\tilde{g}_i}L^{\tilde{g}_{i}} \tilde{\eta}_i']_{C^{\alpha}(B^{\tilde{g}_{i}}(\tilde{p}_i,\lambda_iR_i))},\label{e:fucking_bounds_1}\\
\text{the supremum on the LHS of \eqref{e:fucking_bounds_1} is attained at}\;\,\tilde{x}_i, \tilde{x}_i'\;\,\text{with}\;\,d^{\tilde{g}_i}(\tilde{x}_i,\tilde{x}_i') = 1,\label{e:fucking_bounds_2}\\
\sum_{j=0}^k \lambda_i^{-k+j-\alpha}|(\nabla^{j,\tilde{g}_i} \tilde\eta_i')(\tilde{x}_i)|_{\tilde{g}_i(\tilde{x}_i)}
\leq C.\label{e:fucking_bounds_4}
\end{align}

We would now like to take a pointed limit of the pointed spaces $X_i = (\R^d \times Y, \tilde{g}_{i}, \tilde{x}_i)$. Up to passing to a subsequence, one of the following three cases must occur.\medskip

\noindent {\bf Case 1:} $\lambda_i \to \infty$. {\bf $X_i$ converges to $(\R^{d+e}, g_{\R^{d+e}}, 0)$ in the $C^\infty$ Cheeger-Gromov sense.}\medskip\

\noindent \emph{Deriving a contradiction in Case 1}. We aim to use \eqref{e:fucking_bounds_1} to get a limit $\tilde{\eta}_\infty' \in C$\begin{small}$^{k,\alpha}_{\rm loc}$\end{small}$(\R^{d+e})$ with $\tilde\eta_i' \to \tilde\eta_\infty'$ in $C$\begin{small}$^{k,\beta}_{\rm loc}$\end{small} for all $\beta < \alpha$. If this is possible, then
$[\nabla^k\tilde\eta_\infty']$\begin{small}$_{C^{\alpha}(\R^{d+e})}$\end{small} $\leq$ $C$ thanks to \eqref{e:fucking_bounds_1} and Remark \ref{tiammazzo2}, so that $|\nabla^k \tilde{\eta}_\infty'|= O(r^{\alpha})$ at infinity,  but also $[\nabla^{k-1}L\tilde\eta_\infty']$\begin{small}$_{C^{\alpha}(\R^{d+e})}$\end{small} $=$ $0$ by \eqref{e:fucking_bounds_1} and Remark  \ref{tiammazzo2}, where $L$ now denotes the Euclidean $d + \delta$. In particular, the tensor $\nabla^{k-1}L\tilde\eta_\infty'$ is constant. Applying $L$ to this equation, commuting $L$ and $\nabla^{k-1}$, and using the fact that $L^2 = \Delta$, it follows that $\Delta\tilde\eta_\infty' = 0$ if $k = 1$, and that $\Delta\tilde\eta_\infty'$ is a polynomial of degree at most $k-2$ if $k \geq 2$. Thus, after subtracting a polynomial of degree $\leq k$ if $k \geq 2$, all coefficients of $\tilde\eta_\infty'$ are entire $O(r^{k+\alpha})$ harmonic functions on $\R^{d+e}$, hence are themselves polynomials of degree $\leq k$ by Liouville. This implies $[\nabla^k\tilde\eta_\infty']_{C^{\alpha}(\R^{d+e})} = 0$, in contradiction to the fact that $[\nabla^k\tilde\eta_\infty']_{C^\alpha(\R^{d+e})} = 1$ which follows from \eqref{e:fucking_bounds_2} by using Remark \ref{tiammazzo2}.

The problem with this argument is that the $\epsilon$-term on the right-hand side of \eqref{e:fucking_bounds_1} controls only the $C^{\alpha}$-seminorm of $\nabla^{k,\tilde{g}_i} \tilde\eta_i'$ (on $B^{\tilde{g}_i}(\tilde{p}_i, \lambda_iR_i) \supset B^{\tilde{g}_i}(\tilde{x}_i,  (1 - \delta_i)\lambda_iR_i) \supset B^{\tilde{g}_i}(\tilde{x}_i, \frac{i}{3})$ for $i \geq 3$) rather than the full $C^{k,\alpha}$-norm of $\tilde\eta_i'$. Thus, we lack uniform bounds for $|(\nabla^{k,\tilde{g}_i}\tilde\eta_i')(\tilde{x}_i)|_{\tilde{g}_i(\tilde{x}_i)}, \ldots, |\tilde\eta_i'(\tilde{x}_i)|_{\tilde{g}_i(\tilde{x}_i)}$, and the partial bounds of \eqref{e:fucking_bounds_4} are not enough for this. To fix this, we will use \eqref{e:fucking_bounds_4} to construct a new sequence $\tilde\eta_i''$ with the same good properties as $\tilde\eta_i'$, but with $(\nabla^{k,\tilde{g}_i}\tilde\eta_i'')(\tilde{x}_i) = 0, \ldots, \tilde\eta_i''(\tilde{x}_i) = 0$. Then it will be clear (thanks to Lemma \ref{tiammazzo1}, which allows us to compare H\"older norms with respect to a fixed and a mildly varying metric) that the above argument applies to $\tilde\eta_i''$ in place of $\tilde\eta_i'$.

To construct $\tilde\eta_i''$, let $\mathbf{x}^1, \ldots, \mathbf{x}^{d+e}$ be normal coordinates for $g_P$ centered at $x_i$ such that
\begin{align*}
\left|\frac{\partial^j}{\partial\mathbf{x}^j}(g_P(\mathbf{x})_{ab}-\delta_{ab})\right| \leq \frac{1}{100}|\mathbf{x}|^{2-\min\{2,j\}}\,\;{\rm for}\;\, |\mathbf{x}| \leq 2\;\,{\rm and}\;\,j \in \{0,1,\ldots,k+1\}.
\end{align*}
(Thanks to the compactness of $Y$, this can be achieved by rescaling $g_Y$ by a fixed constant if needed.)
Define $\tilde{\mathbf{x}}^j = \lambda_i \mathbf{x}^j$, so that $\tilde{\mathbf{x}}^1, \ldots, \tilde{\mathbf{x}}^{d+e}$ are normal coordinates for $\tilde{g}_i$ centered at $\tilde{x}_i$ with
\begin{align}\label{e:normal_coords}
\left|\frac{\partial^j}{\partial\tilde{\mathbf{x}}^j}(g_P(\tilde{\mathbf{x}})_{ab} -\delta_{ab})\right| \leq \frac{\lambda_i^{-{\max\{2,j\}}}}{100}|\tilde{\mathbf{x}}|^{2-\min\{2,j\}}\,\;{\rm for}\;\, |\tilde{\mathbf{x}}| \leq 2\lambda_i\;\,{\rm and}\;\,j \in \{0,1,\ldots,k+1\}.
\end{align}
Then let $(\tilde\eta_i')^{\sharp} \in \mathcal{A}^q(B^{\tilde{g}_i}(\tilde{x}_i,\lambda_i))$ denote the $k$-th order Taylor polynomial of $\tilde\eta_i'$ at $\tilde{x}_i$ with respect to the coordinate system $\tilde{\mathbf{x}}^1, \ldots, \tilde{\mathbf{x}}^{d+e}$, and define $\tilde\eta_i'' = \tilde\eta_i' - (\tilde\eta_i')^{\sharp} \in \mathcal{A}^q(B^{\tilde{g}_i}(\tilde{x}_i,\lambda_i))$.\medskip\

\noindent \emph{Claim 2}. There is a $C$ such that for all $R >0$ there is an $i_R \in \N$ such that for all $i \geq i_R$,
\begin{align}
[\nabla^{k,\tilde{g}_i}\tilde{\eta}_i'']_{C^{\alpha}(B^{\tilde{g}_{i}}(\tilde{x}_i,R))} < (1/\epsilon)+ C(\lambda_i/R)^{\alpha-1},\label{e:fucking_bounds_1a_prime_prime}\\
[\nabla^{k-1,\tilde{g}_i}L^{\tilde{g}_{i}} \tilde{\eta}_i'']_{C^{\alpha}(B^{\tilde{g}_{i}}(\tilde{x}_i,R))}< (1/i) +  C(\lambda_i/R)^{\alpha-1},\label{e:fucking_bounds_1b_prime_prime}\\
||(\nabla^{k,\tilde{g}_i}\tilde{\eta}_i'')(\tilde{x}_i) - \mathbf{P}_{\tilde{x}_i'\tilde{x}_i}^{\tilde{g}_i}[(\nabla^{k,\tilde{g}_i}\tilde{\eta}_i'')(\tilde{x}_i')] |_{\tilde{g}_i(\tilde{x}_i)} - 1 | \leq C\lambda_i^{\alpha-1},\label{e:fucking_bounds_2_prime_prime}\\
(\nabla^{j,\tilde{g}_i}\tilde\eta_i'')(\tilde{x}_i) = 0 \;\,{\rm for\;all}\;\,j \in \{0,1,\ldots, k\}.\label{e:fucking_bounds_4_prime_prime}
\end{align}

\noindent \emph{Proof of Claim 2}. \eqref{e:fucking_bounds_4_prime_prime} is clear from the definition of $\tilde\eta''_i$. We will now derive \eqref{e:fucking_bounds_1a_prime_prime}, \eqref{e:fucking_bounds_1b_prime_prime}, \eqref{e:fucking_bounds_2_prime_prime} for $i \geq i_R$ from \eqref{e:fucking_bounds_1}, \eqref{e:fucking_bounds_2} by using the auxiliary estimates \eqref{e:fucking_bounds_4} and \eqref{e:normal_coords}.

The seminorms in \eqref{e:fucking_bounds_1} give control over $B^{\tilde{g}_i}(\tilde{p}_i, \lambda_i R_i) \supset B^{\tilde{g}_i}(\tilde{x}_i, (1-\delta_i)\lambda_iR_i) \supset B^{\tilde{g}_i}(\tilde{x}_i,\frac{i}{3})$ for $i \geq 3$. Thus, as long as $i \geq 3R$, it makes sense to try to use \eqref{e:fucking_bounds_1} to prove \eqref{e:fucking_bounds_1a_prime_prime}, \eqref{e:fucking_bounds_1b_prime_prime} and \eqref{e:fucking_bounds_2_prime_prime}. Since we are going to use \eqref{e:normal_coords}, we also need to choose $i_R$ so large that $i \geq i_R$ implies $\lambda_i \geq \max\{2,2R\}$.

Since $\tilde\eta_i'' = \tilde\eta_i' - (\tilde\eta_i')^{\sharp}$, \eqref{e:fucking_bounds_1a_prime_prime}, \eqref{e:fucking_bounds_1b_prime_prime}, \eqref{e:fucking_bounds_2_prime_prime} would follow from \eqref{e:fucking_bounds_1}, \eqref{e:fucking_bounds_2} if we knew that
$$[\nabla^{k,\tilde{g}_i} (\tilde\eta'_i)^{\sharp}]_{C^{\alpha}(B^{\tilde{g}_{i}}(\tilde{x}_i,\rho))} \leq C(\lambda_i/\rho)^{\alpha-1}$$ for all $\rho \leq \lambda_i/2$. But this is fairly straightforward to prove by noting that
\begin{align*}
[\nabla^{k,\tilde{g}_i}(\tilde\eta_i')^{\sharp}]_{C^{\alpha}(B^{\tilde{g}_{i}}(\tilde{x}_i,\rho))} &\leq (2\rho)^{1-\alpha}\biggl\|\left
(\frac{\partial}{\partial\tilde{\mathbf{x}}} + \Gamma^{\tilde{g}_i}(\tilde{\mathbf{x}})\right)^{k+1}
\sum_{\substack{\beta\in\N^{d+e}\\ |\beta|\leq k}} \frac{1}{\beta!} \frac{\partial^{|\beta|}\tilde\eta_i'}{\partial\tilde{\mathbf{x}}^\beta}(\tilde{\mathbf{x}}_i)(\tilde{\mathbf{x}} - \tilde{\mathbf{x}}_i)^\beta
\biggr\|_{L^\infty(B^{\tilde{g}_i}(\tilde{x}_i,2\rho))}
\end{align*}
and estimating the big $L^\infty$ norm by $C\lambda_i^{\alpha-1}$, as follows.

(1) Schematically $\nabla^a\Gamma = (\partial + \Gamma)^a\Gamma = \sum \partial^{a_1}\Gamma \cdots \partial^{a_\ell}\Gamma$, where $a_1 + \cdots + a_\ell + \ell = a+1$ by counting the total number of $\partial$s and $\Gamma$s in each term of a complete expansion of the left-hand side. Now $ \partial^b \Gamma$ $=$ $O( \lambda_i^{-b-1})$ by \eqref{e:normal_coords} (we can do better for $b = 0$ but this is not useful), so $\nabla^a\Gamma = O(\lambda_i^{-a-1})$. \hfill

(2) Writing $\tilde\eta_i' = \eta$, we have $\partial^b\eta= (\nabla-\Gamma)^b\eta = \sum \nabla^{b_1}\Gamma \cdots \nabla^{b_\ell}\Gamma \cdot \nabla^c\eta$ with $b_1 + \cdots + b_\ell + \ell + c = b$. Evaluating at $\tilde{\mathbf{x}}_i$ and using Step (1) above and \eqref{e:fucking_bounds_4}, we get $(\partial^b\eta)(\tilde{\mathbf{x}}_i) = O(\lambda_i^{k+\alpha-b})$.

(3) The expression we care about can be expanded to
\begin{align*}
(\partial + \Gamma)^{k+1} ((\partial^{|\beta|}\eta)(\tilde{\mathbf{x}}_i) (\tilde{\mathbf{x}}-\tilde{\mathbf{x}}_i)^\beta) = (\partial^{|\beta|}\eta)(\tilde{\mathbf{x}}_i) \sum \partial^{a_1}\Gamma \cdots \partial^{a_\ell}\Gamma \cdot \partial^b (\tilde{\mathbf{x}}-\tilde{\mathbf{x}}_i)^\beta,
\end{align*}
where $a_1 + \cdots + a_\ell + \ell + b = k+1$ again by counting the number of $\partial$s and $\Gamma$s. The desired estimate now follows using that $(\partial^{|\beta|}\eta)(\tilde{\mathbf{x}}_i) = O(\lambda_i^{k+\alpha-|\beta|})$ by Step (2) and $\partial^a\Gamma = O(\lambda_i^{-a-1})$ by \eqref{e:normal_coords}. \hfill $\Box$\medskip\

\begin{rk}
In the proofs of Schauder estimates by contradiction in \cite{HN, Si}, it was important that subtracting a possibly unbounded Taylor polynomial changes neither $[\eta]$ nor $[L\eta]$. For example in \cite{Si} this meant that the blowup argument could be applied only to constant coefficient operators. The idea of Claim 2 (that subtracting unbounded Taylor polynomials does change $[\eta]$ and $[L\eta]$ in general, but the errors may actually be manageable in good cases) is taken from the proof of \cite[Thm 2.8]{Sz}.
\end{rk}

\noindent {\bf Case 2:} $\lambda_i \to \lambda_\infty \in \R^+$. {\bf $X_i$ converges to $(\R^d \times Y, \lambda_\infty^2 g_P, (0,{y}_\infty))$ in the standard $C^\infty$ sense.}\medskip\

\noindent \emph{Deriving a contradiction in Case 2}. In Case 2, the quantities $|(\nabla^{k,\tilde{g}_i}\tilde\eta_i')(\tilde{x}_i)|_{\tilde{g}_i(\tilde{x}_i)}, \ldots, |\tilde\eta_i'(\tilde{x}_i)|_{\tilde{g}_i(\tilde{x}_i)}$ are uniformly bounded already by  \eqref{e:fucking_bounds_4}, so we require no additional modifications of the $\tilde\eta_i'$. Thus, using \eqref{e:fucking_bounds_1}, \eqref{e:fucking_bounds_2}, and $\lambda_i R_i \geq i/2$, we can pass to a limit $\tilde\eta_\infty'$ $\in$ \begin{small}$C^{k,\alpha}_{\rm loc}$\end{small}$(\R^d \times Y)$ with $\tilde\eta_i'\to\tilde\eta_\infty'$ in \begin{small}$C^{k,\beta}_{\rm loc}$\end{small} for all $\beta<\alpha$, $[\nabla^{k,g_P}\tilde\eta_\infty']_{C^{\alpha}(\R^d\times Y,g_P)} \neq 0$, $|\nabla^{k,g_P}\tilde{\eta}_\infty'|_{g_P} = O(r^{\alpha})$, and  $[\nabla^{k-1,g_P}L^{g_P}\tilde\eta_\infty']_{C^{\alpha}(\R^d\times Y, g_P)} = 0$ (here we also need to again apply Lemma \ref{tiammazzo1} and Remark \ref{tiammazzo2} the same way we did in Case 1). In particular, $\tilde\eta_\infty'$ is in fact smooth by elliptic regularity, so that $\tilde\eta_\infty' \in \mathcal{S}(\R^d \times Y)$ by Assumption (1) of Theorem \ref{t:schauder}. Taken together, these properties obviously contradict Assumption (2).\medskip\

\noindent {\bf Case 3:} $\lambda_i \to 0$. {\bf $X_i$ converges to $(\R^d, g_{\R^d}, 0)$ in the Gromov-Hausdorff sense.}\medskip\

\noindent \emph{Deriving a contradiction in Case 3}. We begin by replacing $\tilde{g}_i, \tilde\eta_i', \tilde{p}_i, \tilde{x}_i,\tilde{x}_i'$ by  their pullbacks under the diffeomorphism $(z,y) \mapsto (\lambda_i^{-1}z,y)$. Then \eqref{e:fucking_bounds_1}, \eqref{e:fucking_bounds_2}, \eqref{e:fucking_bounds_4} remain unchanged, but we now have the useful property that $\tilde{g}_i = g_{\R^d} + \lambda_i^2 g_Y \to g_{\R^d}$ smoothly as tensors on $\R^d \times Y$. Also, we can assume as usual that $\tilde{x}_i \to \tilde{x}_\infty \in \R^d \times Y$ by translating the $\R^d$-factor if necessary.

Let us write $\tilde\eta_i' = \sum_{t=0}^e (\tilde\eta_i')^t$ according to the decomposition $\Lambda^q(\R^d \times Y) = \bigoplus_{s+t=q} \Lambda^s\R^d \otimes \Lambda^t Y$. Let  $(\hat\eta_i')^t = \lambda_i^{-t}(\tilde\eta_i')^t$. We would now like to translate \eqref{e:fucking_bounds_1}, \eqref{e:fucking_bounds_2}, \eqref{e:fucking_bounds_4} into analogous statements with respect to the \emph{fixed} reference metric $g_P = g_{\R^d} + g_Y$ for each rescaled component $(\hat\eta_i')^t$.

The decomposition of $\tilde\eta_i'$ is $\tilde{g}_i$-orthogonal at each point, and is invariant under $\tilde{g}_i$-parallel transport. Moreover, $\nabla^{\tilde{g}_i} = \nabla^{g_P}$ (because the Levi-Civita connection of a Riemannian product metric is invariant under scaling the factors), and $\tilde{g}_i \leq C g_P$. Thus, it follows directly from \eqref{e:fucking_bounds_1}, \eqref{e:fucking_bounds_4} that
\begin{align}
[\nabla^{k,g_P}(\hat\eta_i')^t ]_{C^{\alpha}(B^{g_P}(\tilde{x}_i, \frac{i}{4}))} \leq C, \label{e:wtf10^7}\\
\sum_{j=0}^k \lambda_i^{-k+j-\alpha}|(\nabla^{j,g_P}(\hat\eta_i')^t)(\tilde{x}_i)|_{g_P(\tilde{x}_i)} \leq C,\label{e:wtf10^8}
\end{align}
for all $t \in \{0,\ldots,e\}$ and for all large enough $i$. Notice carefully that all norms here are understood to be measured with respect to $g_P$. To prove \eqref{e:wtf10^7}, \eqref{e:wtf10^8}, one also needs to decompose $\nabla = \nabla_{\mathbf{b}} + \nabla_{\mathbf{f}}$, where $\nabla =  \nabla^{\tilde{g}_i} = \nabla^{g_P}$ and the subscripts $\mathbf{b}$ and $\mathbf{f}$ denote base and fiber directions, respectively. The lengths of  $\nabla_{\mathbf{b}}(\tilde\eta_i')^t$ and $\nabla_{\mathbf{f}}(\tilde\eta_i')^t$ scale differently when passing from $\tilde{g}_i$ to $g_P$, but in \eqref{e:wtf10^7}, \eqref{e:wtf10^8} this is actually helpful because we only care about upper bounds.

The following claim is needed to deal with \eqref{e:fucking_bounds_2} and with the $L^{\tilde{g}_i}$-part of \eqref{e:fucking_bounds_1}.\medskip\

\noindent \emph{Claim 3}. There exists a $C$ such that $|\nabla^{\tilde{g}_i}_{\mathbf{f}}\nabla^{j-1,\tilde{g}_i}\tilde{\eta}_i'|_{\tilde{g}_i} \leq C\lambda_i^{k-j+\alpha}$ on $B^{\tilde{g}_i}(\tilde{x}_i, \frac{i}{4})$ for all $j \in \{1,\ldots, k\}$ and for all large $i$.  In particular, since ${\mathbf{b}}$ and ${\mathbf{f}}$ covariant derivatives commute, every $j$-fold covariant derivative of $\tilde\eta_i'$ with at least one subscript $\mathbf{f}$ is locally $O(\lambda_i^{k-j+\alpha})$ with respect to $\tilde{g}_i$.   \medskip\

\noindent \emph{Proof of Claim 3}. If $i$ is large enough, then $\pi^{-1}(\pi(B^{\tilde{g}_i}(\tilde{x}_i, \frac{i}{4}))) \subset B^{\tilde{g}_i}(\tilde{x}_i, \frac{i}{3})$, where $\pi: \R^d \times Y \to \R^d$ is the  projection. Then for all $z \in \pi(B^{\tilde{g}_i}(\tilde{x}_i, \frac{i}{4}))$, viewing $\nabla^{j-1,\tilde{g}_i}\tilde\eta_i'$ as a section of the restriction to $\{z\} \times Y$ of an appropriate vector bundle over $\R^d \times Y$,
\begin{align*}
\|\nabla^{\lambda_i^2 g_Y}_{\mathbf{f}}\nabla^{j-1,\tilde{g}_i}\tilde\eta_i'\|_{L^\infty(\{{z}\}\times Y, \lambda_i^2g_Y)}
&\leq C\lambda_i^{k-j+\alpha}[\nabla^{k-j+1,\lambda_i^2g_Y}_{\mathbf{f}}\nabla^{j-1,\tilde{g}_i}\tilde\eta_i']_{C^{\alpha}(\{{z}\}\times Y, \lambda_i^2 g_Y)}\\
&\leq C\lambda_i^{k-j+\alpha}[\nabla^{k,\tilde{g}_i} \tilde\eta_i']_{C^{\alpha}(B^{\tilde{g}_i}(\tilde{x}_i, \frac{i}{3}))} \leq C\lambda_i^{k-j+\alpha}.
\end{align*}
To see this, apply Lemma \ref{l:braindead} on $(\{z\}\times Y, g_Y)$, rescale the metric by $\lambda_i^2$, and use \eqref{e:fucking_bounds_1}. \hfill$\Box$\medskip\

Given Claim 3, we are now able to rewrite \eqref{e:fucking_bounds_2} in terms of $g_P$ and the individual $(\hat\eta_i')^t$s. Indeed, it follows easily from \eqref{e:fucking_bounds_2}, \eqref{e:fucking_bounds_4}, and Claim 3 that there exists some index $t_* \in \{0,\ldots,e\}$ (which we may assume is independent of $i$) such that for all $i$,
\begin{align}
|(\nabla^{k,g_P}_{\mathbf{b}}(\hat\eta_i')^{t_*})(\tilde{x}_i')|_{g_P(\tilde{x}_i')} \geq \frac{1}{e+1} - C\lambda_i^\alpha,\label{e:wtf10^9}\\
C^{-1}\leq d^{g_P}(\tilde{x}_i, \tilde{x}_i') \leq C.\label{e:wtf10^10}
\end{align}
(Without Claim 3, due to the different scaling behaviors of $\nabla_{\mathbf{b}}$ and $\nabla_{\mathbf{f}}$, we would not be able to rule out that $|(\nabla^{k,g_P}(\hat\eta_i')^t)(\tilde{x}_i')|_{g_P(\tilde{x}_i')} \leq C\lambda_i^\alpha$ for all $t$, and then there would be no contradiction as $i \to \infty$. Even with Claim 3, the same problem arises if we do not consider each $(\tilde\eta_i')^t$ separately.)

It remains to deal with the $L^{\tilde{g}_i}$-part of \eqref{e:fucking_bounds_1}. For this it turns out to be convenient to first pass to a limit. Thanks to \eqref{e:wtf10^7}, \eqref{e:wtf10^8}, we can assume that $(\hat\eta_i')^{t_*}$ converges to a $C$\begin{small}$^{k,\alpha}_{\rm loc}$\end{small} section $(\hat\eta_\infty')^{t_*}$ of the bundle $\Lambda^{q-t_*}\R^d \otimes \Lambda^{t_*}Y$ on $\R^d \times Y$, with convergence taking place in the $C$\begin{small}$^{k,\beta}_{\rm loc}$\end{small} topology for all $\beta<\alpha$. Since $|\nabla^{g_P}_{\mathbf{f}}(\hat{\eta}_i')^{t_*}|_{g_P} = O(\lambda_i^{k+\alpha})$ locally uniformly by Claim 3, it follows that $\nabla^{g_P}_{\mathbf{f}}(\hat\eta_\infty')^{t_*} = 0$. This lets us view $(\hat\eta_\infty')^{t_*}$ as a section (still denoted by the same symbol) of the bundle $\Lambda^{q-t_*}\R^d \otimes \mathcal{P}^{t_*}$ over $\R^d$, where $\mathcal{P}^t$ denotes the space of all $g_Y$-parallel $t$-forms on $Y$. By \eqref{e:wtf10^7}--\eqref{e:wtf10^10},
$$
0 < [\nabla^{k,\R^d}(\hat\eta_\infty')^{t_*}]_{C^{\alpha}(\R^d)}\leq C.
$$
The preceding equation will contradict Liouville's theorem once we deduce from \eqref{e:fucking_bounds_1} that
$$[\nabla^{k-1,\R^d}L^{\R^d}(\hat\eta_\infty')^{t_*}]_{C^{\alpha}(\R^d)} = 0,$$
proving that each component of $(\hat\eta_\infty')^{t_*}$ is the sum of a harmonic function and a polynomial of degree $\leq k$ with values in $\mathcal{P}^{t_*}$. To this end, fix $z \neq z'$ in $\R^d$. Fix $y \in Y$ and let $x = (z,y)$ and $x' = (z',y)$. Then for all large enough $i$ it follows from \eqref{e:fucking_bounds_1} and Claim 3 that
\begin{align}\label{e:almostthere2}
\left|\sum_{t=0}^e \left((\nabla^{k-1,\tilde{g}_i}_{\mathbf{b}}L_{\mathbf{b}}^{\tilde{g}_i}(\tilde\eta_i')^{t})(x) -  \mathbf{P}^{\tilde{g}_i}_{x'x}[(\nabla^{k-1,\tilde{g}_i}_{\mathbf{b}}L_{\mathbf{b}}^{\tilde{g}_i}(\tilde\eta_i')^{t})(x')] \right)\right|_{\tilde{g}_i(x)} < \frac{1}{i}|z-z'|_{\R^d}^\alpha + C\lambda_i^\alpha,
\end{align}
where $L^{\tilde{g}_i}_{\mathbf{b}}$ denotes the part of $L^{\tilde{g}_i}$ that only involves $\tilde{g}_i$-covariant derivatives in the base directions.
Now $L^{\tilde{g}_i}_{\mathbf{b}}$ (unlike $L^{\tilde{g}_i}$) sends sections of $\Lambda^\bullet\R^d \otimes \Lambda^t Y$ to sections of $\Lambda^\bullet\R^d \otimes \Lambda^t Y$ for every $t$, so that the terms of the sum on the left-hand side of \eqref{e:almostthere2} are $\tilde{g}_i(x)$-orthogonal. Thus,
\begin{align*}
|(\nabla^{k-1,\R^d}&L^{\R^d}(\hat\eta_\infty')^{t_*})(z) - (\nabla^{k-1,\R^d}L^{\R^d}(\hat\eta_\infty')^{t_*})(z')|_{\R^d} \\
&= \lim_{i\to\infty} |(\nabla^{k-1,g_P}_{\mathbf{b}} L^{g_P}_{\mathbf{b}}(\hat\eta_i')^{t_*})(x) - \mathbf{P}^{g_P}_{x'x}[(\nabla_{\mathbf{b}}^{k-1,g_P}L^{g_P}_{\mathbf{b}}(\hat\eta_i')^{t_*})(x')] |_{g_P(x)}\\
&=\lim_{i\to\infty} |(\nabla^{k-1,\tilde{g}_i}_{\mathbf{b}}L_{\mathbf{b}}^{\tilde{g}_i}(\tilde\eta_i')^{t_*})(x) - \mathbf{P}^{\tilde{g}_i}_{x'x}[(\nabla^{k-1,\tilde{g}_i}_{\mathbf{b}}L_{\mathbf{b}}^{\tilde{g}_i}(\tilde\eta_i')^{t_*})(x')] |_{\tilde{g}_i(x)} \\
&\leq \liminf_{i \to \infty} \; ((1/i)|z-z'|_{\R^d}^\alpha + C\lambda_i^\alpha) = 0,
\end{align*}
as desired. In the second step we have used once again that $\nabla^{g_P} = \nabla^{\tilde{g}_i}$.
\end{proof}

\subsection{A Schauder estimate for $i\partial\ov{\partial}$-exact $2$-forms}
Let $Y$ be a compact K\"ahler manifold without boundary, let $d = 2m$, identify $\R^d = \C^m$, and let $\mathcal{S}$ be the presheaf of $i\partial\ov{\partial}$-exact real $(1,1)$-forms of class $C^{k,\alpha}_{\rm loc}$ on $\C^m\times Y$ with respect to the product complex structure, where $k \in \mathbb{N}_{\geq 1}$, $\alpha \in (0,1)$. The following two propositions show that Assumptions (1)--(2) of Theorem \ref{t:schauder} hold in this setting.

\begin{proposition}\label{p:closure}
Let $U_i$ be an exhaustion of $\mathbb{C}^m \times Y$ by open sets. Let $\eta_i \in C$\begin{small}$^{k,\alpha}_{\rm loc}$\end{small}$(U_i)$ be a sequence of $i\partial\ov{\partial}$-exact $(1,1)$-forms. If $\eta_i$ converges to some $2$-form $\eta_\infty \in C^\infty_{\rm loc}(\mathbb{C}^m \times Y)$ in the $C$\begin{small}$^{k,\beta}_{\rm loc}$\end{small} topology for some $\beta<\alpha$, then $\eta_\infty$ is again $i\partial\ov{\partial}$-exact.
\end{proposition}

\begin{proof}The proof is similar to the arguments in \cite[p.2936--2937]{TZ}.
The limit $\eta_\infty$ is closed $(1,1)$ because the convergence takes place in $C^{k,\beta}_{\rm loc}$. We would like to show that $\eta_\infty$ is in fact $\ddbar$-exact. For all $z\in \mathbb{C}^m$ the restriction $\eta_{\infty}|_{\{z\}\times Y}$ is $d$-exact because $Y$ is compact without boundary and this form integrates to zero against every $d$-closed form. By the K\"unneth formula, $\eta_\infty$ is $d$-exact on $\mathbb{C}^m\times Y$. Thus, there exists a smooth real $1$-form $\zeta$ on $\C^m \times Y$ such that
$$\eta_\infty=d\zeta=\de\zeta^{0,1}+\ov{\de\zeta^{0,1}},\;\, \db\zeta^{0,1}=0.$$
The Leray spectral sequence computing the Dolbeault cohomology of $\mathbb{C}^m\times Y$ via projection onto the $\C^m$ factor degenerates at the first page, giving
$$H^{0,1}(\mathbb{C}^m\times Y)\cong H^{0,1}(Y)\otimes H^{0}(\mathbb{C}^m,\mathcal{O}_{\mathbb{C}^m}),$$
where we use the fact that $H^{0,1}(\mathbb{C}^m)=0$ by the $\db$-Poincar\'e lemma. (The tensor factor $H^0(\C^m, \mathcal{O}_{\C^m})$ is missing in \cite{TZ}. Here we correct the arguments of \cite{TZ} to account for this term. We thank Y. Zhang for some very helpful discussions regarding this point.) Choosing $\db$-closed $(0,1)$-forms $\theta_j$ on $Y$ whose cohomology classes are a basis of $H^{0,1}(Y)$, we obtain
$$\zeta^{0,1}=\sum \sigma_j\theta_j+\db h$$
for some holomorphic functions $\sigma_j$ on $\mathbb{C}^m$ and a smooth $\mathbb{C}$-valued function $h$ on $\C^m\times Y$. We can pick the $\theta_j$ to be harmonic, so in particular also $\de$-closed, which gives
$$\eta_\infty=2\mathrm{Re}\sum d\sigma_j\wedge\theta_j+2\ddbar\mathrm{Im}\,h.$$
On $U_i$, write $\eta_i=\ddbar v_i$, and let $u_i=v_i-2\mathrm{Im}\,h$, so that
$\ddbar u_i$ converges to $2\mathrm{Re}\sum d\sigma_j\wedge\theta_j$ in $C$\begin{small}$^{k,\beta}_{\rm loc}$\end{small} as $i\to\infty$.
In particular, restricting this to any fiber $\{z\}\times Y$ ($z\in\mathbb{C}^m$), we see that $\ddbar u_i|$\begin{small}$_{\{z\}\times Y}\to 0$\end{small}. Let $\underline{u_i}$ be the $C$\begin{small}$^{k+2,\alpha}_{\rm loc}$\end{small} function on $\mathbb{C}^m$ obtained by averaging $u_i$ over these fibers. Then $\ddbar \underline{u_i}$ has uniform local $C$\begin{small}$^{k,\beta}$\end{small} bounds. The functions $u_i-\underline{u_i}$ thus have fiberwise average zero, and the forms $\ddbar(u_i-\underline{u_i})$ have uniform local $C$\begin{small}$^{k,\beta}$\end{small} bounds and their fiberwise restrictions go to zero in $C$\begin{small}$^{k,\beta}$\end{small}. Thus, $u_i-\underline{u_i}$ converges to zero locally uniformly on $\mathbb{C}^m\times Y$, hence locally in $C$\begin{small}$^{k+2,\gamma}$\end{small} for all $\gamma<\beta$. It follows that $\ddbar(u_i-\underline{u_i})\to 0$ locally in $C$\begin{small}$^{k,\gamma}$\end{small}, so the form $2{\rm Re} \sum d\sigma_{j}\wedge \theta_j,$ which is the limit of the $\ddbar u_i$, is also the limit of the $\ddbar \underline{u_i}$. But the $\ddbar \underline{u_i}$ are forms on $\mathbb{C}^m$, so $2{\rm Re} \sum d\sigma_{j}\wedge \theta_j$ is also the pullback of a form on $\mathbb{C}^m$. This is only possible if $d\sigma_j=0$ for all $j$, and so $\eta_\infty=2\ddbar\mathrm{Im}\,h$, as required.
\end{proof}

\begin{proposition}\label{check}
Let $\eta$ be an $i\partial\ov{\partial}$-exact $(1,1)$-form in $C^{k,\alpha}_{\rm loc}(\mathbb{C}^m \times Y)$ such that $|\nabla^{k,g_P}\eta|_{g_P} = O(r^\alpha)$ and $\nabla^{k-1,g_P}\delta^{g_P}\eta$ is $g_P$-parallel. Then $\eta = i\partial\ov{\partial} p$ for some real polynomial $p$ of degree $\leq k+2$ on $\mathbb{C}^m$.
\end{proposition}

\begin{proof}
In this proof we will omit all sub- and superscripts $g_P$ for simplicity.

By assumption, $\eta = dd^c u$ for some real function $u$ on $\C^m\times Y$, which is necessarily $C^{k+2,\alpha}_{\rm loc}$. Then $\delta\eta$ $=$ $\delta dd^c u = - \delta d^cd u = d^c \delta du = -d^c \Delta u$ by the K\"ahler identities. Thus, $\nabla^{k} \Delta u$ is parallel, so $u$ is smooth with $\Delta u = \ell$ for some real polynomial $\ell$ of degree $\leq k$ on $\mathbb{C}^m$, so $u = h + \ell'$, where $h$ is harmonic on $\mathbb{C}^m \times Y$  and $\ell'$ is a real polynomial of degree $\leq k+2$ on $\C^m$. In particular, $|\nabla^k(i\partial\ov{\partial} h)| = O(r^\alpha)$.

Decompose $h =  \underline{h}+ (h - \underline{h})$, where $\underline{h}$ denotes the smooth function on $\C^m$ obtained by averaging $h$ over every fiber. It is easy to see that $\underline{h}$ is harmonic with $|\nabla^k(i\partial\ov{\partial} \underline{h})| = O(r^\alpha)$. Since every coefficient function of the tensor $\nabla^k(i\partial\ov{\partial}\underline{h})$ is harmonic, it follows from
Liouville's theorem that the coefficients of the closed $(1,1)$-form $\varpi = i\partial\ov{\partial}\underline{h}$ are (harmonic) real polynomials of degree $\leq k$.\medskip\

\noindent \emph{Claim 1}. There exists a real polynomial $\underline{h}'$ of degree $\leq k+2$ such that $i\partial\ov{\partial}\underline{h}' = \varpi$. \medskip\

\noindent \emph{Proof of Claim 1}. This is proved by induction on $k$. For $k = 0$, notice that every constant $(1,1)$-form $\varpi$ on $\C^m$ can obviously be written as $i\partial\ov{\partial}$ of a quadratic polynomial. For the inductive step, let $\varpi$ be a closed $(1,1)$-form on $\C^m$ whose coefficients are real polynomials of degree $\leq k+1$. The degree $\leq k$ part  of $\varpi$ is still closed $(1,1)$, so by the inductive hypothesis we can assume that $\varpi$ is $(k+1)$-homogeneous. Then the usual proof of the $d$-Poincar\'e lemma produces a real $1$-form $\zeta$ with $d\zeta = \varpi$ whose coefficients are $(k+2)$-homogeneous polynomials. Then $\ov{\partial}\zeta^{0,1} = 0$ as usual, so we only need to find a homogeneous polynomial $\phi$ of degree $k+3$ with $\ov{\partial}\phi = \zeta^{0,1}$ and put $\underline{h}' = 2{\rm Im}\,\phi$. To this end, write
$$\zeta^{0,1} = \sum_{\ell=1}^m \sum_{\substack{\beta,\gamma\in\N^m\\|\beta|+|\gamma|=k+2}} A^{\ov{\ell}}_{\beta\ov{\gamma}} z^\beta\ov{z}^\gamma d\ov{z}^\ell,$$
where $A_{\beta\ov{\gamma}}^{\ov{\ell}}\in \C$. The condition $\ov{\partial}\zeta^{0,1} = 0$ translates into
$$ \sum_{\substack{\gamma\in\N^m\\|\beta| + |\gamma| = k+2}} A^{\ov{\ell}}_{\beta\ov{\gamma}} \gamma_i \ov{z}^{\gamma - e_i} =  \sum_{\substack{\gamma\in\N^m\\|\beta| + |\gamma| = k+2}} A^{\ov{\imath}}_{\beta\ov{\gamma}}\gamma_\ell \ov{z}^{\gamma - e_\ell}$$
for all $\beta\in\N^m$ with $|\beta| \leq k+2$ and all $i,\ell \in \{1,\ldots, m\}$, where $e_i \in \N^m$ denotes the $i$-th unit vector. Using this, a straightforward computation shows that $\ov{\partial}\phi = \zeta^{0,1}$ as desired, where\

\begin{minipage}{160mm}
\begin{align*}\phi = \sum_{\ell=1}^m \sum_{\substack{\beta,\gamma\in\N^m\\|\beta|+|\gamma|=k+2}} \frac{1}{|\gamma|+1} A^{\ov{\ell}}_{\beta\ov{\gamma}} z^\beta \ov{z}^\gamma \ov{z}^\ell.
\end{align*}
\end{minipage}\hfill $\Box$\vskip4mm

\noindent \emph{Claim 2}. The function $w = h - \underline{h}$ is identically zero. \medskip\

\noindent \emph{Proof of Claim 2}. By the above, $w$ is harmonic on $\C^m \times Y$ with $|\nabla^k(i\partial\ov{\partial} w)| = O(r^\alpha)$. The latter implies that $|\ddbar w|=O(r^{k+\alpha})$, so the fiberwise Laplacian of $w$ is also $O(r^{k+\alpha})$. Since $w$ has fiberwise average zero, fiberwise Moser iteration or the fiberwise Green's formula give $|w| = O(r^{k
+\alpha})$, and hence $|\nabla^i w| = O(r^{k+\alpha})$ for all $i \in \N$ by standard local estimates for harmonic functions on a manifold of $C^\infty$ bounded geometry. Now define  $Q(r) = \int_{B_r \times Y} |\nabla w|^2$. Since $\Delta w = 0$, it follows that
$$Q(r) = \int_{\partial B_r \times Y} w\frac{\partial w}{\partial r} \leq \int_{\partial B_r} \biggl(\int_Y w^2\biggr)^{1/2}\biggl(\int_Y |\nabla w|^2\biggr)^{1/2} \leq \frac{1}{\lambda}\int_{\partial B_r \times Y} |\nabla w|^2 = \frac{1}{\lambda} \frac{dQ(r)}{dr},$$
where $\lambda > 0$ and $\lambda^2$ is the first eigenvalue of $Y$. Here we have used once again that $w$ is orthogonal to the constants on each fiber. Thus, the quantity $e^{-\lambda r}Q(r)$ is nondecreasing, but is also $O(r^{M}e^{-\lambda r})$ as $r \to \infty$ with $M = 2k+2\alpha+2m$ by what we said before, so that $Q(r) = 0$ for all $r$. \hfill $\Box$\medskip\

The desired statement now follows with $p = 2\underline{h}' + 2\ell'$.
\end{proof}

We conclude by stating the application of Theorem \ref{t:schauder} to this setting as a separate theorem.

\begin{theorem}\label{moron}
Let $(Y,g_Y, J_Y)$ be a compact K\"ahler manifold without boundary. Given $m \in \N$, equip $\C^m \times Y$ with the product K\"ahler structure $g_P = g_{\C^m} + g_Y$ and $J_P = J_{\C^m} + J_Y$. Then for all $k \in \N_{\geq 1}$ and $\alpha\in (0,1)$ there exists a constant $C_k = C_k(\alpha)$ such that for all $p \in \C^m \times Y$ and $0 <\rho<R$,
\begin{align}\label{e:moron}
[\nabla^{k,g_P}\eta]_{C^{\alpha}(B^{g_P}(p,\rho))} \leq C_k([\nabla^{k-1,g_P}\delta^{g_P}\eta]_{C^{\alpha}(B^{g_P}(p,R))} +(R-\rho)^{-k-\alpha}\|\eta\|_{L^\infty(B^{g_P}(p,R))} )
\end{align}
for all real $2$-forms $\eta \in C^{k,\alpha}_{\rm loc}(B^{g_P}(p,2R))$ that are $i\partial\ov{\partial}$-exact with respect to $J_P$. Here $\delta^{g_P}$ denotes the formal adjoint of $d$ with respect to $g_P$, and the H\"older seminorms are the ones of Definition \ref{d:holder}.
\end{theorem}

\subsection{A Schauder estimate for scalar functions}
We now use the abstract Schauder Theorem \ref{t:schauder} to derive a Schauder estimate for scalar functions on cylinders.

\begin{theorem}\label{scalar}
Let $(Y,g_Y)$ be a compact Riemannian manifold without boundary. Given any $d \in \N$, equip $\R^d \times Y$ with the product Riemannian metric $g_P = g_{\R^d} + g_Y$. Then for all $k \in \N_{\geq 2}$ and $\alpha \in (0,1)$ there exists a constant $C_k = C_k(\alpha)$ such that for all $p \in \R^d \times Y$ and $0 <\rho<R$,
\begin{align}\label{fruitless}
[\nabla^{k,g_P} f]_{C^{\alpha}(B^{g_P}(p,\rho))} \leq C_k([\nabla^{k-2,g_P}\Delta^{g_P}f]_{C^{\alpha}(B^{g_P}(p,R))} +(R-\rho)^{-k-\alpha}\|f\|_{L^\infty(B^{g_P}(p,R))} )
\end{align}
for all scalar functions $f \in C^{k,\alpha}_{\rm loc}(B^{g_P}(p,2R))$.
\end{theorem}

In fact, thanks to \cite[Prop 3.2]{He2}, Theorem \ref{scalar} implies Theorem \ref{moron}  (even a version  for $k = 0$, with $\nabla^{-1}\delta = \rm{tr}$). Nevertheless, we find it valuable to have the abstract Theorem \ref{t:schauder} and derive Theorem \ref{moron} directly from it, since this is the right direction for what we will need to do in Section \ref{nonpr}. There, it will not help to work at the level of K\"ahler potentials because the complex structure is not a product, and part of the proof of Theorem \ref{t:fornow} will closely parallel the proof of Theorem \ref{t:schauder}.

\begin{proof}[Proof of Theorem \ref{scalar}]
This follows from Theorem \ref{t:schauder} applied to the presheaf $\mathcal{S}$ of $d$-exact real $1$-forms of class $C^{k-1,\alpha}_{\rm loc}$ on $\mathbb{R}^d\times Y$ once we verify that $\mathcal{S}$ satisfies Assumptions (1)--(2). Indeed, we can apply Theorem \ref{t:schauder} to $\eta = df$ with $\rho, R$ replaced by $\rho, \ti{R}=\rho+\frac{1}{2}(R-\rho)$ to obtain
$$[\nabla^{k,g_P} f]_{C^{\alpha}(B^{g_P}(p,\rho))} \leq C_k([\nabla^{k-2,g_P}\Delta^{g_P}f]_{C^{\alpha}(B^{g_P}(p,\ti{R}))} +(R-\rho)^{-k+1-\alpha}\|df\|_{L^\infty(B^{g_P}(p,\ti{R}))} ),$$
and then use Lemma \ref{l:higher-interpol} with $\rho,R$ replaced by $\ti{R}, \ti{R}+\kappa(R-\rho)$ ($\kappa \in (0,\frac{1}{2})$) to estimate
$$\kappa(R-\rho)\|df\|_{L^\infty(B^{g_P}(p,\ti{R}))}\leq C_k(\kappa^{k+\alpha}(R-\rho)^{k+\alpha}[\nabla^{k,g_P} f]_{C^{\alpha}(B^{g_P}(p,R))}+\|f\|_{L^\infty(B^{g_P}(p,R))}).$$
Combining these two estimates yields
\[\begin{split}
[\nabla^{k,g_P} f]_{C^{\alpha}(B^{g_P}(p,\rho))} &\leq C_k\kappa^{k-1+\alpha}[\nabla^{k,g_P} f]_{C^{\alpha}(B^{g_P}(p,R))}\\
&+ C_k([\nabla^{k-2,g_P}\Delta^{g_P}f]_{C^{\alpha}(B^{g_P}(p,R))} +\kappa^{-1}(R-\rho)^{-k-\alpha}\|f\|_{L^\infty(B^{g_P}(p,R))} ).
\end{split}\]
If we fix $\kappa \in (0,\frac{1}{2})$ such that $C_k\kappa^{k-1+\alpha}\leq\frac{1}{2}$, then an application of Lemma \ref{l:iterate} gives \eqref{fruitless}.

It remains to verify that $\mathcal{S}$ satisfies Assumptions (1)--(2) of Theorem \ref{t:schauder}.

For (1), let $U_i$ be an exhaustion of $\mathbb{R}^d \times Y$ by open sets. Let $\eta_i \in \mathcal{S}(U_i)$ converge to $\eta_\infty \in C^\infty_{\rm loc}(\mathbb{R}^d \times Y)$ in the $C$\begin{small}$^{k-1,\beta}_{\rm loc}$\end{small} topology for some $\beta<\alpha$. Write $\eta_i=df_i$ with $f_i(0,p)=0$ for some fixed $p$ (we may assume that $(0,p)\in U_i$ for all $i$). Using the fundamental theorem of calculus, we see that the functions $f_i$ are uniformly bounded in $C^{k,\beta}_{\rm loc}$, so passing to a subsequence they converge to $f_\infty\in C^{k,\beta}_{\rm loc}(\R^d\times Y)$ and $\eta_\infty=d f_\infty$. It follows that $f_\infty$ is in fact smooth, and so $\eta_\infty\in\mathcal{S}(\R^d\times Y)$.

For (2), let $\eta\in \mathcal{S}(\R^d\times Y)$ with $|\nabla^{k-1,g_P}\eta|_{g_P}=O(r^\alpha)$ and $\nabla^{k-2,g_P}L^{g_P}\eta$ parallel. Then $\eta=df$ and $L^{g_P}\eta=\delta^{g_P}df=\Delta^{g_P}f,$ so $\Delta^{g_P}f=\ell$, where $\ell$ is a polynomial of degree $\leq k-2$ on $\R^d$. In particular, $f$ is smooth and $f=h+\ell'$, where $h$ is harmonic on $\R^d\times Y$ and $\ell'$ is a polynomial of degree $\leq k$ on $\R^d$. This gives $|\nabla^{k-1,g_P} dh|_{g_P}=O(r^\alpha)$. Decomposing $h=\underline{h}+(h-\underline{h})$, where $\underline{h}$ denotes the smooth function on $\R^d$ obtaining by averaging $h$ on each fiber, we see that $\underline{h}$ is harmonic and $|\nabla^{k-1} d\underline{h}|=O(r^\alpha)$. By the Liouville theorem in $\R^d$ we obtain that $\underline{h}$ is a polynomial of degree $\leq k$.

A similar argument as in Claim 2 in the proof of Proposition \ref{check} then shows that $w=h-\underline{h}=0$. Indeed, $w$ is harmonic on $\R^d\times Y$ with $|\nabla^{k-1,g_P} dw|_{g_P}=O(r^{\alpha})$, so  $|w|=O(r^{k+\alpha})$, and $|\nabla^{i,g_P} w|_{g_P}=O(r^{k+\alpha})$ for all $i\in\N$ by local estimates for harmonic functions, so we may conclude as before that $Q(r)=\int_{B_r\times Y}|\nabla^{g_P}w|^2_{g_P}$ is identically zero. We thus obtain that $\nabla^{k,g_P} df=0$ as desired.
\end{proof}

\begin{rk}
It is an interesting problem for future study to find other geometrically meaningful presheaves $\mathcal{S}$ of forms to which Theorem \ref{t:schauder} applies.
\end{rk}

\section{Higher order estimates in the product case}\label{sekt}
In this section we prove Theorem \ref{mainlocal}.

From now on, let $\omega_t=\omega_{\mathbb{C}^m}+e^{-t}\omega_Y$. This is a product K\"ahler form on $B\times Y$ uniformly equivalent to $\omega^\bullet_t$ (independent of $t$). Its Chern connection is independent of $t$ and equals the Chern connection of $\omega_P = \omega_0=\omega_{\mathbb{C}^m}+\omega_Y$. Given any $k\geq 0$, we aim to show that
\begin{align}\label{strongk}
\sup_{ B \times Y} \mu_{k,t} \leq C_k
\end{align}
independent of $t$, where for all $x \in B \times Y$ we define
\begin{align}
\mu_{k,t}(x) = d^{g_t} (x,\partial(B \times Y))^k |(\nabla^{k, g_t}g_t^\bullet)(x)|_{g_t(x)}.
\end{align}
To prove \eqref{strongk} we proceed by induction on $k$, the case $k=0$ being our assumed $C^0$ bound  \eqref{unifequiv} on the collapsing metric.
By induction we may assume that $k \geq 1$ and that
\begin{equation}\label{indukt}
\sum_{j=1}^{k-1}\mu_{j,t}\leq C_k.
\end{equation}
If  \eqref{strongk} does not hold, then
$\limsup_{t\to\infty} \sup_{B \times Y} \mu_{k,t}=\infty$. For simplicity of notation, we will not pass to subsequences, and will instead assume that $\lim_{t\to\infty} \sup_{B \times Y} \mu_{k,t}=\infty$. Choose points $x_t\in B\times Y$ such that the supremum of $\mu_{k,t}$ is achieved at $x_t$, and define real numbers $\lambda_t$ by
$$\lambda_t^k=|(\nabla^{k,g_t}g^\bullet_t)(x_t)|_{g_t(x_t)}.$$
Then $\lambda_t \to \infty$ as $t\to\infty$ because otherwise $\mu_{k,t}$ would be uniformly bounded (since the $g_t$-diameter of $B\times Y$ is uniformly bounded).
Consider then the biholomorphisms
$$\Psi_t:B_{\lambda_t}\times Y\to B\times Y, \;\, \Psi_t(z,y)=(\lambda_t^{-1}z,y),$$
and let
$$\hat{g}^\bullet_t=\lambda_t^{2}\Psi_t^*g^\bullet_t,\;\, \hat{g}_t=\lambda_t^{2}\Psi_t^*g_t=g_{\mathbb{C}^m}+\lambda_t^{2}e^{-t}g_Y,\;\, \hat{x}_t=\Psi_t^{-1}(x_t).$$ We know from \eqref{unifequiv} that on $B_{\lambda_t}\times Y$ we have
\begin{equation}\label{caz}
C^{-1}\hat{g}_t\leq \hat{g}^\bullet_t\leq C\hat{g}_t.
\end{equation}
Note also that $\nabla^{\hat{g}_t}=\nabla^{g_P}$ and
$$\lambda_t^k=|(\nabla^{k,g_t} g^\bullet_t)(x_t)|_{g_t(x_t)}=\lambda_t^{-2}|(\nabla^{k,\hat{g}_t} \hat{g}^\bullet_t)(\hat{x}_t)|_{\Psi_t^*g_t(\hat{x}_t)}=\lambda_t^k
|(\nabla^{k,\hat{g}_t} \hat{g}^\bullet_t)(\hat{x}_t)|_{\hat{g}_t(\hat{x}_t)}.$$
We have therefore showed that
\begin{equation}\label{blow2}
|(\nabla^{k,\hat{g}_t} \hat{g}^\bullet_t)(\hat{x}_t)|_{\hat{g}_t(\hat{x}_t)}=1.
\end{equation}
Also for all $0<j\leq k$ and for all $\hat{x}\in B_{\lambda_t}\times Y$ we have
\begin{equation}\label{stupid}
\mu_{j,t}(\Psi_t(\hat{x}))=d^{\hat{g}_t} (\hat{x},\partial(B_{\lambda_t} \times Y))^j |(\nabla^{j, \hat{g}_t}\hat{g}^\bullet_t)(\hat{x})|_{\hat{g}_t(\hat{x}_t)}.
\end{equation}
Using \eqref{indukt} this implies that for $0<j<k$ (this is vacuous for $k = 1$) we have
\begin{equation}\label{blow3}
\sup_{B_{\frac{1}{2}\lambda_t}\times Y}|\nabla^{j,\hat{g}_t} \hat{g}^\bullet_t|_{\hat{g}_t}=\lambda_t^{-j}\sup_{B_{\frac{1}{2}}\times Y}|\nabla^{j,g_t} g^\bullet_t|_{g_t}\leq C\lambda_t^{-j} \to 0.
\end{equation}
Using similar arguments, we can also show the following properties, which are crucial for passing to a pointed limit with basepoint $\hat{x}_t$. First, for all $\hat{x} \in B_{\lambda_t} \times Y$,
\begin{align}\label{e:pretty_useful}
\mu_{k,t}(\Psi_t(\hat{x})) = d^{\hat{g}_t} (\hat{x},\partial(B_{\lambda_t} \times Y))^k |(\nabla^{k,\hat{g}_t}\hat{g}^\bullet_t)(\hat{x})|_{\hat{g}_t(\hat{x})}.
\end{align}
Using \eqref{blow2} and the fact that $\sup \mu_{k,t} =  \mu_{k,t}(x_t) = \mu_t(\Psi_t(\hat{x}_t)) \to \infty$, we obtain that
\begin{align}\label{fuck_the_boundary}
d^{\hat{g}_t} (\hat{x}_t, \partial(B_{\lambda_t} \times Y)) \to \infty.
\end{align}
This tells us that if we pass to a pointed limit with basepoint $\hat{x}_t$, then the boundary moves away to infinity and the limit space will be complete. Moreover, using the fact that the quantity in \eqref{e:pretty_useful} is maximized at $\hat{x} = \hat{x}_t$ (and the triangle inequality), we learn that for all $\hat{x} \in B_{\lambda_t} \times Y$,
\begin{align}\label{saves_our_asses}
 |(\nabla^{k,\hat{g}_t}\hat{g}^\bullet_t)(\hat{x})|_{\hat{g}_t(\hat{x})} \leq  \frac{d^{\hat{g}_t} (\hat{x}_t,\partial(B_{\lambda_t} \times Y))^k}{d^{\hat{g}_t} (\hat{x},\partial(B_{\lambda_t} \times Y))^k}
\leq \left(1-\frac{d^{\hat{g}_t}(\hat{x}_t, \hat{x})}{d^{\hat{g}_t}(\hat{x}_t,\partial(B_{\lambda_t} \times Y) )}\right)^{-k}.
\end{align}
Thus we have a uniform upper bound on $|\nabla^{k,\hat{g}_t} \hat{g}^\bullet_t|_{\hat{g}_t}$ on $\hat{g}_t$-balls of fixed radii centered at $\hat{x}_t$. We can therefore study the possible complete pointed limit spaces $(B_{\lambda_t}\times Y, \hat{g}_t,\hat{x}_t)$ as $t\to\infty$. Up to passing to a subsequence, and modulo translations in the $\mathbb{C}^m$ factor, we may assume that $\hat{x}_t = (0,\hat{y}_t)\in \mathbb{C}^m\times Y$ and that  $\hat{y}_t\to y_\infty\in Y$.
Define $$\delta_t=\lambda_te^{-\frac{t}{2}}.$$
Up to passing to a subsequence once again, we then need to consider three cases according to whether (1) $\delta_t\to\infty$, (2) $\delta_t$ remains uniformly bounded away from zero and infinity (without loss converging to some $\delta > 0$, and again without loss $\delta = 1$), or (3) $\delta_t\to 0$.

\subsection{Case 1: the blowup is $\C^{m+n}$}
Here we assume that $\delta_t\to\infty$.
 We fix a chart on $Y$ centered at $y_\infty$ given by the ball $B_2\subset\mathbb{C}^n$, and in the induced product coordinates on $\mathbb{C}^m\times B_2$ we pull back $\hat{g}^\bullet_t$ by the biholomorphism $(z,y)\mapsto (z,\hat{y}_t+\delta_t^{-1}y)$, defined on $\mathbb{C}^m\times B_{\delta_t}$ with image $\C^m\times B_1(\hat{y}_t)\subset\C^m\times B_2$. After this pullback, the new Ricci-flat metrics are uniformly equivalent to Euclidean thanks to \eqref{caz}, and their $k$-th covariant derivative (with respect to the pullback of $\hat{g}_t$) at the origin has norm $1$ (also measured with respect to the pullback of $\hat{g}_t$). Thanks to Proposition \ref{localest2}, these metrics have uniform $C^\infty$ bounds on compact subsets, and so a subsequence converges locally smoothly to a limit Ricci-flat K\"ahler metric on $\mathbb{C}^{m+n}$ which is uniformly equivalent to Euclidean and is not constant. If $k = 1$, this is impossible because of the Liouville Theorem \ref{lio1}. If $k > 1$, it immediately contradicts \eqref{blow3}.

\subsection{Case 2: the blowup is $\C^m \times Y$} In this case we have that $\delta_t\to 1$, without loss of generality.
Then the metrics $\hat{g}_t$ converge smoothly uniformly on compact sets of $\mathbb{C}^m\times Y$ to the product metric $g_P$.
Thanks to \eqref{caz}, we can apply Proposition \ref{localest2} and obtain local uniform $C^\infty$ bounds for $\hat{g}^\bullet_t$.
By Ascoli-Arzel\`a (again up to passing to a subsequence $t_i\to\infty$) we have that $\hat{g}^\bullet_t$ converges smoothly to a Ricci-flat K\"ahler metric $\hat{g}^\bullet_\infty$ on $\mathbb{C}^m\times Y$, which is uniformly equivalent to $g_P$ and satisfies
\begin{equation}\label{uno}
\nabla^{j,g_P}\hat{g}^\bullet_\infty = 0 \;\,(0 < j < k), \;\, \sup\nolimits_{\mathbb{C}^m\times Y}|\nabla^{k,g_P} \hat{g}^\bullet_\infty|_{g_P}=1.
\end{equation}
If $k = 1$, then the Liouville Theorem \ref{lioHein} tells us that $\hat{g}^\bullet_\infty$ is parallel with respect to $g_P$, contradicting \eqref{uno}. If $k > 1$, then \eqref{uno} is again plainly self-contradictory without any Liouville theorems.

\subsection{Case 3: the blowup is $\C^m$} Here we finally assume that $\delta_t\to 0$. For any fixed radius $R > 0$, we have thanks to \eqref{caz}, \eqref{blow3}, \eqref{saves_our_asses} that for all $t \geq 0$,
\begin{align}\label{whenwillthiseverend}
\|\hat{g}^\bullet_t\|_{C^{k}(B_R \times Y,\hat{g}_t)}\leq C(R).
\end{align}
This trivially implies a uniform $C^k(B_R \times Y,g_P)$ bound. Thus, by Ascoli-Arzel\`a, for all  $\alpha \in (0,1)$ the metrics $\hat{g}^\bullet_t$ converge in $C$\begin{small}$^{k-1,\alpha}_{\rm loc}$\end{small} (modulo subsequences) to a $C$\begin{small}$^{k-1,\alpha}_{\rm loc}$\end{small} tensor $\hat{g}^\bullet_\infty$ on $\mathbb{C}^m\times Y$ which is the pullback of a $C$\begin{small}$^{k-1,\alpha}_{\rm loc}$\end{small} K\"ahler metric $\hat{g}^\bullet_\infty$ on $\mathbb{C}^m$, uniformly equivalent to Euclidean. Note that for $k = 1$ the $C^\alpha_{\rm loc}$ form $\hat\omega_\infty^\bullet$ is only weakly closed, but this suffices for our subsequent discussion.

Pulling back the complex Monge-Amp\`ere equation \eqref{mageneral} we obtain
\begin{equation}\label{maa}
(\hat{\omega}^\bullet_t)^{m+n}=\lambda_t^{2m+2n} \Psi_t^*(e^F(\omega_{\mathbb{C}^m}+e^{-t}\omega_Y)^{m+n})=\delta_t^{2n} e^{F_t}\omega_P^{m+n},
\end{equation}
where $F_t=F\circ\Psi_t$ is a pluriharmonic function which depends only on $z\in\mathbb{C}^m$.
Note that if $x_t \to x_\infty$ as $t\to\infty$ (as it has to, up to passing to a subsequence, after slightly shrinking the original tube $B \times Y$), then
$F_t$ converges to the constant $F(z_\infty)$ smoothly on each bounded cylinder $B_R \times Y$.

A contradiction will be derived in three steps. Using some arguments from \cite{To}, what we have said so far is sufficient to conclude that $(\hat\omega_\infty^\bullet)^m = c \omega_{\C^m}^m$ for some constant $c > 0$. Then standard arguments show that $\hat\omega_\infty^\bullet$ is smooth and hence constant by Theorem \ref{lio1}. All of this will be explained and proved in Claim 3 below (cf. Section \ref{claim3product}). In fact, for $k > 1$ the conclusion that $\hat\omega_\infty^\bullet$ is constant is actually obvious from \eqref{blow3}. The main difficulty is to derive from \eqref{blow2} that $\hat\omega_\infty^\bullet$ is \emph{not} constant.

As usual, one key step towards this goal is to slightly improve the given regularity \eqref{whenwillthiseverend} of $g_t^\bullet$. In fact, by linearizing the Monge-Amp\`ere equation \eqref{maa} at the product model metric $\hat\omega_t$ and bringing in the linear Schauder theory of Section \ref{linear}, we will prove in Claim 1 (Section \ref{claim1product}) that the $C^k(B_R  \times Y, \hat{g}_t)$ norm in \eqref{whenwillthiseverend} can be replaced by the $C^{k,\alpha}(B_R \times Y,\hat{g}_t)$ norm for any $0 < \alpha < 1$. Now typically such an upgrade would be sufficient to pass to a limit in \eqref{blow2} thanks to the Ascoli-Arzel\`a theorem. However, there seems to be no obvious version of Ascoli-Arzel\`a (for tensors, with respect to {collapsing} metrics) that accomplishes this immediately. In Claim 2 (Section \ref{claim2product}) we will exploit the K\"ahler property of $\hat\omega_t^\bullet$ to deduce from Claim 1 that $\hat\omega_\infty^\bullet$ is not constant, by using  Lemma \ref{l:braindead} to show that the equality in \eqref{blow2} is essentially already attained by the ``all base'' component of the tensor $\nabla^{k,\hat{g}_t}\hat{g}_t^\bullet$.

\subsubsection{}\label{claim1product}

{\bf Claim 1.} For all $\alpha \in (0,1)$ there exists an $\epsilon > 0$ and a $C > 0$ such that for all $t \geq 0$,
\begin{equation}\label{extra}
\|\hat{g}^\bullet_t\|_{C^{k,\alpha}(B_\epsilon \times Y,\hat{g}_t)}\leq C.
\end{equation}
\vskip2mm

Having this strengthening of \eqref{whenwillthiseverend} only for one particular value $R = \epsilon$ suffices for our purposes.  In fact, the case of a general $R$ could be proved along similar lines using also a covering argument, but we choose to omit these additional arguments in order not to clutter notation. \medskip\

\noindent \emph{Proof of Claim 1}. For simplicity, let us change notation by viewing $\delta=\lambda_te^{-\frac{t}{2}}\to 0$ as a new parameter replacing $t$. Replacing $t$ with a suitable sequence $t_i\to\infty$, we may assume that $\delta$ is a strictly decreasing function of $t$, so that we can (and will) effectively view $t$ as a function of $\delta$ as well.

Introduce a new stretching map
\begin{equation*}\Phi_\delta:B_{\delta^{-1}}\times Y\to B_1 \times Y, \;\, \Phi_\delta(z,y)=(\delta z,y),\end{equation*}
as well as new scaled and stretched metrics
\begin{align*}\ti{g}^\bullet_\delta &=\delta^{-2}\Phi_\delta^*\hat{g}^\bullet_t,\\
\tilde{g}_\delta &= \delta^{-2}\Phi_\delta^*\hat{g}_t=g_P.\end{align*}
Then \eqref{caz}, \eqref{blow3}, \eqref{saves_our_asses} imply that on $B_{\delta^{-1}}\times Y$,
\begin{align}\label{cn1}
C^{-1}g_P\leq \ti{g}^\bullet_\delta\leq C g_P,\\
\label{cn3}
\forall j \in \{1,\ldots,k-1\}: |\nabla^{j,g_P}\ti{g}^\bullet_\delta|_{g_P}\leq C\delta^{j}\lambda_t^{-j},\\
\label{cn2}
|\nabla^{k,g_P}\ti{g}^\bullet_\delta|_{g_P}\leq C\delta^{k}.
\end{align}
The basic idea is that \eqref{cn1}, \eqref{cn3},  \eqref{cn2} allow us to linearize the Monge-Amp\`ere equation satisfied by $\tilde\omega_\delta^\bullet$ at the product metric $\omega_P$ on the whole cylinder $B_{\delta^{-1}} \times Y$, and to use the linear Schauder theory of Section \ref{linear} to improve the regularity of $\tilde\omega_\delta^\bullet$. However, \eqref{cn1}, \eqref{cn3}, \eqref{cn2} are still compatible with $\tilde\omega_\delta^\bullet$ approaching some K\"ahler form at bounded but nonzero distance to $\omega_P$. As a preliminary step, we show that the true limit of $\tilde\omega_\delta^\bullet$ differs from $\omega_P$ at worst by some harmless automorphism of $\C^m$.

Indeed, thanks to \eqref{cn1} we can apply Proposition \ref{localest2} and obtain uniform $C^\infty$ bounds for the Ricci-flat metrics $\ti{g}^\bullet_\delta$ on compact sets, so that after passing to a subsequence they converge locally smoothly to a limiting K\"ahler metric $\ti{g}^\bullet_\infty$ on $\C^m \times Y$, which thanks to \eqref{cn2} (if $k=1$) and \eqref{cn3} (if $k>1$) is parallel with respect to $g_P$. Recall that $\ti{\omega}^\bullet_\delta$ and $\omega_P$ are $\ddbar$-cohomologous. Thus,  by Proposition \ref{p:closure} and \eqref{cn1}, \eqref{cn3}, \eqref{cn2}, $\ti{\omega}^\bullet_\infty=\omega_P+\eta_\infty$, where $\eta_\infty$ is $\ddbar$-exact, of bounded $g_P$-norm, and $g_P$-parallel. Proposition \ref{check} applied with $k = 0$ then tells us that $\eta_\infty=\ddbar p$ for some quadratic polynomial $p$ on $\C^m$. (Technically Proposition \ref{check} only applies for $k \geq 1$ because in Section 3.3 we have set $k \geq 1$ by definition, but the proof goes through without any changes for $k = 0$.) It follows that $\ti{\omega}^\bullet_\infty$ differs from $\omega_P$ by a linear automorphism of $\C^m$. There is no reason to expect this automorphism to be ${\rm Id}_{\C^m}$, but we can simply pull back our whole setup by the same automorphism and in this way assume without loss that $\tilde{g}_\delta^\bullet \to g_P$. (A much more technical version of this argument will appear at the analogous stage in Section \ref{nonpr}; cf. the construction of a modified reference metric $\tilde\omega_t^\sharp$ after \eqref{caz699}.)

Pulling back \eqref{maa} by $\Phi_\delta$ we obtain
\begin{equation}\label{maa2}
(\ti{\omega}^\bullet_\delta)^{m+n}=e^{\ti{F}_\delta}\omega_P^{m+n},
\end{equation}
where $\ti{F}_\delta$ is the pluriharmonic function on $B_{\delta^{-1}}$ defined by
$$\ti{F}_\delta(z)=F_t(\delta z)=F(z_t + \lambda_t^{-1}\delta z).$$
Note that $\tilde{F}_\delta \to F(z_\infty)$ smoothly on compact subsets of $\mathbb{C}^m$ and that for all $k \geq 1$,
\begin{equation}\label{e:pitchfork}
\de_z^k \ti{F}_\delta=O(\delta^k)
\end{equation}
uniformly on the entire cylinder $B_{\delta^{-1}} \times Y$. Dropping the subscript $\delta$, let us write $\ti{\omega}^\bullet_\delta = \omega_P + \eta$, where $\eta$ is $i\de\db$-exact. Let us also agree that all  metric operations and (semi-)norms in the rest of this proof are the ones associated with $g_P$, and that all balls are $g_P$-geodesic balls centered at $\hat{x}_t = (0, \hat{y}_t)$. By linearizing \eqref{maa2} and bringing in Theorem \ref{moron} we will prove that
\begin{equation}\label{extra4}
[\nabla^{k}\eta]_{C^{\alpha}(B_{\ve \delta^{-1}})}\leq C \delta^{k+\alpha}
\end{equation}
for some uniform constant $\ve \in (0,1)$. This clearly implies \eqref{extra} up to renaming $\epsilon$.

To prove \eqref{extra4}, we first note that it follows from \eqref{maa2} that
\begin{align}
\label{e:ell2} \delta^{g_P} \eta &= \sum_{i=1}^{m+n-1} \underbrace{(\eta  \circledast \cdots \circledast \eta)}_{i \;{\rm factors}} \circledast \nabla \eta+\nabla e^{\ti{F}_\delta}.
\end{align}
Here $\circledast$ denotes a tensorial contraction that may also involve the metric $g_P$. This is easy to derive from the usual identities expressing $\ast\eta$ in terms of $\omega_P$ and $\eta \wedge \omega_P^{m+n-1}$, the latter being controlled thanks to \eqref{maa2}. Indeed, a standard calculation (cf. \cite[Prop 1.2.31]{Huy}) gives
\begin{equation*}\label{hodgestar}
*\eta=\left(\frac{1}{(m+n-1)!}+\frac{1}{(m+n-2)!}\right)\frac{\eta\wedge\omega_P^{m+n-1}}{\omega_P^{m+n}}\omega_P^{m+n-1}-\frac{1}{(m+n-2)!}\eta\wedge\omega_P^{m+n-2},
\end{equation*}
and therefore
\begin{equation*}\label{hodgestar2}
d(\ast\eta)=\left(\frac{1}{(m+n-1)!}+\frac{1}{(m+n-2)!}\right)d\left(\frac{\eta\wedge\omega_P^{m+n-1}}{\omega_P^{m+n}}\right) \wedge \omega_P^{m+n-1},
\end{equation*}
while from the Monge-Amp\`ere equation \eqref{maa2} we have
$$(m+n)\eta\wedge\omega_P^{m+n-1}=(e^{\ti{F}_\delta}-1)\omega_P^{m+n}-\sum_{i=2}^{m+n}\binom{m+n}{i}\eta^i\wedge\omega_P^{m+n-i},$$
proving \eqref{e:ell2}. Differentiating \eqref{e:ell2} $k-1$ times, we get
\begin{align}\label{e:ell4}
\nabla^{k-1}\delta^{g_P}\eta = \sum_{i=1}^{m+n-1} \biggl(\sum \nabla^{k_1}\eta \circledast \cdots \circledast \nabla^{k_{i+1}}\eta\biggr)  + \nabla^k e^{\ti{F}_\delta},
\end{align}
where the inner sum runs over all $(k_1, \ldots, k_{i+1}) \in \N^{i+1}$ such that $k_1 + \cdots + k_{i+1} = k$. Let $\ve\in (0,1)$ to be determined. Then for any $0<\rho< R \leq (2\delta)^{-1},$ using \eqref{e:ell4} and \eqref{e:moron},
\begin{align}\label{e:final2}
[\nabla^k\eta]_{C^{\alpha}(B_{\ve\rho})} \leq C \sum_{i=1}^{m+n-1} \biggl(\sum [\nabla^{k_1} \eta]_{C^{\alpha}(B_{\epsilon R})}  \|\nabla^{k_2}\eta \|_{L^\infty(B_{\epsilon R})}  \cdots  \| \nabla^{k_{i+1}}\eta\|_{L^\infty(B_{\epsilon R})}\biggr)\\
\label{e:final22}+\; C[\nabla^k e^{\ti{F}_\delta}]_{C^{\alpha}(B_{\epsilon R})} + C(\epsilon(R-\rho))^{-k-\alpha}\|\eta\|_{L^\infty(B_{\epsilon R})}.
\end{align}
The inner sum in \eqref{e:final2} again runs over all $(k_1, \ldots, k_{i+1}) \in \N^{i+1}$ with $k_1+ \cdots + k_{i+1} = k$.

To proceed, we begin by estimating
$$[\nabla^{k_1}\eta]_{C^\alpha(B_{\epsilon R})} \leq (2\epsilon R)^{1-\alpha}\|\nabla^{k_1 + 1}\eta\|_{L^\infty(B_{2\epsilon R})}$$
except when $k_1 = k$ and $k_2 = \ldots = k_{i+1} = 0$; in the latter case, we keep the $[\nabla^{k_1}\eta]$ term in \eqref{e:final2} as is. The $\|\nabla^{k_1+1}\eta\|$ terms thus introduced for $k_1 \leq k-1$, as well as the $\|\nabla^{k_j}\eta\|$ terms appearing in \eqref{e:final2} for $2 \leq j \leq i+1$ and $k_j \geq 1$, can be bounded by the appropriate power of $\delta$ using \eqref{cn3}, \eqref{cn2}. This leaves us with those terms of \eqref{e:final2} where $2 \leq j \leq i+1$ and $k_j = 0$. To control these terms, and also the $\|\eta\|$ term in \eqref{e:final22}, we use the following idea: because $\tilde{g}_\delta^\bullet \to g_P$ smoothly on compact sets, we may assume that $|\eta(\hat{x}_t)| \leq \epsilon$; and combining this with \eqref{cn3}, \eqref{cn2} we get that $|\eta| \leq 2\epsilon$ on $B_{\epsilon\delta^{-1}}$. Lastly, the $[\nabla^ke$\begin{small}$^{\ti{F}_\delta}$\end{small}$]$ term in \eqref{e:final22} is easily controlled using \eqref{e:pitchfork}. The upshot of all of this is that
\begin{align*}\label{e:final3}
[\nabla^k\eta]_{C^{\alpha}(B_{\ve\rho})}
\leqslant  C\epsilon [\nabla^k\eta]_{C^{\alpha}(B_{\ve R} )}+ C\Psi(\epsilon,\delta)+ C \delta^{k+\alpha} +C(\epsilon(R-\rho))^{-k-\alpha}\epsilon,\\
\Psi(\epsilon,\delta) = \max_{\substack{1\leq i \leq m+n-1\\k_1 \leq k-1, k_1 + \cdots + k_{i+1} = k}} \epsilon^{1-\alpha}\delta^{k_1+\alpha} \delta^{\sum_{j\geq 2: k_j\geq 1}k_j} \epsilon^{\#\{j\geq 2: k_j = 0\}} = O(\delta^{k+\alpha}).
\end{align*}
Fixing $\ve=(2C)^{-1}$ and applying Lemma \ref{l:iterate}, we get $[\nabla^k \eta]_{C^{\alpha}(B_{\ve\rho} )} \leqslant  C \delta^{k+\alpha} + C(R-\rho)^{-k-\alpha} $. Finally, choosing $R=(2\delta)^{-1}$ and $\rho=R/2$, and renaming $\epsilon$, we get \eqref{extra4}.

This completes the proof of Claim 1.

\begin{rk}\label{discuss}
In the proof of Claim 1 it was important that the decay rate of the term $\nabla e^{\ti{F}_\delta}$ in \eqref{e:ell2} improved by a factor of $\delta$ upon differentiation, arbitrarily many times. In the setting of Corollary \ref{main2}, if the smooth fibers $X_b$ are neither flat nor biholomorphic to each other, one encounters new terms on the right-hand side of the analog of \eqref{e:ell2} that do not improve in this way upon fiber differentiation. (Roughly speaking, these terms are due to the variation of complex structure of the fibers.) In fact, in this case one cannot expect to obtain uniform $C^{k,\alpha}$ estimates for the Ricci-flat metrics $\omega^\bullet_t$ with respect to a shrinking family of product metrics even for $k = 0$; see Remark \ref{totalfuckage}.
\end{rk}

\subsubsection{}\label{claim2product} {\bf Claim 2.} We have
\begin{equation}\label{absd}
|\nabla^{k,\mathbb{C}^m}\hat{g}^\bullet_\infty(0)|_{g_{\mathbb{C}^m}} = 1.
\end{equation}

\noindent \emph{Proof of Claim 2}. We will work in local holomorphic product coordinates. We will denote any complex $(1,0)$ ``base'' $\mathbb{C}^m$ direction by a subscript $\mathbf{b}$ and any complex $(1,0)$ ``fiber'' $Y$ direction by a subscript $\mathbf{f}$. Since $g_P$ is a Riemannian product metric, $g_P$-covariant derivatives in the base directions commute with $g_P$-covariant derivatives in all other directions. Using the product shape of $g_P$ and the fact that $\hat{g}^\bullet_t$ is K\"ahler with respect to the product complex structure, we also have that \begin{align}\nabla^{g_P}_{\mathbf{b}} (\hat{g}^\bullet_t)_{\mathbf{f}\ov{\mathbf{b}}}=\nabla^{g_P}_{\mathbf{f}} (\hat{g}^\bullet_t)_{\mathbf{b}\ov{\mathbf{b}}},\;\,
\nabla^{g_P}_{\mathbf{b}} (\hat{g}^\bullet_t)_{\mathbf{f}\ov{\mathbf{f}}}=\nabla^{g_P}_{\mathbf{f}} (\hat{g}^\bullet_t)_{\mathbf{b}\ov{\mathbf{f}}}.\label{lakahlerite}\end{align}
From now on we will drop the superscript $g_P$ on covariant derivatives for simplicity.

To prove Claim 2, we first of all remark that by definition, the pointwise norm $|\nabla^{k}\hat{g}^\bullet_t|_{\hat{g}_t}$ is uniformly equivalent (with constants independent of $t$) to
\begin{equation}\label{somma}
\sum_{j=0}^{k+2} |\nabla^k \hat{g}^\bullet_t\{j\}|\delta_t^{-j},
\end{equation}
where $\nabla^k \hat{g}^\bullet_t\{j\}$ contains all the components of $\nabla^k\hat{g}^\bullet_t$ with $j$ fiber indices and $k+2-j$ base indices and the absolute value signs indicate the length of this tensor in our fixed coordinate system.  The goal is to show that all the terms with $j>0$ in \eqref{somma} go to zero as $\delta_t\to 0$, so that only the terms with $j=0$ survive and \eqref{absd} follows from the fact that $|\nabla^{k}\hat{g}^\bullet_t(\hat{x}_t)|_{\hat{g}_t(\hat{x}_t)}=1$, together with the $C^{k,\alpha}(B_\epsilon \times Y, g_P)$ convergence trivially implied by Claim 1. The full strength of the $\hat{g}_t$-bound (as opposed to $g_P$-bound) of Claim 1 is precisely the key to proving that the $j > 0$ terms in \eqref{somma} go to zero.

Thus, let us consider any component of $\nabla^k\hat{g}^\bullet_t\{j\}$ with $j > 0$. To clarify, what this notation means is that we allow $k$ covariant derivatives of $\hat{g}_t^\bullet$ with respect to all possible combinations of $\mathbf{b}$, $\ov{\mathbf{b}}$, $\mathbf{f}$, and $\ov{\mathbf{f}}$ indices, where at least one of the derivative indices, or one of the tensor indices of $\hat{g}_t^\bullet$, is an ${\mathbf{f}}$ or $\ov{\mathbf{f}}$ one. If one of the \emph{derivative} indices is an ${\mathbf{f}}$ or $\ov{\mathbf{f}}$ one, we can commute the left-most index of this type past all the $\mathbf{b}$ and $\ov{\mathbf{b}}$ indices that precede it, obtaining a term of the form
\begin{equation}\label{easy}
\nabla_{\mathbf{f}}\nabla^{k-1}\hat{g}^\bullet_t,\;\,\nabla_{\ov{\mathbf{f}}}\nabla^{k-1}\hat{g}^\bullet_t.
\end{equation}
Otherwise all the derivative indices are of type $\mathbf{b}, \ov{\mathbf{b}}$. Since the absolute values in \eqref{somma} are invariant under complex conjugation, and since $\hat{g}_t^\bullet$ is Hermitian, we are left with terms of the form
\begin{align}\label{cockadoodledoo}
\nabla^{k-1}_{\mathbf{b},\ov{\mathbf{b}}}\left(\nabla_{\mathbf{b}}(\hat{g}_t^\bullet)_{\mathbf{f}\ov{\mathbf{b}}}\right), \;\,
\nabla^{k-1}_{\mathbf{b},\ov{\mathbf{b}}}\left(\nabla_{\mathbf{b}}(\hat{g}_t^\bullet)_{\mathbf{f}\ov{\mathbf{f}}}\right), \;\,
\nabla^{k-1}_{\mathbf{b},\ov{\mathbf{b}}}\left(\nabla_{\mathbf{b}}(\hat{g}_t^\bullet)_{\mathbf{b}\ov{\mathbf{f}}}\right) .
\end{align}
Here the notation $\nabla$\begin{small}$^{k-1}_{\mathbf{b},\ov{\mathbf{b}}}$\end{small} means that we allow any combination of $k-1$ derivatives in directions $\mathbf{b}$ or $\ov{\mathbf{b}}$. Using \eqref{lakahlerite}, the first two of these types can be reduced to \eqref{easy} by commuting the $\mathbf{f}$ derivative past the preceding $\mathbf{b}, \ov{\mathbf{b}}$ ones. Thus, in addition to \eqref{easy}, we need to deal with terms of the form
\begin{equation}\label{hard}
\nabla^k_{\mathbf{b},\ov{\mathbf{b}}}(\hat{g}_t^\bullet)_{\mathbf{b}\ov{\mathbf{f}}}.
\end{equation}

For terms of the form \eqref{easy}, we fix any $z \in \C^m$ with $|z| < \epsilon$ and apply Lemma \ref{l:braindead} on $\{z\}\times Y$ with metric $g_Y$. More precisely, for all $\ell\in\{0,1,\ldots, k+1\}$ we consider the section $$\sigma=\nabla^{k-1}\hat{g}^\bullet_t\{\ell\}|_{\{z\}\times Y}$$ of the bundle
$E=(T^*\mathbb{C}^m)^{\otimes (k+1-\ell)}\otimes (T^*Y)^{\otimes \ell}\otimes \C$ over $\{z\}\times Y$, equipped with the product metric $g_P$ induced by $g_{\C^m}$ and $g_Y$. In this way we obtain that
\begin{align*}
\sup\nolimits_{\{z\}\times Y}|\nabla_{\mathbf{f},\ov{\mathbf{f}}} (\nabla^{k-1}\hat{g}^\bullet_t\{\ell\})|_{g_P}&\leq C[\nabla_{\mathbf{f},\ov{\mathbf{f}}}(\nabla^{k-1}\hat{g}^\bullet_t\{\ell\})]_{C^{\alpha}(\{z\}\times Y,g_Y)}\leq C\delta_t^{\ell+1+\alpha}\|\hat{g}^\bullet_t\|_{C^{k,\alpha}(B_\epsilon\times Y,\hat{g}_t)},
\end{align*}
which is $O(\delta_t^{\ell+1+\alpha})$ thanks to \eqref{extra}. This shows that the original term $|\nabla^k\hat{g}_t^\bullet\{j\}|\delta_t^{-j}$ in \eqref{somma} that we had reduced to the form \eqref{easy} in fact decays like $\delta_t^{\alpha}$, and in particular goes to zero as $\delta_t\to 0$.

Terms of the form \eqref{hard} require a different argument. Recall that, by assumption, $$\hat{\omega}^\bullet_t=\hat{\omega}_t+\ddbar \vp$$
for some smooth $t$-dependent function $\vp$ on $B_1\times Y$. Denote by $\underline{\vp}$ the fiberwise $g_Y$-average of $\vp$, which is a smooth $t$-dependent function on $B_1$. Then we have a well-defined function
$$\psi = \nabla^k_{\mathbf{b},\ov{\mathbf{b}}}(\vp - \underline\vp): B_1 \times Y \to \C^{\binom{2m+k-1}{k}},$$
the complexified $k$-th derivative of $\vp - \underline\vp$ in the base directions. Clearly, for all $z \in B_1$,
$$\Delta^{g_Y}(\psi|_{\{z\}\times Y}) = \nabla^k_{\mathbf{b},\ov{\mathbf{b}}}(\tr{g_Y}{(\hat{g}_t^\bullet|_{\{z\}\times Y})}).$$
This is a contraction of a tensor of the same form as the second term in \eqref{cockadoodledoo} (or its conjugate), so as above it follows from Lemma \ref{l:braindead} that
$$\sup\nolimits_{\{z\}\times Y} |\Delta^{g_Y}(\psi|_{\{z\}\times Y})| \leq C\delta_t^{2+\alpha}$$
for some uniform constant $C$ independent of $t$ and $z$. Since the $g_Y$-average of $\psi$ over $\{z\} \times Y$ is zero by construction, standard $L^p$ elliptic estimates and the Sobolev-Morrey embedding give
$$\sup\nolimits_{\{z\}\times Y} |\nabla^{k-1}_{\mathbf{b},\ov{\mathbf{b}}}(\hat{g}^\bullet_t)_{\mathbf{b}\ov{\mathbf{f}}}|_{g_P}\leq C\sup\nolimits_{\{z\}\times Y}|\nabla^{g_Y}(\psi|_{\{z\}\times Y})|_{g_Y} \leq C\delta_t^{2+\alpha}.$$
On the other hand, Claim 1 in particular gives us
$$[\nabla(\nabla^{k-1}_{\mathbf{b},\ov{\mathbf{b}}}(\hat{g}^\bullet_t)_{\mathbf{b}\ov{\mathbf{f}}})]_{C^\alpha(B_\ve\times Y,g_P)}\leq C\delta_t.$$
Thus, finally, bringing in Lemma \ref{l:higher-interpol},
$$\|\nabla(\nabla^{k-1}_{\mathbf{b},\ov{\mathbf{b}}}(\hat{g}^\bullet_t)_{\mathbf{b}\ov{\mathbf{f}}})\|_{L^\infty(B^{g_P}(\hat{x}_t,\frac{\delta_t}{2}))} \leq C\delta_t^{1+\alpha}$$
as long as $\delta_t \leq \epsilon$ (and the estimate is trivial otherwise). This tells us that the term \eqref{hard} decays like $\delta_t^{1+\alpha}$ when evaluated at the basepoint $\hat{x}_t$. Thus, at $\hat{x}_t$, the corresponding term
$$|\nabla^k_{\mathbf{b},\ov{\mathbf{b}}}(\hat{g}_t^\bullet)_{\mathbf{b}\ov{\mathbf{f}}}|\delta_t^{-1}$$
appearing in \eqref{somma} with $j = 1$ decays like $\delta_t^\alpha$, hence in particular goes to zero as $\delta_t\to 0$, as desired.

This completes the proof of Claim 2.

\subsubsection{}\label{claim3product}{\bf Claim 3.} The $C^{k-1,\alpha}$ K\"ahler form $\hat{\omega}^\bullet_\infty$ on $\mathbb{C}^m$ is parallel with respect to the Euclidean metric.\medskip\

This is already obvious from \eqref{blow3} if $k > 1$, so in the proof we will assume that $k = 1$. \medskip\

\noindent \emph{Proof of Claim 3.} Recall that on $B_{\lambda_t}\times Y$ we have
\begin{align}
\hat{\omega}^\bullet_t=\omega_{\mathbb{C}^m}+\delta_t^2\omega_Y+\ddbar\vp_t,\label{maresca-1}\\
\label{maresca}
(\hat{\omega}^\bullet_t)^{m+n}=\delta_t^{2n} e^{F_t}\omega_P^{m+n},
\end{align}
where $\vp_t = \lambda_t^2 \psi_t \circ \Psi_t$ and $F_t=F\circ\Psi_t$. Moreover, as we already said, by \eqref{whenwillthiseverend} and Ascoli-Arzel\`a, $\hat\omega_t^\bullet$ converges subsequentially in $C^{\alpha}_{\rm loc}(\C^m \times Y)$ to the pullback of a $(1,1)$-form $\hat\omega_\infty^\bullet$ from $\C^m$. Then $\hat\omega_\infty^\bullet$ is obviously weakly closed, hence has a global $\ddbar$-potential of regularity $C$\begin{small}$^{2,\alpha}_{\rm loc}$\end{small} on $\C^m$. (If one generalizes the stronger bound \eqref{extra} from $B_\epsilon \times Y$ to $B_R \times Y$ for an arbitrary $R$, then $\hat\omega_\infty$ is a priori seen to be of class $C$\begin{small}$^{1,\alpha}_{\rm loc}$\end{small} and closed in the classical sense, with a $C$\begin{small}$^{3,\alpha}_{\rm loc}$\end{small} potential.) The crucial point of the proof of Claim 3 is to derive from \eqref{maresca-1}, \eqref{maresca} that $\hat{\omega}^\bullet_\infty$ satisfies
\begin{equation}\label{wp}
(\hat{\omega}^\bullet_\infty)^m=e^{F(z_\infty)}\omega_{\mathbb{C}^m}^m.
\end{equation}
A standard bootstrapping argument then shows that $\hat{\omega}^\bullet_\infty$ is smooth, and hence Ricci-flat. Since $\hat\omega^\bullet_\infty$ is moreover uniformly equivalent to $\omega_{\C^m}$ thanks to \eqref{caz}, Claim 3 then follows from Theorem \ref{lio1}.

The derivation of \eqref{wp} is similar to the proof of \cite[Thm 4.1]{To}, multiplying \eqref{maresca} by a test function pulled back from the base and using \eqref{maresca-1} to integrate by parts. Write $\hat{\omega}^\bullet_\infty=\omega_{\mathbb{C}^m}+\ddbar\vp$ for some $\vp \in C$\begin{small}$^{2,\alpha}_{\rm loc}$\end{small}$(\C^m)$, and also write $\vp$ for the pullback of $\vp$ to $\mathbb{C}^m\times Y$, so that in particular $\ddbar\vp_t\to\ddbar\vp$ in $C$\begin{small}$^{\alpha}_{\rm loc}$\end{small}. As in \cite{To}, let $\underline{\vp_t}$ denote the function on $\C^m$ (as well as its pullback to $\C^m \times Y$) obtained as the fiber average of $\vp_t$ with respect to $\omega_Y^n$. Fix a smooth function $\eta$ on $\mathbb{C}^m$ with compact support $K$ and also write $\eta$ for its pullback to $\mathbb{C}^m\times Y$. Fix $t$ large enough so  that $K\subset B_{\lambda_t}$. Then, from \eqref{maresca},\hfill$\qquad$
\begin{equation*}\begin{split}
\int_{\mathbb{C}^m\times Y}\eta e^{F_t}\omega_P^{m+n}&=\frac{1}{\delta_t^{2n}}\int_{\mathbb{C}^m\times Y}\eta (\omega_{\mathbb{C}^m}+\delta_t^2\omega_Y+\ddbar\vp_t)^{m+n}\\
&=
\frac{1}{\delta_t^{2n}}\int_{\mathbb{C}^m\times Y}\eta ((\omega_{\mathbb{C}^m}+\mn\de\db\underline{\vp_t})+(\delta_t^2\omega_Y+\mn\de\db(\varphi_t-\underline{\vp_t})))^{m+n}\\
&=\frac{1}{\delta_t^{2n}}\int_{\mathbb{C}^m\times Y}\eta \sum_{j=0}^{m+n}\binom{m+n}{j}(\omega_{\mathbb{C}^m}+\mn\de\db\underline{\vp_t})^j\wedge(\delta_t^2\omega_Y+\mn\de\db(\varphi_t-\underline{\vp_t}))^{m+n-j}.
\end{split}
\end{equation*}
Observe that $\omega_{\mathbb{C}^m}+\mn\de\db\underline{\vp_t}$ is pulled back from $\mathbb{C}^m$, hence can be wedged with itself at most $m$ times, so all terms in the sum with $j>m$ are zero. Next, we claim that all the terms with $j<m$ go to zero as $t\to \infty$. To see this, start by observing that any such term can be expanded into
$$\frac{1}{\delta_t^{2n}}\binom{m+n}{j}\sum_{i=0}^{m+n-j}\binom{m+n-j}{i}\int_{\mathbb{C}^m\times Y} \eta(\omega_{\mathbb{C}^m}+\mn\de\db\underline{\vp_t})^j\wedge(\delta_t^2\omega_Y)^{m+n-j-i}\wedge(\mn\de\db(\varphi_t-\underline{\vp_t}))^{i}.$$
The term with $i=0$ is easily seen to go to $0$ because the integrand is $O(\delta_t^{2(m+n-j)})$ and $j<m$, while for each term with $i>0$ we can rewrite the integral as
\begin{equation}\label{e:dontfuckup}
\int_{\mathbb{C}^m\times Y} (\vp_t-\underline{\vp_t})\ddbar\eta\wedge(\omega_{\mathbb{C}^m}+\mn\de\db\underline{\vp_t})^j\wedge(\delta_t^2\omega_Y)^{m+n-j-i}\wedge(\mn\de\db(\varphi_t-\underline{\vp_t}))^{i-1}.
\end{equation}
Work in local holomorphic product coordinates on the total space. The form $\ddbar\eta\wedge(\omega_{\mathbb{C}^m}+\mn\de\db\underline{\vp_t})^j$ is pulled back from $\mathbb{C}^m$, so in the coordinate representation of $(\delta_t^2\omega_Y)^{m+n-j-i}\wedge(\mn\de\db(\varphi_t-\underline{\vp_t}))^{i-1}$, each summand is a wedge product of $2(m-j-1)$ basis $1$-forms pulled back from $\mathbb{C}^m$ and $2n$ basis $1$-forms pulled back from $Y$. Multiplied together, the fiber contributions are $O(\delta_t^{2n})$ because the ones coming in pairs from $(\delta_t^2\omega_Y)^{m+n-j-i}$ have an explicit factor of $\delta_t^2,$ while \eqref{caz} implies that for all $z\in K,$
\begin{equation}\label{pazzoide}
|(\ddbar (\vp_t-\underline{\vp_t}))|_{\{z\}\times Y}|=|(\ddbar\vp_t)|_{\{z\}\times Y}|\leq C\delta_t^2,
\end{equation}
which shows not only that the contributions coming in pairs from $(\ddbar (\vp_t-\underline{\vp_t}))|_{\{z\}\times Y}$ are again $O(\delta_t^2)$ but also, thanks to the Cauchy-Schwarz inequality with respect to the Riemannian metric $\hat{g}_t^\bullet$, that the mixed base-fiber components of $\ddbar (\vp_t-\underline{\vp_t})$ are $O(\delta_t)$. Thus, every term
$$i\partial\bar\partial\eta \wedge (\omega_{\mathbb{C}^m}+\mn\de\db\underline{\vp_t})^j\wedge(\delta_t^2\omega_Y)^{m+n-j-i}\wedge(\mn\de\db(\varphi_t-\underline{\vp_t}))^{i-1}$$
in \eqref{e:dontfuckup} is $O(\delta_t^{2n})$. To see that the whole integral in \eqref{e:dontfuckup} is $o(\delta_t^{2n})$, as desired, it then suffices to note that $\sup_{K\times Y}|\vp_t-\underline{\vp_t}|\leq C\delta_t^2$, which follows from \eqref{pazzoide} by inverting the Laplacian on each fiber.

We are then left with only the term with $j=m$, which is
\[\begin{split}
&\frac{1}{\delta_t^{2n}}\int_{\mathbb{C}^m\times Y}\eta\binom{m+n}{m}(\omega_{\mathbb{C}^m}+\mn\de\db\underline{\vp_t})^m\wedge(\delta_t^2\omega_Y+\mn\de\db(\varphi_t-\underline{\vp_t}))^{n}\\
=\;&\frac{1}{\delta_t^{2n}}\int_{\mathbb{C}^m\times Y}\eta\binom{m+n}{m}(\omega_{\mathbb{C}^m}+\mn\de\db\underline{\vp_t})^m\wedge (\delta_t^2\omega_Y)^{n}
+\frac{1}{\delta_t^{2n}}\int_{\mathbb{C}^m\times Y} \mn\de\db\eta\wedge(\omega_{\mathbb{C}^m}+\mn\de\db\underline{\vp_t})^m\wedge\Upsilon.
\end{split}\]
The second term is zero because $\mn\de\db\eta$ is pulled back from $\C^m$. Altogether, we obtain that
\begin{equation}\label{e:to3}
\int_{\mathbb{C}^m\times Y}\eta e^{F_t}\omega_P^{m+n} = \binom{m+n}{m}\int_{\mathbb{C}^m\times Y}\eta(\omega_{\mathbb{C}^m}+\mn\de\db\underline{\vp_t})^m\wedge \omega_Y^{n}.
\end{equation}
Now observe that  $\ddbar\underline{\vp_t}\to\ddbar\vp$, and hence $\omega_{\C^m} + \ddbar\underline{\vp_t} \to \hat\omega_\infty^\bullet$, in $C^{\alpha}_{\rm loc}$ thanks to the identity
$$\ddbar(\underline{\vp_t}-\vp)=({\rm pr}_{\C^m})_*\left(\ddbar(\vp_t-\vp)\wedge\frac{\omega_Y^n}{\int_Y \omega_Y^n}\right)$$
and to the fact that by assumption $\ddbar\vp_t \to \ddbar\vp$ in $C^\alpha_{\rm loc}$. Expanding $\omega_P^{m+n}$ in \eqref{e:to3}, letting $t \to \infty$ and integrating out the $Y$ factor yields the weak form of \eqref{wp}, as desired.

This completes the proof of Claim 3, hence of Case 3 and of Theorem \ref{mainlocal}. \hfill $\Box$

\begin{rk}
It seems plausible that the method of proof of Theorem \ref{mainlocal} can be combined with the methods of \cite{GTZ,HT} to obtain a generalization of the $C^\infty$ estimates of Theorem \ref{mainlocal} to the setting of proper holomorphic submersions $f: X \to B$ with Calabi-Yau fibers over the unit ball $B \subset \C^m$ (as in Theorem \ref{thm51}) such that the simply-connected Beauville-Bogomolov-Calabi factors of the fibers of $f$ are all biholomorphic (while the torus factors may vary). The idea is that even though in this setting the natural collapsing semi-Ricci-flat K\"ahler reference metrics on $X$ are no longer Riemannian products, they still satisfy very good estimates with respect to certain Riemannian product metrics.
\end{rk}

\section{$C^{\alpha}$ estimates in the non-product case}\label{nonpr}

Let $B = B_1(0)$ denote the unit ball in $\C^m$ and let $f: X \to B$ be a proper surjective holomorphic submersion with $n$-dimensional Calabi-Yau fibers. Let $\omega_X$ be a Ricci-flat K\"ahler form on $X$.
For each $z\in B$ we use the Calabi-Yau theorem to find a unique Ricci-flat K\"ahler metric $\omega_{F,z}$ on $X_z$ in the class $[\omega_X|_{X_z}]$. Writing $\omega_{F,z}=\omega_X|_{X_z}+\ddbar \rho_z$, we may choose the functions $\rho_z$ to depend smoothly on $z$ (for example by normalizing them to have $\omega_X$-fiberwise average zero), and so to define a smooth function $\rho$ on $X$. We can then define a closed real $(1,1)$-form $\omega_F$ on $X$ by $\omega_F=\omega_X+\ddbar\rho$, so that the restriction of $\omega_F$ to $X_z$ equals $\omega_{F,z}$. (In the product case that we treated earlier, we could simply take $\omega_F = \omega_Y$. It was recently shown in \cite{CGP} that $\omega_F$ is in general not semipositive definite on $X$, see also \cite{Choi}, but this will be irrelevant for us.)
We further define a family of closed real $(1,1)$-forms on $X$ by
$$\omega_t^\natural =\omega_\infty+e^{-t}\omega_F,$$
where $\omega_\infty = f^*\omega_{\C^m}$.
In general, the forms $\omega_t^\natural$ are not positive for all $t\geq 0$. However, it is easy to see (using Cauchy-Schwarz for the base-fiber terms of $\omega_F$, which come with a good factor of $e^{-t}$, cf. \eqref{caz2b}) that given any compact subset
$K\subset X$ there is a $t_K$ such that $\omega_t^\natural$ is positive on $K$ for all $t\geq t_K$. In particular, by shrinking $B$ slightly, we can assume that $\omega_t^\natural$ is positive on $X$ for all $t \gg 1$, and without loss even for all $t \geq 0$.
We are concerned with the behavior as $t \to \infty$ of Ricci-flat K\"ahler metrics $\omega_t^\bullet$ on $X$ which satisfy
$\omega_t^\bullet = \omega_t^\natural + i\partial\ov{\partial}\psi_t$
together with the complex Monge-Amp\`ere equation
\begin{equation}\label{mafiber}
(\omega_t^\bullet)^{m+n}= c_t e^{-nt} e^{G} \omega_{\infty}^m\wedge \omega_F^n,
\end{equation}
where $c_t$ is a constant that has a positive limit, and $G$ is a smooth function pulled back from $B$.

By Ehresmann's theorem $f$ is a $C^\infty$ fiber bundle, and (up to shrinking $B$ again) we may choose a smooth trivialization $\Phi:B\times Y\to X$, where $Y=f^{-1}(0)$ is viewed as a smooth real $2n$-manifold, such that the restriction $\Phi|_{\{0\}\times Y}:\{0\}\times Y\to f^{-1}(0)=Y$ equals the identity map. We then equip $B\times Y$ with the complex structure $J^\natural$ induced by the one on $X$ via $\Phi$, so that $\Phi$ becomes a biholomorphism and ${\rm pr}_{\C^m}$ is a $J^\natural$-holomorphic submersion. Let $J_{Y,z}$ denote the restriction of $J^\natural$ to the $J^\natural$-holomorphic fiber $\{z\} \times Y$. Note that $\Phi^*\omega_{F,z}$ is a Ricci-flat $J_{Y,z}$-K\"ahler metric on $\{z\}\times Y$. Denote the associated Riemannian metric on $\{z\} \times Y$ by $g_{Y,z}$, extend it trivially to the $C^\infty$ product $\C^m \times Y$, and define the Riemannian product metric $g_{z,t} = g_{\C^m} + e^{-t}g_{Y,z}$, which is K\"ahler with respect to $J_{z} = J_{\C^m} + J_{Y,z}$. By abuse of notation we will identify $\omega_F$ and $\omega_t^\natural, g_t^\natural$ with their pullbacks to $B \times Y$ under $\Phi$.

The following is a restatement of Theorem \ref{thm51}, and is the main result of this section.
\begin{theorem}\label{t:fornow}
For all $C$ and $\alpha \in (0,1)$ there exists a $C'$ independent of $t$ such that if
\begin{equation}\label{arrrgh}
C^{-1} g^\natural_t\leq {g}_t^\bullet \leq Cg^\natural_t\;\,\text{on}\;\,B \times Y,
\end{equation}
then it holds that
\begin{equation}\label{arrrrgh}
\sup_{x=(z,y) \in B_{\frac{1}{4}} \times Y} \sup_{x' \in B^{g_{z,t}}(x,\frac{1}{8})}\frac{|{g}_t^\bullet(x) -  \mathbf{P}^{g_{z,t}}_{x'x}({g}_t^\bullet(x'))|_{g_{z,t}(x)}}{d^{g_{z,t}}(x,x')^\alpha} \leq C'.
\end{equation}
\end{theorem}

We begin with some remarks to clarify the nature of this estimate.

\begin{rk}
Theorem \ref{t:fornow} holds with respect to any $C^\infty$ trivialization $\Phi$. In fact, if \eqref{arrrrgh} holds with respect to one choice of $\Phi$, then one can prove directly that it also holds with respect to any other choice. This is similar to the proof of Lemma \ref{tiammazzo1} but longer (using also Proposition \ref{bd} below).
\end{rk}

\begin{rk}\label{totalfuckage}
Estimate \eqref{arrrrgh} is weaker than a uniform bound on $[g^\bullet_t]_{C^\alpha(B_{1/4}\times Y,g_{z_0,t})}$ for a fixed $z_0\in B$, which is what one would naively expect after our results in Theorem \ref{mainlocal} in the product case. However such a bound is false except in the product or torus-fibered cases. Indeed, assuming it were true, scale all distances by $e^{t/2}$ (which has the effect of multiplying our assumed H\"older bound by $e^{-\alpha t/2}$), stretch the base directions accordingly, and pass to a pointed limit based at some generic point $z \in B \setminus \{z_0\}$. Thanks to \cite[Thm 3.1]{TZ}, the scaled and stretched $g_t^\bullet$ converges locally smoothly to $g_{\C^m} + g_{Y,z}$, while the scaled and stretched $g_{z_0,t}$ obviously converges to $g_{\C^m} + g_{Y,z_0}$. It follows that the former is parallel with respect to the latter, which is false except in the special cases mentioned above. This argument also shows that \eqref{arrrrgh} provides an effective estimate on the rate of convergence in \cite[Thm 3.1]{TZ}.
\end{rk}

\begin{rk}\label{important}
Estimate \eqref{arrrrgh} does imply a uniform $C^\alpha$ bound for $g^\bullet_t$ with respect to a {\em fixed} metric on $B\times Y$. Indeed, notice that $\mathbf{P}^{g_{z,t}}_{x'x} = \mathbf{P}^{g_{z,0}}_{x'x}$ and $d^{g_{z,t}}(x',x) \leq d^{g_{z,0}}(x',x)$ and $|T|_{g_{z,t}(x)} \geq |T|_{g_{z,0}(x)}$ for any contravariant tensor $T$, so \eqref{arrrrgh} trivially implies a uniform bound on the $g_{z,0}$-H\"older quotient of $g_t^\bullet$ at $x,x'$. By Lemma \ref{tiammazzo1}, this implies a uniform bound on the H\"older quotient of $g_t^\bullet$ at $x,x'$ with respect to any fixed metric smoothly and uniformly comparable to $g_{z,0}$. Thus, Theorem \ref{t:fornow} implies the H\"older estimates stated in Corollaries \ref{main2} and \ref{coro} in the case of compact Calabi-Yau manifolds. These results in particular recover, and greatly strengthen, the main result of \cite{TWY}.
\end{rk}

Before we start the proof of Theorem \ref{t:fornow}, we make a simple observation which will be very useful in the proof. Fix a smooth complex coordinate chart $(y^1, \ldots, y^n)$ on $Y$. Then $(z^1, \ldots, z^m, y^1, \ldots, y^n)$ is a smooth complex coordinate chart on $B \times Y$.
Schematically, and ignoring the distinction between these complex coordinates and their complex conjugates, we may then write
\begin{align}\label{Jblock}(J^\natural - J_{z_0})|_{(z,y)} = A(z_0, z, y) \circledast dz \otimes \partial_y + B(z_0, z, y) \circledast dy \otimes \partial_y,\end{align}
where $A,B$ are smooth matrix-valued functions with $B(z_0, z_0, y) = 0$. There are no $dz \otimes \partial_z$ or $dy \otimes \partial_z$ terms in this formula because ${\rm pr}_{\C^m}$ is holomorphic with respect to both $(J^\natural, J_{\C^m})$ and $(J_{z_0}, J_{\C^m})$. The absence of such terms is crucial for us because they behave poorly with respect to diffeomorphisms of the form $(z,y) \mapsto (\lambda z, y)$ with $\lambda \to 0$. Given \eqref{Jblock}, the definition of $\omega_t^\natural$ implies that
\begin{align}\label{gblock}(g_t^\natural - g_{z_0,t})|_{(z,y)} = e^{-t}(C(z_0,z,y) \circledast dz \otimes dz + D(z_0,z,y)\circledast dz \otimes dy + E(z_0,z,y) \circledast dy \otimes dy), \end{align}
where $C,D,E$ are smooth matrix-valued functions with $E(z_0,z_0,y) = 0$.

Formulas \eqref{Jblock} and \eqref{gblock} will usually be applied as follows. We would like to prove that $J^\natural$ becomes asymptotic to $J_{z_0}$, and $g^\natural_t$ becomes asymptotic to $g_{z_0,t}$, after pulling back by a diffeomorphism of the form $(z,y) \mapsto (\lambda z, y)$ with $\lambda \to 0$ and passing to a pointed limit centered at  $(\lambda^{-1}z_0, y_0)$. This follows from \eqref{Jblock}, \eqref{gblock} by observing that $dz$ gains a factor of $\lambda$ under pullback, hence becomes negligible, and that the $dy$ terms also go to zero locally uniformly by Taylor expanding $B,E$ around $z = z_0$.

We will use \eqref{Jblock} and \eqref{gblock} several times below, often to derive estimates which are global in the $Y$ factor. For example, we will be interested in estimates on the difference of $J^\natural - J_{z_0}$ or $g_t^\natural - g_{z_0,t}$ at two arbitrary points $(z,y)$ and $(z',y')$ such that $y$ and $y'$ do not necessarily belong to the same coordinate chart. In such a situation we will cover $Y$ by finitely many coordinate charts as above, apply \eqref{Jblock}, \eqref{gblock} in each of these charts, and use the triangle inequality.

We also note here that there is a constant $C$ such that for all $z\in B$ and $t \in [0,\infty)$ we have
\begin{equation}\label{caz2b}
C^{-1}g_{z,t} \leq g^\natural_t \leq Cg_{z,t}.
\end{equation}
This easily follows from \eqref{gblock} and Cauchy-Schwarz because the $dz \otimes dy$ terms come with a coefficient of $e^{-t}$ rather than $e^{-t/2}$. Note that this observation was already used in the setup of Theorem \ref{t:fornow}.

We are now ready to start the proof of Theorem \ref{t:fornow}.

\begin{proof}[Proof of Theorem \ref{t:fornow}] We begin with a preliminary reduction.
Define a function $\mu_t$ on $B \times Y$ by
\begin{align}
\mu_{t}(x) = \mu_t(z,y) = d^{g_{z,t}} (x,\partial(B \times Y))^\alpha \sup_{x' \in B^{g_{z,t}}(x,\frac{1}{4}d^{g_{z,t}}(x,\partial(B\times Y)))} \frac{|\eta_t(x) -  \mathbf{P}^{g_{z,t}}_{x'x}(\eta_t(x'))|_{g_{z,t}(x)}}{d^{g_{z,t}}(x,x')^\alpha},
\end{align}
where $\eta_t = \omega_t^\bullet - \omega_t^\natural=\ddbar\psi_t$. This quantity is more convenient for us than the analogous one involving $g_t^\bullet$ because all of our key analytic arguments work at the level of $i\partial\bar\partial$-exact $(1,1)$-forms.  \medskip\

\noindent \emph{Claim.} Theorem \ref{t:fornow} can be deduced from the statement that
\begin{align}\label{strongk2}
\max_{B\times Y} \mu_t \leq C.
\end{align}
\noindent \emph{Proof of Claim.} Note that for all tangent vectors $v,w$ we have
$$g^\bullet_t(v,w)=g^\natural_t(v,w)+\eta_t(v,J^\natural w).$$
We denote the last term by $h_t(v,w)$.
From \eqref{arrrgh} and \eqref{caz2b}, and using the fact that $J^\natural$ has fixed length with respect to $g_t^\natural$ because $g^\natural_t$ is $J^\natural$-Hermitian, we deduce that for all $z\in B$,
\begin{equation}\label{tiammazz}
\sup_{B\times Y} (|h_t|_{g_{z,t}} + |J^\natural|_{g_{z,t}} + |\eta_t|_{g_{z,t}})\leq C.
\end{equation}
Then for all $x=(z,y) \in B_{1/4} \times Y$ and $x'=(z',y') \in B^{g_{z,t}}(x,1/8)$ we estimate
\begin{equation}\label{bs0}
|{g}_t^\bullet(x) -  \mathbf{P}^{g_{z,t}}_{x'x}({g}_t^\bullet(x'))|_{g_{z,t}(x)}\leq |{g}_t^\natural(x) -  \mathbf{P}^{g_{z,t}}_{x'x}({g}_t^\natural(x'))|_{g_{z,t}(x)}+|h_t(x) -  \mathbf{P}^{g_{z,t}}_{x'x}(h_t(x'))|_{g_{z,t}(x)}.
\end{equation}
For the first term in \eqref{bs0}, we clearly have
$$|{g}_t^\natural(x) -  \mathbf{P}^{g_{z,t}}_{x'x}({g}_t^\natural(x'))|_{g_{z,t}(x)}=|({g}_t^\natural-g_{z,t})(x) -  \mathbf{P}^{g_{z,t}}_{x'x}(({g}_t^\natural-g_{z,t})(x'))|_{g_{z,t}(x)},$$
which we aim to estimate by $Cd^{g_{z,t}}(x,x')$ using \eqref{gblock} and the covering argument indicated after \eqref{gblock}. This is easy for the $C,D$ terms in \eqref{gblock} because there is a factor of $e^{-t}$ in front while the $g_{z,t}$-lengths of $dz$ and $dy$ are $O(1)$ and $O(e^{t/2})$ respectively. For the $E$ term, using the fact that $E(z,z,y)=0$ for all $y$, we see that its contribution is $O(|z-z'|) = O(d^{g_{z,t}}(x,x'))$. Thus, as desired,
\begin{equation}\label{bs1}
|{g}_t^\natural(x) -  \mathbf{P}^{g_{z,t}}_{x'x}({g}_t^\natural(x'))|_{g_{z,t}(x)}\leq Cd^{g_{z,t}}(x,x').
\end{equation}
For the second term in \eqref{bs0}, we bound
\begin{align}
&|h_t(x) -  \mathbf{P}^{g_{z,t}}_{x'x}(h_t(x'))|_{g_{z,t}(x)}\\
&\leq (\sup|\eta_t|_{g_{z,t}})|J^\natural(x) -  \mathbf{P}^{g_{z,t}}_{x'x}(J^\natural(x'))|_{g_{z,t}(x)}+C(\sup |J^\natural|_{g_{z,t}})(\max\mu_t)d^{g_{z,t}}(x,x')^\alpha,\label{bs2}
\end{align}
where the sup and max are taken over $B \times Y$.
We have
$$|J^\natural(x) -  \mathbf{P}^{g_{z,t}}_{x'x}(J^\natural(x'))|_{g_{z,t}(x)}=|(J^\natural-J_{z})(x) -  \mathbf{P}^{g_{z,t}}_{x'x}((J^\natural-J_z)(x'))|_{g_{z,t}(x)},$$
which we again aim to estimate by $Cd^{g_{z,t}}(x,x')$ using \eqref{Jblock}. This is again easy for the $A$ term in \eqref{Jblock} because the $g_{z,t}$-length of $\de_y$ is $O(e^{-t/2})$, and for the $B$ term, using the fact that $B(z,z,y)=0$, we see that its contribution is $O(|z-z'|) = O(d^{g_{z,t}}(x,x'))$. This gives
\begin{equation}\label{bs3}
|J^\natural(x) -  \mathbf{P}^{g_{z,t}}_{x'x}(J^\natural(x'))|_{g_{z,t}(x)}\leq Cd^{g_{z,t}}(x,x').
\end{equation}
Plugging this together with \eqref{strongk2}, \eqref{tiammazz}, \eqref{bs3} into \eqref{bs2} gives
\begin{equation}\label{bs4}
|h_t(x) -  \mathbf{P}^{g_{z,t}}_{x'x}(h_t(x'))|_{g_{z,t}(x)}\leq Cd^{g_{z,t}}(x,x')^\alpha.
\end{equation}
Taken together, \eqref{bs1} and \eqref{bs4} allow us to appropriately bound the right-hand side of \eqref{bs0}, proving \eqref{arrrrgh}. Thus, Theorem \ref{t:fornow} follows from \eqref{strongk2}, as claimed.\hfill$\Box$\medskip\

Thanks to the above claim, it suffices to prove \eqref{strongk2}, which we do by contradiction.

If \eqref{strongk2} is false, then
$\limsup_{t\to\infty} \max_{B\times Y} \mu_{t}=\infty.$
For simplicity of notation we will assume that $\lim_{t\to\infty}\max_{B\times Y} \mu_{t}=\infty$ (usually this will be true only along some sequence $t_i\to \infty$). Choose $x_t = (z_t, y_t) \in B\times Y$ such that the maximum of $\mu_{t}$ is achieved at $x_t$, and define $\lambda_t$ by
$$\lambda_t^\alpha=\sup_{x' \in B^{g_{z_t,t}}(x_t,\frac{1}{4}d^{g_{z_t,t}}(x_t,\partial(B\times Y)))}\frac{|\eta_t(x_t) -  \mathbf{P}^{g_{z_t,t}}_{x'x_t}(\eta_t(x'))|_{g_{z_t,t}(x_t)}}{d^{g_{z_t,t}}(x_t,x')^\alpha}.$$
Let us note for later purposes that after passing to a subsequence,
\begin{align}\label{whoops666}
z_t \to z_\infty \in \overline{B}, \;\, y_t \to y_\infty \in Y.
\end{align}
Now $\lambda_t \to \infty$ since otherwise $\max_{B\times Y} \mu_{t}$ would be uniformly bounded because the $g_{z,t}$-diameter of $B\times Y$ is uniformly bounded independent of $z$ and $t$. Let us also choose any point
$$x_t' = (z_t', y_t')  \in \overline{B^{g_{z_t,t}}(x_t,\frac{1}{4}d^{g_{z_t,t}}(x_t,\partial(B\times Y)))}$$ realizing the sup in the definition of $\lambda_t$.
Consider the diffeomorphisms
$$\Psi_t:B_{\lambda_t}\times Y\to B\times Y, \;\, (z,y) = \Psi_t(\hat{z},\hat{y})=(\lambda_t^{-1}\hat{z},\hat{y}),$$
and define
$$\hat{J}_t=\Psi_t^*J_{z_t}, \;\, \hat{g}_t = \lambda_t^{2}\Psi_t^*g_{z_t,t},\;\, \hat{J}^\natural_{t} = \Psi_t^*J^\natural,\;\, \hat\omega_t^\natural = \lambda_t^{2}\Psi_t^*\omega_t^\natural, \;\, \hat\eta_t=\lambda_t^{2}\Psi_t^*\eta_t, \;\, \hat{x}_t=\Psi_t^{-1}(x_t), \;\, \hat{x}'_t=\Psi_t^{-1}(x'_t).$$
Then $\hat{g}_t$ is a Ricci-flat $\hat{J}_t$-K\"ahler product metric, $\hat{\omega}_t^\natural$ is a semi-Ricci-flat $\hat{J}_t^\natural$-K\"ahler form, $\hat\omega_t^\bullet = \hat\omega_t^\natural + \hat\eta_t$ is a Ricci-flat $\hat{J}^\natural_t$-K\"ahler form, and we have the following basic properties:
\begin{align}
\label{caz01}\hat{g}_t = g_{\C^m} +  \lambda_t^{2}e^{-t}g_{Y,z_t}, \;\,
\hat{\omega}_t^\natural = \omega_{\C^m} +  \lambda_t^{2}e^{-t}\Psi_t^*\Phi^*\omega_F,\\
\label{caz2}
C^{-1}\hat{g}_t \leq \hat{g}^\bullet_t \leq C\hat{g}_t.
\end{align}
Here \eqref{caz01} is obvious and \eqref{caz2} follows from  \eqref{arrrgh} and \eqref{caz2b}.

The following properties are crucial for our proof. First note that
\begin{align*}
\lambda_t^\alpha=\frac{|\eta_t(x_t) -  \mathbf{P}^{g_{z_t,t}}_{x'_tx_t}(\eta_t(x'_t))|_{g_{z_t,t}(x_t)}}{d^{g_{z_t,t}}(x_t,x'_t)^\alpha}
&=\frac{|\hat\eta_t(\hat{x}_t) -  \mathbf{P}^{\Psi_t^*g_{z_t,t}}_{\hat{x}'_t\hat{x}_t}(\hat\eta_t(\hat{x}'_t))|_{\Psi_t^*g_{z_t,t}(\hat{x}_t)}}{d^{\Psi_t^*g_{z_t,t}}(\hat{x}_t,\hat{x}'_t)^\alpha} \lambda_t^{-2} \\
&=\frac{|\hat\eta_t(\hat{x}_t) -  \mathbf{P}^{\hat{g}_t}_{\hat{x}'_t\hat{x}_t}(\hat\eta_t(\hat{x}'_t))|_{\hat{g}_t(\hat{x}_t)}}{d^{\hat{g}_t}(\hat{x}_t,\hat{x}'_t)^\alpha}\lambda_t^\alpha,
\end{align*}
which implies that
\begin{equation}\label{blow22}
\frac{|\hat\eta_t(\hat{x}_t) -  \mathbf{P}^{\hat{g}_t}_{\hat{x}'_t\hat{x}_t}(\hat\eta_t(\hat{x}'_t))|_{\hat{g}_t(\hat{x}_t)}}{d^{\hat{g}_t}(\hat{x}_t,\hat{x}'_t)^\alpha}
=1.
\end{equation}
Since the numerator is uniformly bounded thanks to \eqref{caz2}, it follows in particular that
\begin{align}\label{blow223}
d^{\hat{g}_t}(\hat{x}_t, \hat{x}_t') \leq C.
\end{align}
Now recall that $\hat{x}_t'$ was chosen to maximize the difference quotient of \eqref{blow22} among all points $$\hat{x}' \in \ov{B^{\hat{g}_t}(\hat{x}_t,\frac{1}{4}d^{\hat{g}_t}(\hat{x}_t,\partial(B_{\lambda_t}\times Y)))}.$$
Moreover, if we define
\begin{align}\label{fuckup!!}
\hat{g}_{\hat{z},t} =  \lambda_t^{2}\Psi_t^*g_{z,t} = g_{\C^m} + \lambda_t^{2}e^{-t}g_{Y,z}\;\,(z = \lambda_t^{-1}\hat{z})
\end{align}
for all $\hat{z} \in B_{\lambda_t}$ (so that $\hat{g}_{\hat{z}_t,t}=\hat{g}_t$ by definition), then the point $\hat{x}_t$ itself maximizes the quantity
\begin{align*}
\mu_{t}(\Psi_t(\hat{x}))= d^{\hat{g}_{\hat{z},t}} (\hat{x},\partial(B_{\lambda_t} \times Y))^\alpha \sup_{\hat{x}' \in B^{\hat{g}_{\hat{z},t}}(\hat{x},\frac{1}{4}d^{\hat{g}_{\hat{z},t}}(\hat{x},\partial(B_{\lambda_t}\times Y)))} \frac{|\hat\eta_t(\hat{x}) -  \mathbf{P}^{\hat{g}_{\hat{z},t}}_{\hat{x}'\hat{x}}(\hat\eta_t(\hat{x}'))|_{\hat{g}_{\hat{z},t}(\hat{x})}}{d^{\hat{g}_{\hat{z},t}}(\hat{x},\hat{x}')^\alpha}
\end{align*}
among all $\hat{x} = (\hat{z},\hat{y}) \in B_{\lambda_t} \times Y$. Together with the fact that $\max \mu_t \to \infty$, these properties yield
\begin{align}\label{fuck_the_boundary2}
d^{\hat{g}_t} (\hat{x}_t, \partial(B_{\lambda_t} \times Y)) \to \infty.
\end{align}
This tells us that if we pass to a pointed limit with basepoint $\hat{x}_t$, the boundary moves away to infinity and the limit space will be complete. We also learn that for all $\hat{x} = (\hat{z}, \hat{y}) \in B_{\lambda_t} \times Y$,
\begin{align*}
\sup_{\hat{x}' \in B^{\hat{g}_{\hat{z},t}}(\hat{x},\frac{1}{4}d^{\hat{g}_{\hat{z},t}}(\hat{x},\partial(B_{\lambda_t}\times Y)))} \frac{|\hat\eta_t(\hat{x}) -  \mathbf{P}^{\hat{g}_{\hat{z},t}}_{\hat{x}'\hat{x}}(\hat\eta_t(\hat{x}'))|_{\hat{g}_{\hat{z},t}(\hat{x})}}{d^{\hat{g}_{\hat{z},t}}(\hat{x},\hat{x}')^\alpha}
&\leq  \frac{d^{\hat{g}_t} (\hat{x}_t,\partial(B_{\lambda_t} \times Y))^\alpha}{d^{\hat{g}_{\hat{z},t}} (\hat{x},\partial(B_{\lambda_t} \times Y))^\alpha}.
\end{align*}
Using the fact that $C^{-1}\hat{g}_t \leq \hat{g}_{\hat{z},t} \leq C\hat{g}_t$ by \eqref{fuckup!!}, the triangle inequality for $d^{\hat{g}_t}$, and \eqref{fuck_the_boundary2}, we deduce that there exists a $C$ such that for all $R > 0$ there exists a $t_R$ such that for all $t \geq t_R$,
\begin{align}
\label{saves_our_asses2}
\sup_{\hat{x} = (\hat{z},\hat{y}) \in B^{\hat{g}_t}(\hat{x}_t,R)} \Biggl(\sup_{\hat{x}' \in B^{\hat{g}_{t}}(\hat{x}_t,R)} \frac{|\hat\eta_t(\hat{x}) -  \mathbf{P}^{\hat{g}_{\hat{z},t}}_{\hat{x}'\hat{x}}(\hat\eta_t(\hat{x}'))|_{\hat{g}_{t}(\hat{x})}}{d^{\hat{g}_{t}}(\hat{x},\hat{x}')^\alpha}\Biggr)
\leq C.
\end{align}
The quantity on the left-hand side of \eqref{saves_our_asses2} is subtly weaker than the $C^{\alpha}(B^{\hat{g}_t}(\hat{x}_t,R))$ seminorm of $\hat{\eta}_t$ because the parallel transport from $\hat{x}'$ to $\hat{x}$ is performed with respect to $\hat{g}_{\hat{z},t}$, where $\hat{x} = (\hat{z},\hat{y})$. In fact, the $C^{\alpha}(B^{\hat{g}_t}(\hat{x}_t,R))$ seminorm of $\hat{\eta}_t$ may well be unbounded. However, \eqref{saves_our_asses2} is sufficient for us.

We are now in a position to study the possible complete pointed limit spaces of
$$(B_{\lambda_t}\times Y, \hat{g}_t,\hat{x}_t)$$
as $t\to\infty$. Modulo translations in the $\mathbb{C}^m$ factor, we may assume that $\hat{x}_t = (0,\hat{y}_t)\in \mathbb{C}^m\times Y$. Recall here that $\hat{y}_t\to \hat{y}_\infty\in Y$ by \eqref{whoops666}. Abbreviating
$$\delta_t=\lambda_te^{-\frac{t}{2}},$$
we read from \eqref{caz01} that again up to passing to a subsequence, three cases need to be considered: (1) $\delta_t\to\infty$; (2) $\delta_t \to \delta \in (0,\infty)$, and without loss of generality $\delta = 1$; and (3) $\delta_t\to 0$. Observe that thanks to \eqref{whoops666}, \eqref{Jblock}, it holds in all three cases that
\begin{align}\label{caz66}
\hat{J}_t, \hat{J}_t^\natural \to J_{\C^m} + J_{Y,z_\infty}\;\,{\rm in}\;\,C^\infty_{\rm loc}(\C^m \times Y).
\end{align}

\subsection{Case 1: the blowup is $\C^{m+n}$} In this case we assume that $\delta_t\to \infty$. We fix a $J_{Y,z_\infty}$-holomorphic chart $(\hat{y}^1, \ldots, \hat{y}^n)$ on $Y$ centered at $y_\infty$, with range the ball $B_2 \subset \C^n$. We may assume without loss that $\hat{y}_t \in B_1$ for all $t$, and $\hat{y}_t \to 0$ as $t \to \infty$. Then the map $(\tilde{z},\tilde{y}) \mapsto (\tilde{z}, \hat{y}_t +\delta_t^{-1}\tilde{y})$ with domain $\mathbb{C}^m \times B_{\delta_t}$ is a diffeomorphism onto its image $\C^m \times B_1(\hat{y}_t) \subset \C^m \times B_2$. Let us pull back $\hat{J}_t, \hat{g}_t, \hat{J}_t^\natural, \hat{\omega}_t^\natural, \hat{\eta}_t, \hat{x}_t, \hat{x}_t'$ under this map, obtaining new objects
$\tilde{J}_t, \tilde{g}_t, \tilde{J}_t^\natural, \tilde{\omega}_t^\natural, \tilde{\eta}_t, \tilde{x}_t, \tilde{x}_t'$. Then $\tilde{x}_t = 0$ and by \eqref{gblock}, \eqref{caz01}, \eqref{caz66},
\begin{align}
\label{caz69}
\tilde{J}_t, \tilde{J}_t^\natural  \to J_{\C^{m+n}}, \;\, \tilde{g}_t, \tilde{g}_t^\natural \to g_{\C^{m+n}}\;\,{\rm in}\;\,C^\infty_{\rm loc}(\C^{m+n}).
\end{align}
We write $\tilde\omega_t^\bullet = \tilde\omega_t^\natural + \tilde\eta_t$ for the pullback of our Ricci-flat $\hat{J}_t^\natural$-K\"ahler form, which is of course $\tilde{J}_t^\natural$-K\"ahler.
Thanks to \eqref{caz69}, \eqref{caz2}, \eqref{blow22}, \eqref{blow223} we have that
\begin{align}
C^{-1}g_{\C^{m+n}} \leq \tilde g_t^\bullet \leq Cg_{\C^{m+n}},\label{caz44}\\
\frac{|\tilde\omega_t^\bullet(0) -  \mathbf{P}^{\tilde{g}_t}_{\tilde{x}'_t 0}(\tilde\omega^\bullet_t(\tilde{x}'_t))|_{\tilde{g}_t(0)}}{d^{\tilde{g}_t}(0,\tilde{x}'_t)^\alpha}
=1 + o(1),\label{caz45}\\
d^{\tilde{g}_t}(0,\tilde{x}_t') \leq C.\label{noreasontolabelthis}
\end{align}
(Note that the $o(1)$ term in \eqref{caz45} comes from the contribution of $\tilde\omega_t^\natural$ to $\tilde\omega_t^\bullet$. This is indeed $o(1)$ thanks to \eqref{caz69} combined with Remark \ref{tiammazzo2}.) Using \eqref{caz44}, Proposition \ref{localest3} now yields uniform $C$\begin{small}$^\infty_{\rm loc}$\end{small}$(\C^{m+n})$ bounds for $\tilde\omega_t^\bullet$. Even the $C$\begin{small}$^1_{\rm loc}$\end{small}$(\C^{m+n})$ or $C$\begin{small}$^\beta_{\rm loc}$\end{small}$(\C^{m+n})$ estimate for any $\beta > \alpha$ applied to the numerator of \eqref{caz45} (combined with Lemma \ref{tiammazzo1} to compare H\"older norms) then tells us that
$$d^{\tilde{g}_t}(0,\tilde{x}_t') \geq C^{-1}.$$
Thus, without loss, $\tilde{x}_t'$ converges to $\tilde{x}_\infty' \neq 0$, and $\tilde\omega_t^\bullet$ converges in $C^\infty_{\rm loc}(\C^{m+n})$ to some Ricci-flat
K\"ahler form $\tilde\omega^\bullet_\infty$ on $\C^{m+n}$ which is uniformly equivalent to $\omega_{\C^{m+n}}$ by \eqref{caz44} but satisfies $\tilde\omega_\infty^\bullet(0) \neq \tilde\omega_\infty^\bullet(\tilde{x}_\infty')$ by \eqref{caz45} (using Remark \ref{tiammazzo2}). This is impossible because of the Liouville Theorem \ref{lio1}.

\subsection{Case 2: the blowup is $\C^m \times Y$} In this case we have that $\delta_t\to 1$, without loss of generality.
The argument here is closely analogous to Case 1 but slightly easier, so we will sketch it more briefly. The main simplification is that there is now no need to apply an additional diffeomorphism and pass from $\hat{J}_t$ to $\tilde{J}_t$, etc. Indeed, we now have that
$$ \hat{J}_t, \hat{J}_t^\natural \to J_{\C^m} + J_{Y,z_\infty}, \;\, \hat{g}_t, \hat{g}_t^\natural \to g_{\C^{m}} + g_{Y,z_\infty}\;\,{\rm in}\;\,C^\infty_{\rm loc}(\C^m \times Y).$$
Thanks to \eqref{caz2}, we can apply Proposition \ref{localest3} on small balls to obtain $C^\infty_{\rm loc}(\C^m \times Y)$ bounds for $\hat\omega_t^\bullet$, hence (after passing to a further subsequence) $C^\infty_{\rm loc}(\C^m \times Y)$ convergence to a Ricci-flat K\"ahler metric $\hat\omega_\infty^\bullet$ on $\C^m \times Y$ uniformly equivalent to $g_P = g_{\C^m} + g_{Y,z_\infty}$. Because of these $C^\infty_{\rm loc}(\C^m \times Y)$ bounds (and using Lemma \ref{tiammazzo1} to compare H\"older norms), the two points $\hat{x}_t = (0,\hat{y}_t)$ and $\hat{x}_t'$ have $\hat{g}_t$-distance uniformly bounded away from zero. Thus,
\begin{align*}
\frac{|\hat\omega_\infty^\bullet(\hat{x}_\infty) -  \mathbf{P}^{g_P}_{\hat{\gamma}_\infty}(\hat\omega_\infty^\bullet(\hat{x}'_\infty))|_{g_P(\hat{x}_\infty)}}{d^{g_P}(\hat{x}_\infty,\hat{x}')^\alpha} = 1,
\end{align*}
after passing to a subsequence so that $\hat{x}_t' \to \hat{x}_\infty' \neq \hat{x}_\infty$. (Here we again use Remark \ref{tiammazzo2}. Note that $\hat\gamma_\infty$ is \emph{some} minimal geodesic connecting $\hat{x}_\infty$ and $\hat{x}_\infty'$; there may very well be others but for the subsequent contradiction this is irrelevant.) This contradicts the Liouville Theorem \ref{lioHein}.

\subsection{Case 3: the blowup is $\C^m$} We finally assume that $\delta_t\to 0$. This is the hardest case. We begin by explaining the strategy. Fix any $R > 0$ such that $R > d^{\hat{g}_t}(\hat{x}_t, \hat{x}_t')$ for all $t$. Recall that $\hat{x}_t = (0, \hat{y}_t)$ and $\hat{y}_t \to \hat{y}_\infty \in Y$. Then \eqref{caz2}, \eqref{saves_our_asses2} imply that $\hat\eta_t$, and hence $\hat\omega_t^\bullet$, have uniformly bounded $C^\alpha$ norm with respect to any \emph{fixed} (i.e., non-collapsing) reference metric on $B^{\hat{g}_t}(\hat{x}_t, R)$, using Remark \ref{important}.
Thus, by passing to a diagonal sequence, we can assume that $\hat\omega_t^\bullet \to \hat\omega_\infty^\bullet$ in the $C^\beta_{\rm loc}(\C^m \times Y)$ topology for all $\beta \in (0,\alpha)$, where $\hat\omega_\infty^\bullet \in C^\alpha_{\rm loc}(\C^m \times Y)$ satisfies the following properties:
\begin{itemize}
\item[(1)] $\hat\omega_\infty^\bullet$ is a section of  ${\rm pr}_{\C^m}^*(\Lambda^{1,1}\C^m)$, uniformly equivalent to ${\rm pr}_{\C^m}^*(\omega_{\C^m})$.
\item[(2)] $\hat\omega_\infty^\bullet$ is $g_{Y,z_\infty}$-parallel in the fiber directions.
\item[(3)] $\hat\omega_\infty^\bullet$ is weakly closed.
\end{itemize}
Here (1), (2) follow by passing to the limit in \eqref{caz2}, \eqref{saves_our_asses2}, respectively (using Remark \ref{tiammazzo2}). For (2), notice in particular that if we fix any $\hat{z} \in B_R \subset \C^m$ and recall that $z=z_t+\lambda_t^{-1}\hat{z}$, then $g_{Y,z} \to g_{Y,z_\infty}$ as $t \to \infty$ because $z_t \to z_\infty$ and $|z - z_t|$ $<$ \begin{small}$\lambda_t^{-1}$\end{small}$R$.
(3) is clear since $\hat\omega_\infty^\bullet$ is a uniform limit of closed forms. Together, (1), (2), (3) imply that $\hat\omega_\infty^\bullet$ is the pullback under ${\rm pr}_{\C^m}$ of a weakly closed $(1,1)$-form of class $C^{\alpha}_{\rm loc}$ on $\C^m$, uniformly equivalent to $\omega_{\C^m}$. Abusing notation, we denote this form, which is in particular a K\"ahler current with a global potential of class $C^{2,\alpha}_{\rm loc}(\C^m)$, by $\hat\omega_\infty^\bullet$ as well.

As in Section \ref{claim3product}, it is not hard to see by adapting an  argument of \cite{To} that the complex Monge-Amp\`ere equation \eqref{mafiber} satisfied by $\omega_t^\bullet$ implies that $\hat\omega_\infty^\bullet$ has volume form $c \omega_{\C^m}^m$. Given this, it follows from a standard elliptic bootstrap that $\hat\omega_\infty^\bullet$ is smooth, and is therefore constant by Theorem \ref{lio1}. All of this will be proved in Claim 3 below (Section \ref{claim3nonproduct}). The main difficulty of this section consists in deducing from \eqref{mafiber}, \eqref{caz2}, \eqref{blow22}, \eqref{saves_our_asses2} that $\hat\omega_\infty^\bullet$ is {not} constant on $\C^m$, contradicting Claim 3.

For this we first need to rule out that $d^{\hat{g}_t}(\hat{x}_t, \hat{x}_t') \to 0$. This will be done in Claim 1 (Section \ref{claim1nonproduct}). One way to prove this would be to improve the $C^\alpha$ type bound \eqref{saves_our_asses2} for $\hat\eta_t$ to a $C^\beta$ type bound for some $\beta > \alpha$, by linearizing the Monge-Amp\`ere equation as in Section \ref{claim1product} and using a Schauder estimate. It seems possible but cumbersome to prove a version of Theorem \ref{moron} that accomplishes this. Instead we will argue by contradiction: if  $d^{\hat{g}_t}(\hat{x}_t, \hat{x}_t') \to 0$, then by mimicking the proof of Theorem \ref{t:schauder} we can produce a harmonic $(1,1)$-form contradicting Liouville's theorem on $\C^{m+n}$, $\C^m \times Y$, or $\C^m$.

Second, we need to show that the nontrivial difference quotient \eqref{blow22} passes to the limit. Currently we have $\hat\omega_t^\bullet \to \hat\omega_\infty^\bullet$ in \begin{small}$C^\beta_{\rm loc}$\end{small}$(\C^m \times Y)$ for all $\beta < \alpha$, but this convergence is too weak because \eqref{blow22} may be due to base-fiber or fiber-fiber components, which go to zero in \begin{small}$C^{\beta}_{\rm loc}$\end{small}$(\C^m \times Y)$. In Claim 2 (Section \ref{claim2nonproduct}) we will prove using the $i\partial\ov{\partial}$-exactness of $\hat\eta_t$ that \eqref{blow22} is entirely due to base-base components. In fact, the same argument is already needed to treat the $\C^m$ subcase of the blowup proof of Claim 1, but we defer this argument to Claim 2 because it is long and involved.

Claims 1, 2, 3 imply a contradiction, which completes Case 3, hence the proof of Theorem \ref{t:fornow}.

\subsubsection{}\label{claim1nonproduct} {\bf Claim 1.} There exists an $\epsilon>0$ such that for all $t$ it holds that
$$d^{\hat{g}_t}(\hat{x}_t, \hat{x}_t') \geq \epsilon.$$

\noindent \emph{Proof of Claim 1}. If this was false, then, since $\hat{x}_t \neq \hat{x}_t'$ for all $t$, there would exist a sequence $t_i \to \infty$ such that $d_{t_i} = d^{\hat{g}_{t_i}}(\hat{x}_{t_i}, \hat{x}_{t_i}') \to 0$. As usual, we will pretend that $d_t = d^{\hat{g}_t}(\hat{x}_t, \hat{x}_t') \to 0$.

Consider the diffeomorphisms
$$\Theta_t:  B_{d_t^{-1}\lambda_t} \times Y \to B_{\lambda_t} \times Y, \;\, (\hat{z},\hat{y}) = \Theta_t(\tilde{z},\tilde{y}) = (d_t \tilde{z}, \tilde{y}).$$
Pull back $\hat{J}_t, \hat{g}_t, \hat{g}_{\hat{z},t}, \hat{J}_t^\natural, \hat{\omega}_t^\natural, \hat{\eta}_t, \hat{\omega}_t^\bullet, \hat{x}_t, \hat{x}_t'$ under $\Theta_t$, multiply the metrics and $2$-forms by $d_t^{-2}$, and denote the resulting objects by the same letters with each hat replaced by a tilde. Then first of all
\begin{align}\tilde{g}_{\tilde{z},t} = g_{\C^m} +  \epsilon_t^2 g_{Y,z}\;\,(\epsilon_t = d_t^{-1}\delta_t,  z = d_t \lambda_t^{-1}\tilde{z}),\;\,
\tilde{\omega}_t^\natural = \omega_{\C^m} + \epsilon_t^2 \Theta_t^*\Psi_t^*\omega_F.
\end{align}
Secondly, thanks to \eqref{mafiber},
\begin{align}
(\tilde{\omega}_t^\natural + \tilde{\eta}_t)^{m+n} = c_t e^{\tilde{G}_t + \tilde{H}_t}(\tilde{\omega}_t^\natural)^{m+n},\label{zumteufel151}
\end{align}
where the constants $c_t$ have a positive limit and
\begin{align}\label{calmund}
\tilde{G}_t = \Theta_t^*\Psi_t^*G, \;\,\tilde{H}_t =  \log \frac{\omega_{\C^m}^m \wedge (\epsilon_t^2\Theta_t^*\Psi_t^*\omega_F)^n}{(\omega_{\C^m} + \epsilon_t^2 \Theta_t^*\Psi_t^*\omega_F)^{m+n}}.
\end{align}
Next, from \eqref{caz2}, \eqref{blow22}, \eqref{saves_our_asses2},
\begin{align}
|\tilde\eta_t|_{\tilde{g}_t} \leq C, \label{whoknows111} \\
\sup_{\tilde{x} = (\tilde{z},\tilde{y}) \in B^{\tilde{g}_t}(\tilde{x}_t,d_t^{-1})} \Biggl(\sup_{\tilde{x}' \in B^{\tilde{g}_{t}}(\tilde{x}_t,d_t^{-1})} \frac{|\tilde\eta_t(\tilde{x}) -  \mathbf{P}^{\tilde{g}_{\tilde{z},t}}_{\tilde{x}'\tilde{x}}(\tilde\eta_t(\tilde{x}'))|_{\tilde{g}_{t}(\tilde{x})}}{d^{\tilde{g}_{t}}(\tilde{x},\tilde{x}')^\alpha}\Biggr)
\leq Cd_t^\alpha,\label{whoknows112}\\
\frac{|\tilde\eta_t(\tilde{x}_t) -  \mathbf{P}^{\tilde{g}_t}_{\tilde{x}'_t\tilde{x}_t}(\tilde\eta_t(\tilde{x}'_t))|_{\tilde{g}_t(\tilde{x}_t)}}{d^{\tilde{g}_t}(\tilde{x}_t,\tilde{x}'_t)^\alpha}
=d_t^\alpha,\label{zumteufel161}\\
d^{\tilde{g}_t}(\tilde{x}_t,\tilde{x}_t') = 1.\label{zumteufel162}
\end{align}
Roughly speaking, \eqref{whoknows111}, \eqref{whoknows112} say that $\tilde{\eta}_t$ is asymptotic to a parallel form as $t \to \infty$. The remainder of this section is concerned with proving that after subtracting this parallel form and dividing by $d_t^\alpha$, we can extract a limiting $(1,1)$-form on $\C^{m+n}$, $\C^m \times Y$, or $\C^m$ that is harmonic thanks to \eqref{zumteufel151} and contradicts Liouville's theorem thanks to \eqref{zumteufel161}, \eqref{zumteufel162}.

The first step is to subtract the parallel part of $\tilde\eta_t$ in the $\C^m$-directions (this is already enough when the blowup is $\C^m \times Y$ or $\C^m$, but a further subtraction will be required when the blowup is $\C^{m+n}$). More precisely, decompose $\tilde{\eta}_t = \tilde{\eta}_t^\sharp + \tilde{\eta}_t'$,
where $\tilde{\eta}_t^\sharp$ is the unique $\tilde{g}_t$-parallel $(1,1)$-form pulled back from $\C^m$ such that $\tilde\eta_t^\sharp(\tilde{x}_t)$ is the $\tilde{g}_t(\tilde{x}_t)$-orthogonal projection of $\tilde\eta_t(\tilde{x}_t)$ onto ${\rm pr}_{\C^m}^*(\Lambda^{1,1}\C^m)|_{\tilde{x}_t}$. Notice that $\tilde{\eta}_t^\sharp$ is trivially $\ddbar$-exact (on $\C^m$, hence on $\C^m \times Y$), and is parallel with respect to all the metrics $\ti{g}_{\ti{z},t}$. \medskip\

\noindent \emph{Subclaim 1.1}. There exists a $C$ such that for all $t$,
\begin{align}|d_t^{-\alpha}\tilde{\eta}_t'(\tilde{x}_t)|_{\tilde{g}_t(\tilde{x}_t)} \leq C\epsilon_t^\alpha.\label{kotz15}\end{align}

\noindent \emph{Proof of Subclaim 1.1}. Assume the statement is false, i.e., without loss of generality, that $$\sigma_t = \epsilon_t^{-\alpha}|d_t^{-\alpha}\tilde{\eta}_t'(\tilde{x}_t)|_{\tilde{g}_t(\tilde{x}_t)} \to \infty.$$
Consider the diffeomorphism
$$\Pi_t: B_{\delta_t^{-1}\lambda_t} \times Y \to B_{d_t^{-1}\lambda_t} \times Y, \;\,(\tilde{z},\tilde{y}) = \Pi_t(\check{z},\check{y}) = (\epsilon_t\check{z}, \check{y}).$$
Define new reference product metrics
$$\check{g}_{\check{z},t} = \epsilon_t^{-2}\Pi_t^*\tilde{g}_{\tilde{z},t} = g_{\C^m} + g_{Y,z} \;\,(\tilde{z} = \epsilon_t \check{z}), \;\, \check{g}_t = \epsilon_t^{-2}\Pi_t^*\tilde{g}_{t} = g_{\C^m} + g_{Y,z_t},$$
as well as a new $2$-form
$$\check{\eta}_t' =  \sigma_t^{-1}\epsilon_t^{-2-\alpha}\Pi_t^*(d_t^{-\alpha}\tilde{\eta}_t').$$
Then by definition and by \eqref{whoknows112}, using the fact that $\tilde{\eta}_t^\sharp$ is $\tilde{g}_{\tilde{z},t}$-parallel for all $\tilde{z}$,
\begin{align}
|\check{\eta}_t'(\check{x}_t)|_{\check{g}_t(\check{x}_t)} = 1,\label{kotz1}\\
\sup_{\check{x} = (\check{z},\check{y}) \in B^{\check{g}_t}(\check{x}_t,\delta_t^{-1})} \Biggl(\sup_{\check{x}' \in B^{\check{g}_{t}}(\check{x}_t,\delta_t^{-1})} \frac{|\check\eta_t'(\check{x}) -  \mathbf{P}^{\check{g}_{\check{z},t}}_{\check{x}'\check{x}}(\check\eta_t'(\check{x}'))|_{\check{g}_{t}(\check{x})}}{d^{\check{g}_{t}}(\check{x},\check{x}')^\alpha}\Biggr)
\leq C\sigma_t^{-1},\label{kotz2}
\end{align}
where $\check{x}_t = \Pi_t^{-1}(\tilde{x}_t) = (0, \check{y}_t)$.  Now \eqref{kotz1}, \eqref{kotz2} first of all imply that
\begin{align}
\label{kotz3}
|\check\eta'_t(\check{x}')|_{\check{g}_t(\check{x}')} \leq 1 + C\sigma_t^{-1}d^{\check{g}_t}(\check{x}_t, \check{x}')^\alpha
\end{align}
for all $\check{x}' \in B^{\check{g}_t}(\check{x}_t, \delta_t^{-1})$, i.e., $\check{\eta}'_t$ is locally uniformly bounded on $\C^m \times Y$. Moreover, since the metrics $\check{g}_{\check{z},t}$ are non-collapsing and hence uniformly smoothly equivalent to each other, it follows from \eqref{kotz2}, \eqref{kotz3} that $\check{\eta}'_t$ is locally uniformly bounded in $C^\alpha$ on $\C^m \times Y$ (using also Lemma \ref{tiammazzo1}). Thus, by Ascoli-Arzel\`a, up to passing to a subsequence, $\check\eta'_t \to \check\eta'_\infty$ in $C^\beta_{\rm loc}(\C^m \times Y)$ for all $\beta < \alpha$. By \eqref{kotz1}, $\check\eta'_\infty \neq 0$. By \eqref{kotz2} and Remark \ref{tiammazzo2}, $\check\eta'_\infty$ is parallel with respect to $\check{g}_\infty = g_{\C^m} + g_{Y,z_\infty}$, hence in particular smooth. Moreover, $\check\eta_\infty'$ is $i\partial\ov{\partial}$-exact with respect to $J_{\C^m} + J_{Y,z_\infty}$ by the same arguments as in Proposition \ref{p:closure} (some of the $\ddbar$-operators must now be understood with respect to \begin{small}$\check{J}^\natural_t$\end{small} $=$ $\Pi_t^*\Theta_t^*\Psi_t^*J^\natural$, but this does not affect the arguments because \begin{small}$\check{J}_t^\natural$\end{small} $\to$ $J_{\C^m} + J_{Y,z_\infty}$ locally smoothly). In particular, $\check\eta_\infty'|_{\{\check{z}\}\times Y}$ is harmonic and exact for all $\check{z} \in \C^m$, hence zero, so that any global $i\partial\ov{\partial}$-potential of $\check\eta_\infty'$ must be pulled back from the base. Thus, $\check\eta_\infty'$ is a nonzero parallel $(1,1)$-form pulled back from $\C^m$, contradicting the fact that $(\check\eta_\infty')(\check{x}_\infty)$ is orthogonal to the values of all such forms at $\check{x}_\infty$ by construction.
\hfill $\Box$\medskip\

The role of \eqref{kotz15} together with \eqref{whoknows112}, \eqref{zumteufel161}, \eqref{zumteufel162} is to pass $d_t^{-\alpha}\tilde\eta_t'$ to a limiting $i\partial\ov{\partial}$-exact $(1,1)$-form on $\C^{m+n}$, $\C^m \times Y$, or $\C^m$ that is $O(r^\alpha)$ at infinity and not parallel. (In the $\C^{m+n}$ case, \eqref{kotz15} is not sufficient for this and further subtractions are required.) On the other hand, we will deduce from \eqref{zumteufel151} that this limit form has constant trace, contradicting Liouville's theorem.

Before entering into the three cases of the blowup argument, we first prove some important partial estimates that will help us establish the constant trace property in all cases.

Fix any $R > 0$. It follows from \eqref{kotz15} and \eqref{whoknows112} that
\begin{align}\label{caz699}
|\tilde\eta_t'|_{\tilde{g}_t} \leq C \delta_t^\alpha + C d_t^\alpha R^\alpha\;\,{\rm on}\;\,B^{\tilde{g}_t}(\tilde{x}_t,R).
\end{align}
Note that, crucially, the right-hand side goes to zero as $t\to\infty$ for $R$ fixed.
Define $\tilde\omega_t^\sharp = \tilde\omega_t^\natural + \tilde\eta_t^\sharp$. Since $\tilde\omega_t^\sharp = \tilde\omega_t^\bullet - \tilde\eta_t'$, we can see using  \eqref{caz2b}, \eqref{caz699}, and the fact that
$$|\tilde{J}_t^\natural|_{\tilde{g}_t} \leq C|\tilde{J}_t^\natural|_{\tilde{g}_t^\natural} \leq C,$$
 that there exists a $t_R$ such that $\tilde\omega_t^\sharp$ is a K\"ahler form on $B^{\tilde{g}_t}(\tilde{x}_t, R)$ for all $t \geq t_R$, with associated metric uniformly equivalent to $\tilde{g}_t$. This allows us to expand the Monge-Amp\`ere equation \eqref{zumteufel151} as
\begin{align}\label{satanas}
\tr{\tilde\omega_t^\sharp}{\tilde\eta_t'} + \sum_{i=2}^{m+n}\binom{m+n}{i}\frac{(\tilde\eta_t')^i \wedge (\tilde\omega_t^\sharp)^{m+n-i}}{(\tilde\omega_t^\sharp)^{m+n}}  = c_t e^{\tilde{G}_t + \tilde{H}_t} \frac{(\tilde\omega_t^\natural)^{m+n}}{(\tilde\omega_t^\sharp)^{m+n}} - 1.
\end{align}
Write $c_t e^{\tilde{K}_t} - 1$ for the right-hand side of \eqref{satanas}. \medskip\

\noindent \emph{Subclaim 1.2}. There is a $C$, and for all $R$ there is a $t_R$, such that for all $t \geq t_R$ and $\tilde{x}' \in B^{\tilde{g}_t}(\tilde{x}_t,R)$,
\begin{align}
\forall i \geq 2: d_t^{-\alpha}|(\tilde\eta_t')^i(\tilde{x}_t) - \mathbf{P}^{\tilde{g}_t}_{\tilde{x}'\tilde{x}_t}((\tilde\eta_t')^i(\tilde{x}'))|_{\tilde{g}_t(\tilde{x}_t)} \leq C(\delta_t^\alpha+ d_t^\alpha R^\alpha)R^\alpha,\label{caca1}\\
d_t^{-\alpha}|\tilde\omega_t^\sharp(\tilde{x}_t) - \mathbf{P}^{\tilde{g}_t}_{\tilde{x}'\tilde{x}_t}(\tilde\omega_t^\sharp(\tilde{x}'))|_{\tilde{g}_t(\tilde{x}_t)} \leq C d_t^{1-\alpha}\lambda_t^{-1}R, \label{caca2}\\
d_t^{-\alpha}|e^{\tilde{K}_t(\tilde{x}_t)} -e^{\tilde{K}_t(\tilde{x}')}| \leq C d_t^{1-\alpha}\lambda_t^{-1}R. \label{caca3}
\end{align}

\noindent \emph{Proof of Subclaim 1.2}. To prove \eqref{caca1} we take out a factor of $(\tilde\eta_t')(\tilde{x}_t) - \mathbf{P}^{\tilde{g}_t}_{\tilde{x}'\tilde{x}_t}((\tilde\eta_t')(\tilde{x}'))$ and estimate it using \eqref{whoknows112}. For the remaining factor we only need to estimate the length of $\tilde\eta_t'$ using \eqref{caz699}. (As a side remark, note that \eqref{caca1}, which effectively allows us to drop the nonlinearities of the Monge-Amp\`ere equation, is the only place in the proof of Claim 1 where we use our assumption that $d_t \to 0$.)

To prove \eqref{caca2} we pull $\tilde{g}_t, \tilde{\omega}_t^\natural, \tilde{\omega}_t^\sharp, \tilde{J}_t, \tilde{J}_t^\natural, \tilde{x}_t, \tilde{x}'$ back by the diffeomorphism $(\tilde{z},\tilde{y}) = (\epsilon_t\check{z}, \check{y})$ as in the proof of Subclaim 1.1 above. In addition we multiply all metrics and $2$-forms by $\epsilon_t^{-2}$. If we denote the resulting objects by a check instead of a tilde, the advantage of this construction is that
$$|\tilde\omega_t^\sharp(\tilde{x}_t) - \mathbf{P}^{\tilde{g}_t}_{\tilde{x}'\tilde{x}_t}(\tilde\omega_t^\sharp(\tilde{x}'))|_{\tilde{g}_t(\tilde{x}_t)} = |\check\omega_t^\sharp(\check{x}_t) - \mathbf{P}^{\check{g}_t}_{\check{x}'\check{x}_t}(\check\omega_t^\sharp(\check{x}'))|_{\check{g}_t(\check{x}_t)}=|\check\omega_t^\natural(\check{x}_t) - \mathbf{P}^{\check{g}_t}_{\check{x}'\check{x}_t}(\check\omega_t^\natural(\check{x}'))|_{\check{g}_t(\check{x}_t)},$$
since $\check{\eta}^\sharp_t$ is $\check{g}_t$-parallel,
but that $\check{g}_t = g_{\C^m} + g_{Y,z_t}$ is now essentially fixed. From the definition of $\check\omega_t^\natural$, and using \eqref{Jblock}, \eqref{gblock} and the covering argument indicated after \eqref{gblock}, we can rewrite this as
\begin{align}
|[{U}_t(0,\check{y}_t) - {U}_t(\epsilon_t d_t \lambda_t^{-1}\check{z}',\check{y}')] &\circledast (\epsilon_t d_t \lambda_t^{-1})d\check{z} \circledast (\epsilon_t d_t \lambda_t^{-1})d\check{z}\label{fucker}\\
+\, [V_t(0,\check{y}_t) - V_t'(\epsilon_t d_t \lambda_t^{-1}\check{z}',\check{y}')] &\circledast (\epsilon_t d_t \lambda_t^{-1})d\check{z} \circledast d\check{y}\label{fucker2}\\
-\, W_t'(\epsilon_t d_t \lambda_t^{-1}\check{z}',\check{y}') &\circledast d\check{y} \circledast d\check{y}|_{\check{g}_t(\check{x}_t)}.\label{fucker3}
\end{align}
Here ${U}_t, {V}_t, V_t', W_t'$ are smooth matrix-valued functions depending precompactly on $t$ in all $C^k$ norms (the primes indicate the effect of $g_{Y,z_t}$-parallel transport in the fiber directions), with $W_t'(0,\check{y}) = 0$ for all $\check{y}\in Y$.  Note that there is no $W_t$ term in \eqref{fucker3} because we used \eqref{gblock} where $E(z,z,y)=0$.

Recall that $(\tilde{z}',\tilde{y}') = (\epsilon_t\check{z}',\check{y}') \in B^{\tilde{g}_t}(\tilde{x}_t, R)$, where $R$ is fixed, so $|\check{z}'|\leq C\ve_t^{-1}R$. To estimate the terms in \eqref{fucker}--\eqref{fucker3}, first note that $\ve_td_t\lambda_t^{-1}=\delta_t\lambda_t^{-1}=e^{-t/2}$. Taylor expansion in the $z$-variables allows us to bound \eqref{fucker3} by $C\ve_t^{-1}R e^{-t/2}=Cd_t\lambda_t^{-1}R$. For \eqref{fucker2}, notice that $V_t(0,\check{y}_t)=V'_t(0,\check{y}_t)$, so (since we are working in a fixed small chart in $Y$ where $g_{Y,z_t}$-parallel transport enjoys smooth dependence) we can again use Taylor expansion, as follows. In the $z$-variables the contribution to \eqref{fucker2} (including the form part) is bounded by $Cd_t\lambda_t^{-1}Re^{-t/2}$, while in the $y$-variables it is bounded by $|\check{y}_t-\check{y}'|e^{-t/2}$. The latter expression is in turn bounded by $Ce^{-t/2}$ because $\check{y}_t, \check{y}'$ lie in a fixed chart, and by $C\ve_t^{-1}Re^{-t/2}$ because $\check{y}'=\tilde{y}'$ lies in a metric ball of radius $R$ centered at $\check{y}_t=\ti{y}_t$ with respect to $\tilde{g}_t = g_{\C^m} + \ve_t^2g_{Y,z_t}$.
The line \eqref{fucker} can be estimated in exactly the same way and even has an additional helpful factor of $e^{-t/2}$.
Thus, we may estimate \eqref{fucker}--\eqref{fucker3} by
\begin{align*}
C(d_t\lambda_t^{-1}R + \min\{1,\epsilon_t^{-1}R\} )e^{-\frac{t}{2}} + Cd_t \lambda_t^{-1}R.
\end{align*}
The estimate \eqref{caca2} now follows by observing that $e^{-t/2} \leq d_t \lambda_t^{-1}R$ if $\epsilon_t^{-1}R \geq 1$.

Finally, there are three contributions to \eqref{caca3}: from $\tilde{G}_t$, from $\tilde{H}_t$, and from $(\tilde\omega_t^\natural)^{m+n}/(\tilde\omega_t^\sharp)^{m+n}$. The latter is easily controlled using \eqref{caca2}, the fact that $\tilde{g}_t^\sharp$ is uniformly comparable to $\tilde{g}_t$, and the fact that $\tilde\eta_t^\sharp$ is $\tilde{g}_t$-parallel and uniformly bounded thanks to \eqref{whoknows111}, \eqref{caz699}. Next, notice that
$$|\tilde{G}_t(\tilde{x}_t) - \tilde{G}_t(\tilde{x}')| = |G(0) - G(d_t\lambda_t^{-1}\tilde{z}')| \leq C d_t \lambda_t^{-1}R$$ by \eqref{calmund} and because $G$ is pulled back from the base. The $\tilde{H}_t$ contribution is more complicated but can be treated using the same method as in the proof of \eqref{caca2} above. Indeed, writing $(\ti{z},\ti{y})=\Pi_t(\check{z},\check{y}) = (\ve_t\check{z},\check{y})$, multiplying all $2$-forms by $\ve_t^{-2}$, and denoting the new objects by a check, we have
$$e^{\check{H}_t}=\frac{\omega_{\C^m}^m\wedge(\Pi_t^*\Theta_t^*\Psi_t^*\omega_F)^n}{(\omega_{\C^m}+\Pi_t^*\Theta_t^*\Psi_t^*\omega_F)^{m+n}}=
\biggl(\binom{m+n}{m}+\sum_{j=0}^{m-1}\binom{m+n}{j}\frac{\omega_{\C^m}^j\wedge(\Pi_t^*\Theta_t^*\Psi_t^*\omega_F)^{m+n-j}}{\omega_{\C^m}^m\wedge(\Pi_t^*\Theta_t^*\Psi_t^*\omega_F)^n}\biggr)^{-1}.$$
Writing $\omega_F$  as a block matrix $(\begin{smallmatrix} A & B \\ C & D \end{smallmatrix})$ with smooth entries in the original coordinates $(z,y)$, we see that $\omega_{\C^m}^m\wedge(\Pi_t^*\Theta_t^*\Psi_t^*\omega_F)^n$ is equal to the pullback of $\det D$, while all the terms $\omega_{\C^m}^j\wedge(\Pi_t^*\Theta_t^*\Psi_t^*\omega_F)^{m+n-j}$ are multiplied by a factor of $(\epsilon_t d_t \lambda_t^{-1})^2 =e^{-t}$ thanks to our pullback maps. Thus,
$$e^{\check{H}_t}=\frac{m!n!}{(m+n)!}+e^{-t}\Pi_t^*\Theta_t^*\Psi_t^*U_t$$ for some smooth function $U_t$ depending precompactly on $t$ in all $C^k$ norms, and hence
$$|e^{\check{H}_t(\check{x}_t)}-e^{\check{H}_t(\check{x}')}| = |U_t(0,\check{y}_t)-{U}_t(\epsilon_t d_t \lambda_t^{-1}\check{z}',\check{y}')|e^{-t}.$$
The latter term can be bounded in analogy with our estimate of \eqref{fucker} above.
 \hfill $\Box$\medskip\

As a direct consequence of Subclaim 1.2 together with \eqref{satanas}, we see that there exists a $C$, and for all $R$ there exists a $t_R$, such that for all $t \geq t_R$ and $\tilde{x}' \in B^{\tilde{g}_t}(\tilde{x}_t,R)$,
\begin{equation}\label{killmenow}
d_t^{-\alpha}(\tr{\tilde\omega_t^\sharp}{\tilde\eta_t'}(\tilde{x}_t) - \tr{\tilde\omega_t^\sharp}{\tilde\eta_t'}(\tilde{x}')) \leq C((\delta_t^\alpha + d_t^\alpha R^\alpha)R^\alpha + d_t^{1-\alpha}\lambda_t^{-1}R).
\end{equation}
We are now in position to derive a contradiction on each of the three possible blowup spaces. Subcase B shows the key idea without any technical complications, while A and C are more involved.\medskip\

\noindent {\bf Subcase A: $\epsilon_t \to \infty$. The limit of $(B_{d_t^{-1}\lambda_t} \times Y,\tilde{g}_t,\tilde{x}_t)$ is $\C^{m+n}$.} \medskip\

\noindent \emph{Deriving a contradiction in Subcase A}. The complication of this case compared to Subcase B is that \eqref{kotz15} does not provide a uniform bound on $|d_t^{-\alpha}\tilde\eta_t'(\tilde{x}_t)|_{\tilde{g}_t(\tilde{x}_t)}$.
To fix this, recall that $\tilde{g}_t = g_{\C^m} + \epsilon_t^2 g_{Y,z_t}$, where $g_{Y,z_t} \to g_{Y,z_\infty}$ smoothly.
Let $\mathbf{x}^{2m+1}, \ldots, \mathbf{x}^{2m+2n}$ be normal coordinates for $g_{Y,z_t}$ centered at $y_t$. Viewed as a map from $Y$ to $\R^{2n}$ these depend on $t$, but we prefer to instead pull back our setup to $\mathbb{R}^{2n}$ under the inverse map. In this sense we may then assume without loss that
\begin{align*}
\left|\frac{\partial^j}{\partial\mathbf{x}^j}(g_{Y,z_t}(\mathbf{x})_{ab}-\delta_{ab})\right| \leq \frac{1}{100}|\mathbf{x}|^{2-j}\,\;{\rm for}\;\, |\mathbf{x}| \leq 2\;\,{\rm and}\;\,j=0,1.
\end{align*}
This is possible thanks to the compactness of $Y$ and compact dependence of $g_{Y,z_t}$ on $t$, after rescaling all metrics $g_{Y,z_t}$ by the same large constant if necessary. Define $\tilde{\mathbf{x}}^j = \epsilon_t \mathbf{x}^j$, so that $\tilde{\mathbf{x}}^{2m+1}, \ldots, \tilde{\mathbf{x}}^{2m+2n}$ are normal coordinates for $\epsilon_t^2 g_{Y,z_t}$ centered at $y_t$.
Formally also write $\tilde{\mathbf{x}}^1, \ldots, \tilde{\mathbf{x}}^{2m}$ for the standard real coordinates on $\C^m$. Then $\tilde{\mathbf{x}}^1, \ldots, \tilde{\mathbf{x}}^{2m+2n}$ are normal coordinates for $\tilde{g}_{t}$ centered at $\tilde{x}_t$ with
\begin{align}\label{e:normal_coords23}
\left|\frac{\partial^j}{\partial\tilde{\mathbf{x}}^j}(\tilde{g}_{t}(\tilde{\mathbf{x}})_{ab}-\delta_{ab})\right| \leq \frac{\epsilon_t^{-2}}{100}|\tilde{\mathbf{x}}|^{2-j}\,\;{\rm for}\;\, |\tilde{\mathbf{x}}| \leq 2\epsilon_t\;\,{\rm and}\;\,j=0,1.
\end{align}
Then let $(\tilde{\eta}_t')^{\sharp} \in \mathcal{A}^2(B^{\tilde{g}_t}(\tilde{x}_t,\epsilon_t))$ denote the $0$-th order Taylor polynomial of $\tilde{\eta}'_t$ at $\tilde{x}_t$ with respect to the coordinate system $\tilde{\mathbf{x}}^1, \ldots, \tilde{\mathbf{x}}^{2m+2n}$, and define $\tilde{\eta}_t'' = \tilde{\eta}_t' - (\tilde{\eta}_t')^{\sharp} \in \mathcal{A}^2(B^{\tilde{g}_t}(\tilde{x}_t,\epsilon_t))$.\medskip\

\noindent \emph{Subclaim 1.3}. There is a $C$, and for all $R$ there is a $t_R$, such that for all $t \geq t_R$,
\begin{align}
d_t^{-\alpha}\tilde\eta_t''(\tilde{x}_t) = 0,\label{???}\\
\sup_{\tilde{x} = (\tilde{z},\tilde{y}) \in B^{\tilde{g}_t}(\tilde{x}_t,R)} \Biggl(\sup_{\tilde{x}' \in B^{\tilde{g}_{t}}(\tilde{x}_t,R)} \frac{d_t^{-\alpha}|\tilde\eta_t''(\tilde{x}) -  \mathbf{P}^{\tilde{g}_{\tilde{z},t}}_{\tilde{x}'\tilde{x}}(\tilde\eta_t''(\tilde{x}'))|_{\tilde{g}_{t}(\tilde{x})}}{d^{\tilde{g}_{t}}(\tilde{x},\tilde{x}')^\alpha}\Biggr)
\leq C,\label{e:fucking_bounds_1a_prime_prime23}\\
|d_t^{-\alpha}|\tilde{\eta}_t''(\tilde{x}_t) - \mathbf{P}_{\tilde{x}_t'\tilde{x}_t}^{\tilde{g}_t}[\tilde{\eta}_t''(\tilde{x}_t')] |_{\tilde{g}_t(\tilde{x}_t)} - 1 | \leq C\epsilon_t^{\alpha-1},\label{e:fucking_bounds_2_prime_prime23}\\
d^{\tilde{g}_t}(\tilde{x}_t,\tilde{x}_t') = 1,\label{zumteufel162a}\\
d_t^{-\alpha}(\tr{\tilde{\omega}^\sharp_t}{\tilde{\eta}_t''}(\tilde{x}_t) - \tr{\tilde{\omega}^\sharp_t}{\tilde{\eta}_t''}(\tilde{x}') ) \leq C((\delta_t^\alpha + d_t^\alpha R^\alpha)R^\alpha + d_t^{1-\alpha}\lambda_t^{-1}R +
\epsilon_t^{\alpha-1}R),\label{e:fucking_bounds_1b_prime_prime23}
\end{align}
where in the last equation $\tilde{x}' \in B^{\tilde{g}_t}(\tilde{x}_t,R)$ is arbitrary.\medskip\

\noindent \emph{Proof of Subclaim 1.3}. \eqref{???} is clear by definition and \eqref{zumteufel162a} is copied from \eqref{zumteufel162}. We will now derive \eqref{e:fucking_bounds_1a_prime_prime23}, \eqref{e:fucking_bounds_2_prime_prime23}, \eqref{e:fucking_bounds_1b_prime_prime23} from \eqref{whoknows112}, \eqref{zumteufel161}, \eqref{killmenow} by using \eqref{kotz15}, \eqref{caca2}, \eqref{e:normal_coords23}.

Equations \eqref{whoknows112} and \eqref{killmenow} give control over $B^{\tilde{g}_t}(\tilde{x}_t,R)$ provided that $t \geq t_R$ is large enough. Thus, it makes sense to apply \eqref{whoknows112}, \eqref{zumteufel161}, \eqref{killmenow} to prove \eqref{e:fucking_bounds_1a_prime_prime23}, \eqref{e:fucking_bounds_2_prime_prime23}, \eqref{e:fucking_bounds_1b_prime_prime23}. Since we are also going to use \eqref{e:normal_coords23}, we moreover need to choose $t_R$ so large that $t \geq t_R$ implies $\epsilon_t \geq \max\{2,2R\}$.

The following is the key point: if $t \geq t_R$, then for all $\tilde{x},\tilde{x}' \in B^{\tilde{g}_t}(\tilde{x}_t,R)$, $\tilde{x} = (\tilde{z},\tilde{y})$, it holds that
\begin{align}\label{satanas42}
d_t^{-\alpha}|(\tilde\eta_t')^\sharp(\tilde{x}) -  \mathbf{P}^{\tilde{g}_{\tilde{z},t}}_{\tilde{x}'\tilde{x}}((\tilde\eta'_t)^\sharp(\tilde{x}'))|_{\tilde{g}_{t}(\tilde{x})} \leq C\epsilon_t^{\alpha-1}d^{\tilde{g}_t}(\tilde{x}, \tilde{x}').
\end{align}
To prove this, note that $g_{Y,z}$ will be at bounded distance to $g_{Y,z_t}$ in $C^1(Y)$ if $t \geq t_R$. Thus,
\begin{align}\label{!!!}
|\nabla^{\tilde{g}_t}\tilde{g}_{\tilde{z},t}|_{\tilde{g}_t} \leq C\epsilon_t^{-1} \;\,{\rm on}\;\,B^{\tilde{g}_t}(\tilde{x}_t,2R).
\end{align}
Moreover, once $t \geq t_R$ is sufficiently large, there will be a unique $\tilde{g}_{\tilde{z},t}$-minimal geodesic joining $\tilde{x}$ to $\tilde{x}'$, and this geodesic will be contained in $B^{\tilde{g}_t}(\tilde{x}_t,2R)$. Bound the left-hand side of \eqref{satanas42} by integrating the $\tilde{g}_{\tilde{z},t}$-covariant derivative of $d_t^{-\alpha}(\tilde\eta_t')^\sharp$ along this geodesic. We have $\partial_{\tilde{\mathbf{x}}}(\tilde\eta_t')^\sharp = 0$, the Christoffel
symbols satisfy $|\Gamma^{\tilde{g}_{\tilde{z},t}}|\leq C\epsilon_t^{-1}$ by \eqref{e:normal_coords23}, \eqref{!!!}, and $|d_t^{-\alpha}\tilde{\eta}_t'(\tilde{x}_t)|_{\tilde{g}_t(\tilde{x}_t)}  \leq C\epsilon_t^\alpha$ by \eqref{kotz15}. This implies \eqref{satanas42}.

\eqref{e:fucking_bounds_1a_prime_prime23} is clear from \eqref{whoknows112} and \eqref{satanas42}, and
\eqref{e:fucking_bounds_2_prime_prime23} is clear from  \eqref{zumteufel161} and \eqref{satanas42}. To prove \eqref{e:fucking_bounds_1b_prime_prime23}, we combine \eqref{killmenow} and \eqref{satanas42}. Specifically, thanks to \eqref{killmenow}, to prove \eqref{e:fucking_bounds_1b_prime_prime23} it suffices to prove that
\begin{equation}\label{teguida}
d_t^{-\alpha}(\tr{\tilde{\omega}^\sharp_t}{(\tilde\eta_t')^\sharp}(\tilde{x}_t) - \tr{\tilde{\omega}^\sharp_t}{(\tilde\eta_t')^\sharp}(\tilde{x}') ) \leq C(d_t^{1-\alpha}\lambda_t^{-1}R+\epsilon_t^{\alpha-1}R).
\end{equation}
Using the fact that $\ti{g}^\sharp_t$ is uniformly equivalent to $\ti{g}_t$, we can bound the left-hand side of \eqref{teguida} by
$$d_t^{-\alpha}|(\tilde\eta_t')^\sharp(\tilde{x}_t) -  \mathbf{P}^{\tilde{g}_t}_{\tilde{x}'\tilde{x}_t}((\tilde\eta'_t)^\sharp(\tilde{x}'))|_{\tilde{g}_{t}(\tilde{x_t})}+|(\tilde\eta'_t)^\sharp(\ti{x}_t)|_{\ti{g}_t(\ti{x}_t)}d_t^{-\alpha}|\tilde\omega_t^\sharp(\tilde{x}_t) - \mathbf{P}^{\tilde{g}_t}_{\tilde{x}'\tilde{x}_t}(\tilde\omega_t^\sharp(\tilde{x}'))|_{\tilde{g}_t(\tilde{x}_t)}.$$
Using \eqref{satanas42}, we can further bound the first term by $C\ve_t^{\alpha-1}R$. On the other hand, by \eqref{kotz15},
$$|(\tilde\eta'_t)^\sharp(\ti{x}_t)|_{\ti{g}_t(\ti{x}_t)}=|\tilde\eta'_t(\ti{x}_t)|_{\ti{g}_t(\ti{x}_t)}\leq C\delta_t^\alpha\leq C,$$
and this and \eqref{caca2} allow us to control the second term by $Cd_t^{1-\alpha}\lambda_t^{-1}R$, as desired.\hfill $\Box$\medskip\

Thanks to Subclaim 1.3 we are able to say that $d_t^{-\alpha}\tilde\eta_t''$ converges to some $2$-form $\tilde\eta_\infty'' \in C^\alpha_{\rm loc}(\C^{m+n})$ in the topology of $C$\begin{small}$^\beta_{\rm loc}$\end{small}$(\C^{m+n})$ for every $\beta<\alpha$ such that
$\tilde\eta_\infty''$ is $O(r^\alpha)$ at infinity and not parallel with respect to $g_{\C^{m+n}}$ (using Remark \ref{tiammazzo2}, but also Lemma \ref{tiammazzo1} to compare the $C^\alpha_{\rm loc}$ topologies with respect to a fixed and a mildly varying metric). Also, $\tilde\eta_\infty''$ is clearly $(1,1)$ with respect to $J_{\C^{m+n}}$.

On the other hand, we may assume that $\tilde{g}_t^\sharp \to g_{\C^{m+n}}^\sharp$ locally smoothly, where $g_{\C^{m+n}}^\sharp$ is a constant K\"ahler metric on $\C^{m+n}$ (possibly different from $g_{\C^{m+n}}$ but this is not relevant). Then it follows from \eqref{e:fucking_bounds_1b_prime_prime23} that $\tilde\eta_\infty''$ has constant trace with respect to $\omega_{\C^{m+n}}^\sharp$.

Notice that $\tilde\eta_\infty''$ is weakly closed (as a locally uniform limit of smooth closed forms). Thus, $\tilde\eta_\infty''$ is a closed $(1,1)$-current of class $C^\alpha_{\rm loc}(\C^{m+n})$, hence has a global potential of class $C^{2,\alpha}_{\rm loc}(\C^{m+n})$. This and the constant trace property imply that $\tilde\eta_\infty'' = i\partial\ov{\partial}\vp$ for some smooth function $\vp$ on $\C^{m+n}$ with
$$\Delta^{g_{\C^{m+n}}^\sharp}\vp = const.$$
Thus, $\vp=\ell+h$, where $\ell$ is a real polynomial of degree $\leq 2$ on $\C^{m+n}$ and $h$ is harmonic on $\C^{m+n}$ with $|\ddbar h|=O(r^\alpha)$. Liouville's theorem now tells us that the coefficient functions of $i\partial\ov{\partial} h$ are constant, which contradicts the fact that $\tilde\eta_\infty''$ is not parallel.\hfill $\Box$\medskip\

\noindent {\bf Subcase B:  $\epsilon_{t} \to 1$ (without loss). The limit of $(B_{d_t^{-1}\lambda_t} \times Y,\tilde{g}_t,\tilde{x}_t)$ is $\C^m \times Y$.}\medskip\

\noindent \emph{Deriving a contradiction in Subcase B}. Thanks to \eqref{caz699} together with \eqref{whoknows112}, \eqref{zumteufel161}, \eqref{zumteufel162} and Remark \ref{tiammazzo2}  (using also Lemma \ref{tiammazzo1} in order to compare the mildly varying $C^\alpha_{\rm loc}$ topologies) we are able to say that $d_t^{-\alpha}\tilde\eta_t'$ converges to some $2$-form $\tilde\eta_\infty' \in C^\alpha_{\rm loc}(\C^m \times Y)$ in the topology of $C$\begin{small}$^\beta_{\rm loc}$\end{small}$(\C^m \times Y)$ for all $\beta < \alpha$ such that $\tilde\eta_\infty'$ is $O(r^\alpha)$ at infinity and not parallel with respect to $g_{\C^m} + g_{Y,z_\infty}$. It is also clear at this point that $\tilde\eta_\infty'$ has type $(1,1)$ with respect to $J_{\C^m} + J_{Y,z_\infty}$.

On the other hand, we may clearly assume that $\tilde{g}_t^\sharp \to g_{\C^m}^\sharp + g_{Y,z_\infty}$ locally smoothly, where $g_{\C^m}^\sharp$ is a constant K\"ahler metric on $\C^m$ (possibly different from $g_{\C^m}$ but this is irrelevant). Then it follows from \eqref{killmenow} that $\tilde\eta_\infty'$ has constant trace with respect to $\omega_{\C^m}^\sharp + \omega_{Y,z_\infty}$.

Finally, notice that $\tilde\eta_\infty'$ is weakly closed (as a locally uniform limit of smooth closed forms). Thus, $\tilde\eta_\infty'$ is a closed $(1,1)$-current of class $C^\alpha_{\rm loc}(\C^m \times Y)$, hence has local potentials of class $C^{2,\alpha}$. Together with the constant trace property, this implies that $\tilde\eta_\infty'$ is actually smooth. Now the same arguments as in Proposition \ref{p:closure} (see also the proof of Subclaim 1.1 above) give that $\tilde\eta_\infty'$ is globally $i\partial\ov{\partial}$-exact. Thus, $\tilde\eta_\infty'=\ddbar \vp$ for some smooth function $\vp$ on $\C^m\times Y$ with
$$\Delta^{g_{\C^m}^\sharp + g_{Y,z_\infty}}\vp = const.$$
Thus, $\vp=\ell+h$, where $\ell$ is a real polynomial of degree $\leq 2$ on $\C^m$ and $h$ is harmonic on $\C^m\times Y$ with $|\ddbar h|=O(r^\alpha)$. Now the end of the proof of Proposition \ref{check} can be used to prove that $\tilde\eta_\infty'=\ddbar p$ for some real polynomial of degree $\leq 2$ on $\mathbb{C}^m$, contradicting the fact that $\tilde\eta_\infty'$ is not parallel.\hfill $\Box$\medskip\

\noindent {\bf Subcase C: $\epsilon_t \to 0$. The limit of $(B_{d_t^{-1}\lambda_t} \times Y,\tilde{g}_t,\tilde{x}_t)$ is $\C^m$.}\medskip\

\noindent \emph{Deriving a contradiction in Subcase C}. It follows from \eqref{caz699} and \eqref{whoknows112} that $d_t^{-\alpha}\tilde\eta_t'$ converges to some $2$-form $\tilde\eta_\infty' \in C^\alpha_{\rm loc}(\C^m \times Y)$ in the $C^\beta_{\rm loc}(\C^m \times Y)$ sense for all $\beta < \alpha$, using Remark \ref{important}.
Moreover:
\begin{itemize}
\item[(1)] $\tilde\eta_\infty'$ is a section of  ${\rm pr}_{\C^m}^*(\Lambda^{1,1}\C^m)$.
\item[(2)] $\tilde\eta_\infty'$ is $g_{Y,z_\infty}$-parallel in the fiber directions.
\item[(3)] $\tilde\eta_\infty'$ is weakly closed.
\item[(4)] $\tilde\eta_\infty'$ is $O(r^\alpha)$ with respect to $g_{\C^m}$ at infinity.
\end{itemize}
Here (1) and (2), (4) follow by passing to the limit in \eqref{caz699} and \eqref{whoknows112}, respectively (cf. Remark \ref{tiammazzo2}). For (2), note that if we fix $\tilde{z}$ $\in$ \begin{small}$B_{d_t^{-1}}$\end{small}$\subset$ $\C^m$, then $g_{Y,z} \to g_{Y,z_\infty}$ as $t \to \infty$ because $z_t \to z_\infty$ and $|z - z_t|$ $<$ \begin{small}$\lambda_t^{-1}$\end{small}.
(3) is clear. Together, (1), (2), (3) imply that $\tilde\eta_\infty'$ is the pullback under ${\rm pr}_{\C^m}$ of a weakly closed $(1,1)$-form of class $C^{\alpha}_{\rm loc}$ on $\C^m$, growing at worst like $O(r^\alpha)$ by (4). Abusing notation, we denote this form, which is a closed $(1,1)$-current with a global potential of class $C^{2,\alpha}_{\rm loc}(\C^m)$, by $\tilde\eta_\infty'$ as well.

Unfortunately $C^{\beta}_{\rm loc}(\C^m \times Y)$ convergence is too weak to conclude from \eqref{zumteufel161}, \eqref{zumteufel162} that $\tilde\eta_\infty'$ is not constant, or to conclude from \eqref{killmenow} that $\tilde\eta_\infty'$ has constant trace. The issue with both of these points is that \eqref{zumteufel161} or \eqref{killmenow} may be due to base-fiber or fiber-fiber components of $d_t^{-\alpha}\tilde\eta_t'$, which go to zero in $C^\beta_{\rm loc}(\C^m \times Y)$. However, it turns out that the $i\partial\ov{\partial}$-exactness of $d_t^{-\alpha}\tilde\eta_t'$ can be used to enhance \eqref{caz699}, \eqref{whoknows112}, proving that the base-fiber and fiber-fiber components of $d_t^{-\alpha}\tilde\eta_t'$ actually go to zero with respect to the \emph{collapsing} reference metric $\tilde{g}_t$. A precise statement is given in Proposition \ref{bd}, which we defer to the next section because this is also the key to Claim 2 and its proof is long and involved. It is then clear from Proposition \ref{bd} and \eqref{zumteufel161}, \eqref{zumteufel162}, \eqref{killmenow} that $\tilde\eta_\infty'$ is nonconstant, of constant trace.

A contradiction to Liouville's theorem can now easily be derived as at the end of Subcase A.
\hfill $\Box$\medskip\

Claim 1 has been proved by deriving a contradiction in all three subcases A, B, and C, modulo the use of Proposition \ref{bd} below in Subcase C.

\subsubsection{}\label{claim2nonproduct}{\bf Claim 2.} There exist two points $\hat{z}, \hat{z}' \in \C^m$ such that $\hat\omega^\bullet_\infty(\hat{z}) \neq \hat\omega^\bullet_\infty(\hat{z}')$.\medskip\

\noindent \emph{Proof of Claim 2.} Recall that $\hat\omega_\infty^\bullet$ is the \begin{small}$C^\beta_{\rm loc}$\end{small}$(\C^m \times Y)$ limit of the $\hat{J}_t^\natural$-K\"ahler forms $\hat\omega_t^\bullet = \hat\omega_t^\natural + \hat\eta_t$, and is the pullback of a K\"ahler form with \begin{small}$C^{2,\alpha}_{\rm loc}$\end{small} potentials on $\C^m$. Moreover, $\hat\omega_t^\natural$ converges in \begin{small}$C^\infty_{\rm loc}$\end{small}$(\C^m \times Y)$ to the degenerate K\"ahler form $\omega_{\C^m}$ pulled back from $\C^m$. Thus, in order to prove Claim 2, thanks to \eqref{blow223}, Claim 1, and Remark \ref{tiammazzo2}, it suffices to prove that there exists an $\epsilon > 0$ such that for all $t$,
\begin{align} |\hat\eta_t^0(\hat{x}_t) - \mathbf{P}^{\hat{g}_t}_{\hat{x}_t'\hat{x}_t}(\hat{\eta}_t^0(\hat{x}_t'))|_{\hat{g}_{t}(\hat{x}_t)} \geq \epsilon.\label{whatweneed}\end{align}
Here, as in Section \ref{s:abstract} (proof of Proposition \ref{p:schauder}, Case 3), if $\eta$ is any $2$-form on the smooth manifold $\C^m\times Y$, we write $\eta = \eta^0 + \eta^1 + \eta^2$ according to $\Lambda^2(\C^m \times Y) = \Lambda^2\C^m \oplus (\Lambda^1\C^m \otimes \Lambda^1 Y) \oplus \Lambda^2Y$.  Now \eqref{blow22} and Claim 1 tell us that \eqref{whatweneed} is indeed true if we replace $\hat\eta_t^0$ by $\hat\eta_t$. Unfortunately we only know that $\hat\eta_t^1, \hat\eta^2_t$ go to zero with respect to a \emph{noncollapsing} reference metric on $\C^m \times Y$, which is too weak to conclude that their contribution to \eqref{blow22} goes to zero (which would prove \eqref{whatweneed}).

The following proposition exploits the $i\partial\ov{\partial}$-exactness of $\hat\eta_t$ with respect to $\hat{J}_t^\natural$ to resolve this issue, as well as an analogous issue at the end of the proof of Claim 1, Subcase C above. Thus, by proving this proposition we will not only prove  \eqref{whatweneed}, hence Claim 2, but also complete the proof of Claim 1.

\begin{proposition}\label{bd}
Let $B = B_1(0) \subset \C^m$ be the unit ball. Let $(Y, g_Y, J_Y)$ be a compact K\"ahler manifold without boundary. Fix $\alpha \in (0,1)$. Let $J_i$ be a sequence of complex structures on $B \times Y$ such that

{\rm (1)} the projection ${\rm pr}_{\C^m}$ is holomorphic with respect to $J_i$ and $J_{\C^m}$ for all $i$, and

{\rm (2)} $J_i \to J_{\C^m} + J_Y$ in $C^{1,\alpha}(B \times Y)$ as $i \to \infty$.

\noindent Let $\{g_{Y,z,i}\}_{z\in B}$ be a sequence of fiberwise $J_i$-K\"ahler metrics on $B \times Y$ such that

{\rm (3)} $\{g_{Y,z,i}\}_{z\in B} \to \{g_{Y}\}_{z\in B}$ in $C^{1,\alpha}(B \times Y)$ as $i \to \infty$, with uniformly bounded $C^2$ norm.

\noindent For a sequence of positive real numbers $\lambda_i \to 0$ define Riemannian product metrics ${g}_{z,i} = g_{\C^m} + \lambda_i^2 g_{Y,z,i}$ on $\C^m \times Y$. Let $\eta_i$ be a sequence of real $2$-forms on $B\times Y$ such that

{\rm (4)} $\eta_i$ is $i\partial\ov{\partial}$-exact with respect to $J_i$ for all $i$, and

{\rm (5)} there exists a $C$ such that for all $i$,
\begin{align}
\sup_{x \in B \times Y}|\eta_i(x)|_{g_{0,i}(x)} + \sup_{x=(z,y) \in B \times Y} \Biggl( \sup_{x' \in B \times Y} \frac{|\eta_i(x) - \mathbf{P}^{g_{z,i}}_{x'x}(\eta_i(x'))|_{g_{0,i}(x)}}{d^{g_{0,i}}(x,x')^\alpha} \Biggr) \leq C.\end{align}

\noindent Write $\eta_i=\sum_{t=0}^2(\eta_i)^t$ according to the decomposition $\Lambda^2(\C^m \times Y) = \bigoplus_{s+t=2} \Lambda^s\C^m \otimes \Lambda^t Y$. Then
\begin{equation}\label{e:sadness}
|(\eta_i)^1|_{{g}_{0,i}} \to 0\;\,{\text{and}}\;\;|(\eta_i)^2|_{{g}_{0,i}}\to 0
\end{equation}
in $C^{\beta}(B \times Y)$ as $i\to\infty$ for any $\beta < \alpha$.
\end{proposition}

\begin{proof}
For $0\leq t\leq 2$ define $(\hat{\eta}_i)^t=\lambda_i^{-t}(\eta_i)^t$ and $\hat{\eta}_i=\sum_{t=0}^2(\hat{\eta}_i)^t.$ Also abbreviate $\hat{g}_{z,i}=g_{\C^m}+g_{Y,z,i}$,
$g_{i} = g_{\C^m} + \lambda_i^2 g_{Y}$, and $\hat{g}=g_{\C^m}+g_Y$. Then Assumptions (3) and (5) tell us that
\begin{align}
\sup_{x \in B \times Y} |\hat{\eta}_i(x)|_{\hat{g}(x)} + \sup_{x=(z,y) \in B \times Y} \Biggl( \sup_{x' \in B \times Y} \frac{|\hat\eta_i(x)-\mathbf{P}^{\hat{g}_{z,i}}_{x'x}(\hat\eta_i(x'))|_{\hat{g}(x)}}{d^{{g}_{i}}(x,x')^\alpha}\Biggr) \leq C.\label{e:sadness2}
\end{align}
Now assume \eqref{e:sadness} was false. Then there would exist a sequence of indices going to infinity such that at least one of the two sequences of \eqref{e:sadness} is uniformly bounded away from zero in $C^{\beta}(B \times Y)$ along this sequence of indices. Then thanks to \eqref{e:sadness2} we can pass to a further subsequence of the $\hat\eta_i$ converging to some limit $\hat{\eta}_\infty\in C^{\alpha}(B \times Y)$ in the topology of $C^{\beta}(B \times Y)$. (We will pretend that this subsequence is actually the entire sequence.) Also, if $x,x'$ lie on the fiber over $z \in B$, then
$$|\hat\eta_i(x)-\mathbf{P}^{\hat{g}_{z,i}}_{x'x}(\hat\eta_i(x'))|_{\hat{g}(x)}\leq Cd^{{g}_i}(x,x')^\alpha\leq C\lambda_i^\alpha \to 0.$$
Since $g_{Y,z,i}\to g_Y$ in $C^{1,\alpha}(Y)$ by Assumption (3), $\mathbf{P}^{\hat{g}_{z,i}}_{x'x}$ converges to $\mathbf{P}^{\hat{g}}_{x'x}$ by Remark \ref{tiammazzo2}. Thus, in the limit we get that $\hat\eta_\infty$ is $g_Y$-parallel in the $Y$-directions, i.e.,
\begin{equation}\label{e:sadness4}
\nabla^{\hat{g}}_{\mathbf{f}}\hat{\eta}_\infty = 0.
\end{equation}
The key point of the following arguments is that \eqref{e:sadness4} together with Assumption (4) contradicts our standing assumption that for some $\epsilon > 0$ and for all $i$,
\begin{equation}\label{e:sadness5}
\||(\eta_i)^1|_{{g}_{0,i}}\|_{C^{\beta}(B\times Y)} + \||(\eta_i)^2|_{{g}_{0,i}}\|_{C^{\beta}(B \times Y)} \geq \epsilon.
\end{equation}

We will first prove that the $(\eta_i)^2$ term in \eqref{e:sadness5} goes to zero. Assumption (4) implies in particular that $\eta_i$ is $d$-exact. This implies for all $z \in B$ that $(\eta_i)^2|_{\{z\}\times Y}$ integrates to zero against all $g_Y$-parallel $2$-forms on $\{z\}\times Y$. (Note that this is a nontrivial constraint because $(Y,g_Y)$ is K\"ahler.) The same is true for $(\hat\eta_i)^2 = \lambda_i^{-2}(\eta_i)^2$ restricted to $\{z\}\times Y$, and since $(\hat\eta_i)^2$ converges to $(\hat\eta_\infty)^2$ in $C^{\beta}(B \times Y)$, it is true for $(\hat\eta_\infty)^2$ as well. But $(\hat\eta_\infty)^2|_{\{z\}\times Y}$ is itself $g_Y$-parallel by \eqref{e:sadness4}, so $(\hat\eta_\infty)^2|_{\{z\}\times Y} = 0$. Since this holds for all $z \in B$, it follows that $(\hat\eta_\infty)^2 = 0$ on $B \times Y$. Thus,
$$\|(\hat\eta_i)^2\|_{C^{\beta}(B \times Y)} \to 0.$$
The claim now follows by using  Lemma \ref{tiammazzo1} (which is possible because $\hat{g}_{0,i} \to \hat{g}$ in $C^{1,\alpha}(B\times Y)$, with bounded $C^2$ norm) and the fact that $|(\hat\eta_i)^2|_{\hat{g}_{0,i}} = |(\eta_i)^2|_{g_{0,i}}$.

Next, we treat the $(\eta_i)^1$ term in \eqref{e:sadness5}. As in the previous step, our goal is to prove that $(\hat\eta_\infty)^1 = 0$ on $B \times Y$. Because of \eqref{e:sadness4}, this is actually trivial when $b^1(Y) = 0$ because a nonzero parallel $1$-form on $Y$ represents a nonzero class in $H^1(Y)$. Thus, the following steps, which rely on the full strength of Assumption (4), are needed only to treat the case $b^1(Y) \neq 0$. This is reminiscent of \cite[Section 4.2]{HHN}, where a divisor $D$ with holomorphically trivial normal bundle did not necessarily move in a pencil if $b^1(D) \neq 0$, leading to technical difficulties of a similar nature as the ones below.

For $j \in \{1,\ldots, m\}$ let $Z_j$ denote the obvious extension to $\C^m \times Y$ of the $j$-th complex coordinate vector field on $\C^m$. The key step towards proving that $(\hat\eta_\infty)^1 = 0$ on $B\times Y$ is the following Subclaim A. Based on this, Subclaim B below will then prove that $(\hat\eta_\infty)^1 = 0$ on $B\times Y$, as desired. \medskip\

\noindent \emph{Subclaim A}. For all $j \in \{1,\ldots,m\}$ and $z \in B$ we have that
\begin{align}\label{e:kohelet}
(Z_j \,\lrcorner\, (\eta_i)^1)|_{\{z\} \times Y} = \ov{\partial}(Z_j(\vp_i)|_{\{z\}\times Y}) + \epsilon_{z,i}.
\end{align}
Here the $\ov{\partial}$-operator is the one associated with the complex structure induced by $J_i$ on $\{z\} \times Y$ (which makes sense by Assumption (1)), $\vp_i$ is an arbitrary $i\partial\ov{\partial}$-potential for $\eta_i $ with respect to $J_i$ on $B \times Y$, and $\epsilon_{z,i}$ is uniformly $o(\lambda_i^2)$ on $\{z\} \times Y$ (independent of the choice of $\vp_i$). \medskip\

Observe that any two choices of $\vp_i$ differ at most by the pullback of a pluriharmonic function on $B$ under ${\rm pr}_{\C^m}$, so the function $Z_j(\vp_i)|_{\{z\}\times Y}$ in \eqref{e:kohelet} is actually well-defined up to a constant.\medskip\

\noindent \emph{Proof of Subclaim A}. Fix $z \in B$ and $j \in \{1,\ldots,m\}$. Notice that $$(Z_j \,\lrcorner\,(\eta_i)^1)|_{\{z\} \times Y} = (Z_j\,\lrcorner\,\eta_i)|_{\{z\}\times Y}.$$
Thanks to Assumption (4) we have $\eta_i=\ddbar\vp_i$ for some function $\vp_i$ on $B \times Y$, where here and below, all operators $\de,\db$ are with respect to $J_i$. Let $(z^1, \ldots, z^m)$ denote the complex coordinates on $\C^m$. Fix any $y \in Y$ and a $J_Y$-holomorphic chart $(y^1, \ldots, y^n)$ on $Y$ near $y$. Assumptions (1)--(2) and Proposition \ref{p:nn} allow us to find a $J_i$-holomorphic chart of the form $(z^1, \ldots, z^m, \hat{y}^1_{i}, \ldots, \hat{y}^n_{i})$ in a definite neighborhood of $(z,y)$ in $B \times Y$ that converges to $(z^1, \ldots, z^m, y^1, \ldots, y^n)$ in $C^{2,\alpha}$ as $i \to \infty$. The key property of a ``fibered'' chart of this type is that if $y^p_i$ denotes the restriction of $\hat{y}^p_i$ to $\{z\} \times Y$, then
\begin{align}\label{e:qoelet}
\frac{\partial \vp}{\partial\hat{y}^p_i}\biggr|_{\{z\}\times Y} = \frac{\partial (\vp|_{\{z\}\times Y})}{\partial y^p_i}
\end{align}
for all local functions $\vp$.  Thus, expanding $\eta_i = i\partial\ov{\partial}\vp_i$ in terms of this chart,
\begin{align}
(Z_j \,\lrcorner\, \eta_i )|_{\{z\}\times Y}&= i \sum_{p=1}^n \frac{\partial}{\partial \ov{y}^p_{i}}\biggl(\frac{\partial\vp_i}{\partial z^j}\biggr|_{\{z\}\times Y}\biggr) d\ov{y}^p_i \label{e:kohelet2} \\&+i \sum_{p,q=1}^n \biggl(\biggl[\frac{\partial^2(\vp_i|_{\{z\}\times Y})}{\partial y^q_{i} \partial \ov{y}^p_{i} } Z_j(\hat{y}^q_{i}) \biggr]d\ov{y}^p_{i} - \biggl[\frac{\partial^2(\vp_i|_{\{z\}\times Y})}{\partial y_i^p \partial\ov{y}_i^q} Z_j(\ov{\hat{y}}^q_i)\biggr] dy^{p}_i\biggr). \label{e:kohelet3}
\end{align}
This is straightforward to check. Observe that in the product case ($J_i = J_{\C^m} + J_Y$) we may choose $\hat{y}^p_i$ equal to the trivial extension of $y^p$ from $Y$ to $B \times Y$; then the terms in  \eqref{e:kohelet3} vanish and in addition $\partial/\partial z^j = Z_j$, so that the right-hand side of \eqref{e:kohelet2} is globally $\ov{\partial}$-exact, proving \eqref{e:kohelet} with $\epsilon_{z,i} = 0$. Also note that in general there is no reason for $Z_j$ to be of type $(1,0)$ (let alone holomorphic) with respect to $J_i$, which is directly related to the presence of the $dy^p_i$ terms in \eqref{e:kohelet3}.

In order to control the errors of \eqref{e:kohelet3}, observe that $\hat{y}^q_{i} \to y^q$ in $C^{2,\alpha}$ on some definite neighborhood of $(z,y)$, so that
$Z_j(\hat{y}^q_{i}), Z_j(\ov{\hat{y}}^q_i) \to 0$ in $C^{1,\alpha}$ on this neighborhood. Moreover,
$$\biggl|\frac{\partial^2(\vp_i|_{\{z\}\times Y})}{\partial y^q_{i} \partial \ov{y}^p_{i}}\biggr| = \biggl|(\eta_i)^2\biggl(\frac{\partial}{\partial y^q_i},\frac{\partial}{\partial \ov{y}^p_i}\biggr)\biggr| \leq C \lambda_i^2$$
by \eqref{e:sadness2} and because the coordinate vector fields attached to the chart $(y^1_i, \ldots, y^n_i)$ converge in $C^{1,\alpha}$. This proves that the errors of \eqref{e:kohelet3} are uniformly $o(\lambda_i^2)$, as desired.

Making the right-hand side of \eqref{e:kohelet2} globally $\ov{\partial}$-exact up to small errors is slightly more subtle. By Assumption (1), the $(0,1)$-part of $Z_j$ with respect to $J_i$ is a section of $T^*Y \otimes \C$, and \eqref{e:qoelet} says that $T^*Y \otimes \C$ is generated by the vector fields $\partial/\partial \hat{y}^p$ and their complex conjugates. Thus, if we expand $Z_j$ in terms of our chart, there will be no $\partial/\partial\ov{z}^k$ components. Since $dz^k(Z_j) = \delta^k_j$, it follows that
\begin{align}\label{e:kohelet5}
Z_j = \frac{\partial}{\partial z^j} - \sum_{p=1}^n \biggl(a_{j,i}^p \frac{\partial}{\partial \hat{y}^p_{i}} + b_{j,i}^p\frac{\partial}{\partial\ov{\hat{y}}^p_i}\biggr)
\end{align}
for some smooth local functions $a_{j,i}^p(z,\hat{y}_i), b_{j,i}^p(z,\hat{y}_i)$ that go to zero in $C^{1,\alpha}$. Thus,
\begin{align}\label{e:kohelet6}
\frac{\partial}{\partial \ov{y}^p_{i}}\biggl(\frac{\partial\vp_i}{\partial z^j}\biggr|_{\{z\}\times Y}\biggr) - \frac{\partial}{\partial \ov{y}^p_{i}}(Z_j(\vp_i)|_{\{z\}\times Y})
= \sum_{q=1}^n \biggl(\frac{\partial a^q_{j,i}}{\partial \ov{\hat{y}}^p_{i}} \frac{\partial(\vp_i|_{\{z\}\times Y})}{\partial y^q_{i}}
& + \frac{\partial b^q_{j,i}}{\partial \ov{\hat{y}}^p_{i}} \frac{\partial(\vp_i|_{\{z\}\times Y})}{\partial \ov{y}^q_{i}} \\
\label{e:kohelet7}
+\,a^q_{j,i}\frac{\partial^2(\vp_i|_{\{z\}\times Y})}{\partial \ov{y}^p_{i}\partial y^q_{i}}
&+ b^q_{j,i}\frac{\partial^2(\vp_i|_{\{z\}\times Y})}{\partial \ov{y}^p_{i}\partial \ov{y}^q_{i}}\biggr).
\end{align}
The coefficient functions involving $a_{j,i}^q, b_{j,i}^q$ on the right-hand side of \eqref{e:kohelet6}--\eqref{e:kohelet7} all go to zero in $C^{\alpha}$.
Moreover, by standard elliptic estimates applied to the differential inequality
$$\|i\partial\ov{\partial}(\vp_i|_{\{z\}\times Y})\|_{C^{\alpha}} \leq C\lambda_i^2,$$
which follows from \eqref{e:sadness2} and Lemma \ref{tiammazzo1}, the derivatives of $\vp_i|_{\{z\}\times Y}$ featuring on the right-hand side of \eqref{e:kohelet6}--\eqref{e:kohelet7} are all uniformly $O(\lambda_i^2)$. (Note that for this it is crucial that \eqref{e:kohelet5} contains no $\partial/\partial\ov{z}^k$ components.) In conclusion, \eqref{e:kohelet} has been proved with $\epsilon_{z,i}$ uniformly $o(\lambda_i^2)$ on $\{z\} \times Y$. \hfill $\Box$\medskip\

As discussed before Subclaim A, the following suffices to complete the proof of Proposition \ref{bd}.\medskip\

\noindent \emph{Subclaim B.} We have $(\hat\eta_\infty)^1 = 0$ on $B \times Y$. \medskip\

\noindent \emph{Proof of Subclaim B.} Fix $z \in B$ and $j \in \{1,\ldots, m\}$. Integrate \eqref{e:kohelet} against any harmonic $(0,1)$-form $\zeta_i$ of $L^2$-norm $1$ with respect to the $J_i$-K\"ahler metric $g_{Y,z,i}$ on $\{z\} \times Y$. The globally $\ov{\partial}$-exact part goes away and the remainder is $o(\lambda_i^2)$. Multiply by $\lambda_i^{-1}$ in order to pass from $(\eta_i)^1$ to $(\hat\eta_i)^1$. It follows that the $L^2$-inner product between $(\hat\eta_i)^1$ and $\zeta_i$ goes to zero. Since $g_{Y,z,i} \to g_{Y}$ in $C^{1,\alpha}(Y)$, and since the Hodge number $h^{0,1}(Y,J_i|_{Y})$ is equal to $h^{0,1}(Y,J_Y)$ for all $i$ sufficiently large (thanks to our assumption that these manifolds are K\"ahler), it is a standard fact from Hodge theory with parameters (cf. the proof of \cite[Thm 5]{KS}) that every $g_Y$-harmonic $J_Y$-$(0,1)$-form $\zeta$ of $L^2$-norm $1$ can be written as the $C^{2,\alpha}(Y)$ limit of a sequence $\zeta_i$ of $g_{Y,z,i}$-harmonic $J_i$-$(0,1)$-forms $\zeta_i$ as above. Thus, passing to a limit, we learn that $(Z_j \,\lrcorner\, (\hat\eta_\infty)^1)|_{\{z\} \times Y}$, which is $g_Y$-parallel thanks to \eqref{e:sadness4}, hence $g_Y$-harmonic, is $L^2$-orthogonal to $\zeta$. This leaves the possibility that $(Z_j \,\lrcorner\, (\hat\eta_\infty)^1)|_{\{z\} \times Y}$ is nonzero of $J_Y$-type $(1,0)$; but this possibility is easily ruled out by integrating the $J_i$-$(1,0)$-part of $(Z_j \,\lrcorner\, (\hat\eta_\infty)^1)|_{\{z\} \times Y}$ against \eqref{e:kohelet} with respect to $g_{Y,z,i}$, multiplying by $\lambda_i^{-1}$, and passing to the limit.

In conclusion, $(Z_j \,\lrcorner\,(\hat{\eta}_\infty)^1)|_{\{z\}\times Y} = 0$ for all $z \in B$ and $j \in \{1,\ldots, m\}$. Since $(\hat\eta_\infty)^1$ is a real $2$-form, the complex $1$-forms $Z_j \,\lrcorner\,(\hat\eta_\infty)^1$ and $\ov{Z}_j \,\lrcorner\,(\hat\eta_\infty)^1$ are complex conjugates of each other. Moreover, these $1$-forms are sections of $T^*Y \otimes \C$, hence are determined by their restrictions to $\{z\} \times Y$ for all $z \in B$. It follows that $Z_j \,\lrcorner\,(\hat\eta_\infty)^1 = \ov{Z}_j \,\lrcorner\,(\hat\eta_\infty)^1 = 0$ on $B \times Y$, and hence that $(\hat\eta_\infty)^1 = 0$.\hfill $\Box$\medskip\

Proposition \ref{bd} has been proved.
\end{proof}

This completes the proof of Claim 2.

\subsubsection{}\label{claim3nonproduct} {\bf Claim 3.} The $C^\alpha$ K\"ahler current $\hat\omega_\infty^\bullet$ on $\mathbb{C}^m$ is parallel with respect to the Euclidean metric.\medskip\

\noindent \emph{Proof of Claim 3}. It suffices to prove that $(\hat\omega_\infty^\bullet)^m = c\omega_{\C^m}^m$ for some constant $c > 0$. Indeed, since $\hat\omega_\infty^\bullet$ has potentials of class $C^{2,\alpha}_{\rm loc}$, it follows from this by a standard elliptic bootstrap that $\hat\omega_\infty^\bullet$ is smooth. Since $\hat\omega_\infty^\bullet$ is uniformly comparable to $\omega_{\C^m}$, Claim 3 then follows from Theorem \ref{lio1}.

The proof of the fact that $(\hat\omega_\infty^\bullet)^m$ is standard is very similar to the proof of the corresponding fact in Section \ref{claim3product}. Recall that on $B_{\lambda_t}\times Y$ we have
\begin{align}
\hat\omega_{t}^\bullet =\omega_{\mathbb{C}^m}+\delta_t^2 \Psi_t^*\omega_F+\ddbar\vp_t,\label{marushka}\\
\label{maresca2}
(\hat{\omega}_t^\bullet)^{m+n}=c_t\delta_t^{2n} e^{\hat{G}_t}[\lambda_t^{2m}\Psi_t^*(\omega_{\infty}^m\wedge \omega_F^n)],
\end{align}
where $\hat{G}_t=G\circ\Psi_t$, and where here and in the rest of this claim the operators $\de,\db$ are understood to be with respect to $\hat{J}^\natural_t$. Also recall that as $t\to\infty$ we have that
\begin{align*}
\hat{J}_t^\natural \to J_{\C^m} + J_{Y,z_\infty},\\
\Psi_t^*\omega_F\to  \omega_{Y,z_\infty},\\
[\lambda_t^{2m}\Psi_t^*(\omega_{\infty}^m\wedge \omega_F^n)]\to \omega_{\C^m}^m \wedge \omega_{Y,z_\infty}^n,
\end{align*} all of these locally smoothly on $\C^m \times Y$.
The goal is to show that with $c_\infty=\lim_{t\to\infty}c_t > 0$,
\begin{equation}\label{wp2}
{m+n\choose m}(\hat\omega_\infty^\bullet)^m=c_\infty e^{G(z_\infty)}\omega_{\mathbb{C}^m}^m.
\end{equation}

First of all, we can write $\hat\omega_\infty^\bullet =\omega_{\mathbb{C}^m}+i\partial\ov{\partial}\vp$ for some $\vp \in C$\begin{small}$^{2,\alpha}_{\rm loc}$\end{small}$(\C^m)$. We also write $\vp$ for the pullback of $\vp$ to $\mathbb{C}^m\times Y$, so that in particular $\ddbar\vp_t\to\ddbar\vp$ in $C$\begin{small}$^{\alpha}_{\rm loc}$\end{small}. As in \cite{To}, let $\underline{\vp_t}$ denote the function on $\C^m$ (as well as its pullback to $\C^m \times Y$) obtained as the fiber average of $\vp_t$ with respect to $\Psi_t^*\omega_F^n$. Fix a smooth function $\eta$ on $\mathbb{C}^m$ with compact support $K$ and also write $\eta$ for its pullback to $\mathbb{C}^m\times Y$. Fix $t$ large enough so  that $K\subset B_{\lambda_t}$. From the Monge-Amp\`ere equation \eqref{marushka}--\eqref{maresca2} we have
\begin{equation*}\begin{split}
c_t\int_{\mathbb{C}^m\times Y}\eta e^{\hat{G}_t}&[\lambda_t^{2m}\Psi_t^*(\omega_{\infty}^m\wedge \omega_F^n)]=\frac{1}{\delta_t^{2n}}\int_{\mathbb{C}^m\times Y}\eta (\omega_{\mathbb{C}^m}+\delta_t^2\Psi_t^*\omega_F+\ddbar\vp_t)^{m+n}\\
&=
\frac{1}{\delta_t^{2n}}\int_{\mathbb{C}^m\times Y}\eta ((\omega_{\mathbb{C}^m}+\mn\de\db\underline{\vp_t})+(\delta_t^2\Psi_t^*\omega_F+\mn\de\db(\varphi_t-\underline{\vp_t})))^{m+n}\\
&=\frac{1}{\delta_t^{2n}}\int_{\mathbb{C}^m\times Y}\eta \sum_{j=0}^{m+n}\binom{m+n}{j}(\omega_{\mathbb{C}^m}+\mn\de\db\underline{\vp_t})^j\wedge(\delta_t^2\Psi_t^*\omega_F+\mn\de\db(\varphi_t-\underline{\vp_t}))^{m+n-j}.
\end{split}
\end{equation*}
Observe that $\omega_{\mathbb{C}^m}+\mn\de\db\underline{\vp_t}$ is pulled back from $\mathbb{C}^m$, hence can be wedged with itself at most $m$ times, so all terms in the sum with $j>m$ are zero. Next, we claim that all the terms with $j<m$ go to zero as $t\to \infty$. To see this, start by observing that any such term can be expanded into
$$\frac{1}{\delta_t^{2n}}\binom{m+n}{j}\sum_{i=0}^{m+n-j}\binom{m+n-j}{i}\int_{\mathbb{C}^m\times Y} \eta(\omega_{\mathbb{C}^m}+\mn\de\db\underline{\vp_t})^j\wedge(\delta_t^2\Psi_t^*\omega_F)^{m+n-j-i}\wedge(\mn\de\db(\varphi_t-\underline{\vp_t}))^{i}.$$
The term with $i=0$ is easily seen to go to $0$ because the integrand is $O(\delta_t^{2(m+n-j)})$ and $j<m$, while for each term with $i>0$ we can rewrite the integral as
\begin{equation}\label{e:dontfuckup2}
\int_{\mathbb{C}^m\times Y} (\vp_t-\underline{\vp_t})\ddbar\eta\wedge(\omega_{\mathbb{C}^m}+\mn\de\db\underline{\vp_t})^j\wedge(\delta_t^2\Psi_t^*\omega_F)^{m+n-j-i}\wedge(\mn\de\db(\varphi_t-\underline{\vp_t}))^{i-1}.
\end{equation}
Work in local $\hat{J}_t^\natural$-holomorphic product coordinates on the total space. The form $\ddbar\eta\wedge(\omega_{\mathbb{C}^m}+\mn\de\db\vp)^j$ is pulled back from $\mathbb{C}^m$, so in the coordinate representation of $(\delta_t^2\Psi_t^*\omega_F)^{m+n-j-i}\wedge(\mn\de\db(\varphi_t-\underline{\vp_t}))^{i-1}$, each summand is a wedge product of $2(m-j-1)$ basis $1$-forms pulled back from $\mathbb{C}^m$ and $2n$ basis $1$-forms pulled back from $Y$. Multiplied together, the fiber contributions are $O(\delta_t^{2n})$. Indeed, the fiber- fiber components of $\delta_t^2\Psi_t^*\omega_F$ have an explicit factor of $\delta_t^2$, and its base-fiber components have a factor of $\delta_t^2$ as well where $\delta_t$ would be enough. Moreover, \eqref{caz2} implies that for all $z \in K$,
\begin{equation}\label{pazzoide2}
|(\mn\de\db(\varphi_t-\underline{\vp_t}))|_{\{z\}\times Y}| = |(\mn\de\db\varphi_t)|_{\{z\}\times Y}|\leq C\delta_t^2,
\end{equation}
so the fiber-fiber components of $\ddbar (\vp_t-\underline{\vp_t})$ are again $O(\delta_t^2)$, and its base-fiber components are $O(\delta_t)$ thanks to the Cauchy-Schwarz inequality with respect to $\hat{g}_t^\bullet$. Thus, every term
$$(\omega_{\mathbb{C}^m}+\mn\de\db\underline{\vp_t})^j\wedge(\delta_t^2\Psi_t^*\omega_F)^{m+n-j-i}\wedge(\mn\de\db(\varphi_t-\underline{\vp_t}))^{i-1}$$
in \eqref{e:dontfuckup2} is $O(\delta_t^{2n})$. To see that the whole integral in \eqref{e:dontfuckup2} is $o(\delta_t^{2n})$, as desired, it then suffices to note that $\sup_{K\times Y}|\vp_t-\underline{\vp_t}|\leq C\delta_t^2$, which follows from \eqref{pazzoide2} by inverting the Laplacian on each fiber.

We are then left with only the term with $j=m$, which is
\begin{align*}&\frac{1}{\delta_t^{2n}}\binom{m+n}{m}\int_{\mathbb{C}^m\times Y}\eta(\omega_{\mathbb{C}^m}+\mn\de\db\underline{\vp_t})^m\wedge(\delta_t^2\Psi_t^*\omega_F+\mn\de\db(\varphi_t-\underline{\vp_t}))^{n}\\
=\,&\frac{1}{\delta_t^{2n}}\binom{m+n}{m}\int_{\mathbb{C}^m\times Y}\eta(\omega_{\mathbb{C}^m}+\mn\de\db\underline{\vp_t})^m\wedge (\delta_t^2\Psi_t^*\omega_F)^{n} +\frac{1}{\delta_t^{2n}}\int_{\mathbb{C}^m\times Y} \mn\de\db\eta\wedge(\omega_{\mathbb{C}^m}+\mn\de\db\underline{\vp_t})^m\wedge \Upsilon.
\end{align*}
The second term is zero because $\mn\de\db\eta$ is pulled back from $\C^m$. Altogether, we obtain that
\begin{equation}\label{e:to6}
c_t\int_{\mathbb{C}^m\times Y}\eta e^{\hat{G}_t}[\lambda_t^{2m}\Psi_t^*(\omega_{\infty}^m\wedge \omega_F^n)]= \binom{m+n}{m}\int_{\mathbb{C}^m\times Y}\eta(\omega_{\mathbb{C}^m}+\mn\de\db\underline{\vp_t})^m\wedge (\Psi_t^*\omega_F)^{n}.
\end{equation}
Now observe that because $\Psi_t^*\omega_F^n$ is closed and of $\hat{J}_t^\natural$-type $(n,n)$, we have the identity
$$\ddbar(\underline{\vp_t}-\vp)=({\rm pr}_{\C^m})_*\left(\ddbar(\vp_t-\vp) \wedge \frac{\Psi_t^*\omega_F^n}{\int_Y \Psi_t^*\omega_F^n}\right).$$
Because $\ddbar\vp_t \to \ddbar\vp$ in $C^\alpha_{\rm loc}$, it follows that
$\ddbar\underline{\vp_t}\to\ddbar\vp$, and hence $\omega_{\C^m} + \ddbar\underline{\vp_t} \to \hat\omega_\infty^\bullet$, in $C^{\alpha}_{\rm loc}$.
Letting $t \to \infty$ in \eqref{e:to6} and integrating out the $Y$ factor yields the weak form of \eqref{wp2}, as desired.

This completes the proof of Claim 3, hence of Case 3 and of Theorem \ref{t:fornow}.
\end{proof}

\section{The case of compact Calabi-Yau manifolds}\label{compac}

Let us first derive Corollary \ref{main} from Theorem \ref{mainlocal}.

\begin{proof}[Proof of Corollary \ref{main}]
Let $f:X\to B$ be as in Corollary \ref{main}. By \cite{FG}, $f$ is a holomorphic fiber bundle over $B\setminus f(S)$. Fix any small coordinate ball in $B$ over which this holomorphic fiber bundle is trivial. Replacing $B$ with this ball and $f$ with the product map, we may assume that $B$ is a ball in $\mathbb{C}^m$ and $f:B\times Y\to B$ is the projection, with $Y=X_b$ a compact Calabi-Yau manifold. In order to be able to apply Theorem \ref{mainlocal} we first need to apply a gauge transformation. By \cite[Prop 3.1]{He2} (cf. \cite[Prop 3.1]{GTZ}, \cite[Lemma 4.1]{GW}, \cite[Claim 1, p.382]{He},  \cite[p.2936--2937]{TZ}, and the proof of Proposition \ref{p:closure}), we can find a biholomorphism $T$ of $B\times Y$ (over $B$) such that $T^*\omega_X=\omega_Y+\ddbar u_1$ for some smooth real function $u_1$. Note that \cite[Prop 3.1]{He2} is stated with $B=\mathbb{C}^m$, but the proof applies also if $B$ is a ball in $\mathbb{C}^m$. Let us also note for later purposes that $T$ takes the form $T(z,y) =  (z, y + \sigma(z))$, where $\sigma$ is a holomorphic function from $B$ to the space of $g_Y$-parallel $(1,0)$-vector fields on $Y$, and where the addition $y+\sigma(z)$ has the same meaning as in \cite[(1.1)]{He2}. Fix a smooth real function $u_2$ on $B$ such that $\omega_\infty=\omega_{\mathbb{C}^m}+\ddbar u_2$. Then setting
$$\ti{\psi}_t=\psi_t\circ T+e^{-t}u_1+u_2,$$
we have that
$$\omega_{\mathbb{C}^m}+e^{-t}\omega_Y+\ddbar\ti{\psi}_t=\omega_\infty+e^{-t}T^*\omega_X+T^*\ddbar\psi_t=T^*\omega^\bullet_t,$$
and
$$(\omega_{\mathbb{C}^m}+e^{-t}\omega_Y+\ddbar\ti{\psi}_t)^{m+n}=c_t e^{-nt}T^*\omega_X^{m+n}=c_t e^{-nt+F}\omega_{\mathbb{C}^m}^m\wedge\omega_Y^n,$$
where we define
$$e^F=\frac{T^*\omega_X^{m+n}}{\omega_{\mathbb{C}^m}^m\wedge\omega_Y^n},$$
so that $F$ is the pullback of a pluriharmonic function on $B$. The constants $c_t$ approach a positive limit as $t\to \infty$, so up to a global rescaling we may assume that the metrics $T^*\omega^\bullet_t$ precisely satisfy \eqref{mageneral}.
In \cite{To} (cf. \cite[Lemma 4.1]{GTZ}) it is proved that there is a constant $C$ such that on $B\times Y$ we have
\begin{equation}\label{itsthefinalfuckup0}
C^{-1}(\omega_\infty+e^{-t}\omega_Y)\leq \omega^\bullet_t\leq C(\omega_\infty+e^{-t}\omega_Y).
\end{equation}
This clearly implies the bound \eqref{unifequiv} for $T^*\omega^\bullet_t$, so applying Theorem \ref{mainlocal} we deduce that
\begin{equation}\label{itsthefinalfuckup}
\|T^*\omega_t^\bullet\|_{C^k(K \times Y, \omega_t)} \leq C_{K,k}
\end{equation}
for all compact sets $K \subset B$, with $\omega_t = \omega_{\C^m} + e^{-t}\omega_Y$. We will now show that the $\omega_t$-norm in \eqref{itsthefinalfuckup} can be replaced by the $T^*\omega_t$-norm, which then clearly implies \eqref{estimate2} and \eqref{estimate3}, proving Corollary \ref{main}.

For $k = 1$, since $S_t = \nabla^{T^*g_t} - \nabla^{g_t}$ is a tensor and thanks to \eqref{itsthefinalfuckup0}, it suffices to prove that $|S_t|_{g_t} \leq C$ on $B \times Y$ for some constant $C$ which is independent of $t$. By multiplying all metrics by $e^t$ and pulling back by the diffeomorphism $(z,y) \mapsto (e^{-t/2}z,y)$, this is seen to be equivalent to proving that
$$|\nabla^{T_t^*g_P} - \nabla^{g_P}|_{g_P} \leq Ce^{-\frac{t}{2}}\;\,{\rm on}\;\,B_{e^{\frac{t}{2}}} \times Y,$$
where $g_P = g_{\C^m} + g_Y$ and $T_t(z,y) = (z, y + \sigma(e^{-t/2}z))$. But in fact we have an even stronger estimate (with $e^{-t}$ rather than $e^{-t/2}$ on the right) because $|\nabla^{g_P}(T_t^*g_P)|_{g_P} \leq C e^{-t}$ on $B_{e^{t/2}}\times Y$. The case $k > 1$ can easily be treated by induction. The essential point is that
$T_t^*g_P$ improves by a factor of $e^{-t/2}$ upon $g_P$-covariant differentiation because $\sigma$ takes values in the $g_Y$-parallel vector fields on $Y$.
\end{proof}

Lastly, we derive Corollary \ref{main2} from Theorem \ref{thm51}.

\begin{proof}[Proof of Corollary \ref{main2}]
The proof is analogous to the one of Corollary \ref{main}. Recall that we have fixed a K\"ahler metric $\omega_B$ on $B$ (in the sense of analytic spaces), to solve \eqref{mageneral2}, and on $X\setminus f^{-1}(f(S))$ we have constructed in Section \ref{nonpr} a smooth function $\rho$ such that the $(1,1)$-form $\omega_F=\omega_X+\ddbar\rho$ restricts to a Ricci-flat metric on all regular fibers $X_z$, and such that $\omega_F^n\wedge\omega_{\infty}^m$ is a strictly positive volume form on $X\setminus f^{-1}(f(S))$. Fix any coordinate ball in $B$ which is compactly supported in $B\setminus f(S)$, and replace $B$ with this ball, so now $f:X\to B$ is as in the setting of Theorem \ref{thm51}. We can write $\omega_B=\omega_{\C^m}+\ddbar u$ for some smooth function $u$ on $B$. Thus, if we define
\begin{align*}
\ti{\psi}_t=\psi_t-e^{-t}\rho+u,\\
\omega_\infty=f^*\omega_{\C^m},\\
e^G=\frac{\omega_X^{m+n}}{\omega_\infty^m\wedge\omega_F^n}
\end{align*}
(so that $G$ is in fact the pullback of a function from the base, cf.~\cite[p.445]{To}), then
\begin{align*}
\omega^\bullet_t=\omega_\infty+e^{-t}\omega_F+\ddbar\ti{\psi}_t,\\
(\omega^\bullet_t)^{m+n}=c_te^{-nt}\omega_X^{m+n}=c_t e^{-nt+G}\omega_\infty^m\wedge\omega_F^n.
\end{align*}
As recalled above, it is proved in \cite{To} (cf.~\cite[Lemma 4.1]{GTZ}) that
$$
C^{-1}(\omega_\infty+e^{-t}\omega_X)\leq \omega^\bullet_t\leq C(\omega_\infty+e^{-t}\omega_X).
$$
Up to increasing the uniform constant $C$, this implies
$$
C^{-1}(\omega_\infty+e^{-t}\omega_F)\leq \omega^\bullet_t\leq C(\omega_\infty+e^{-t}\omega_F),
$$
which is \eqref{unifequivx}.
We are now in the setting of Theorem \ref{thm51}, so we obtain \eqref{estimatex}, which as explained in Remark \ref{important} implies that \eqref{estimate4} holds on $f^{-1}(B)$. This completes the proof of Corollary \ref{main2}.
\end{proof}

\begin{rk}
If one is only interested in the setting of Corollaries \ref{main} and \ref{main2} of a global fiber space with total space a compact Calabi-Yau manifold (as opposed to the local settings of Theorems \ref{mainlocal} and \ref{thm51}), then the proofs above can be modified to avoid using the more advanced Liouville Theorem \ref{lioHein} of \cite{He2,LLZ}, replacing it instead with the easier Liouville Theorem \ref{lio} together with the main result of \cite{TWY}. Indeed, the result of \cite{He2,LLZ} was only used in Sections \ref{sekt} and \ref{nonpr} in Case 2. If we are in the setting of Corollary \ref{main} (say in Case 2 of Section \ref{sekt}), we can then appeal to \cite[(1.10)]{TWY} to see that the restrictions $\hat{\omega}^\bullet_\infty|_{\{z\}\times Y}$ ($z\in\mathbb{C}^m$) are in fact all equal to $\omega_Y$, while the argument of \cite[p.2936--2937]{TZ} (cf. the proof of Proposition \ref{p:closure}) shows that $\hat{\omega}^\bullet_\infty$ is $\ddbar$-cohomologous to $\omega_{\mathbb{C}^m}+\omega_Y$ on $\mathbb{C}^m\times Y$. The easy Liouville Theorem \ref{lio} then shows that $\hat{\omega}^\bullet_\infty$ is the product of $\omega_Y$ and a constant metric on $\mathbb{C}^m$, which contradicts \eqref{uno}. The same argument works in Case 2 in Section \ref{nonpr} for Corollary \ref{main2}.
\end{rk}

\end{document}